%% file: HigherIndex.tex
\documentclass{amsart}
\usepackage{amsmath,amssymb,latexsym}
\usepackage{mathpazo}
\usepackage[mathscr]{eucal}
\usepackage[all]{xy}
\usepackage{pb-diagram}

\newcommand{\sigmapr}{\sigma_\text{\tiny\rm pr}}
\newcommand{\sigmares}{\sigma_\text{\tiny\rm res}}
\newcommand{\deltaAS}{\overline{\delta}}
\newcommand{\lambdaAS}{\overline{\lambda}}
\newcommand{\sAS}{\overline{s}}
\newcommand{\BAS}{B_\text{\tiny \rm AS}}
\newcommand{\calA}{{\mathcal{A}}}
\newcommand{\calAS}{{\mathcal{AS}}}
\newcommand{\calB}{{\mathcal{B}}}
\newcommand{\calW}{{\mathcal{W}}}
\newcommand{\calV}{{\mathcal{V}}}
\newcommand{\calC}{\mathcal{C}}

\newcommand{\calN}{\mathcal{N}}
\newcommand{\calO}{\mathcal{O}}
\newcommand{\twcalO}{\mathcal{O}_\text{\tiny\rm tw}}
\newcommand{\calE}{\mathcal{E}}
\newcommand{\calJ}{\mathcal{J}}
\newcommand{\calS}{\mathcal{S}}
\newcommand{\calU}{\mathcal{U}}

\newcommand{\scrA}{\mathscr{A}}
\newcommand{\scrC}{\mathscr{C}}

\newcommand{\scrQ}{\mathscr{Q}}

\newcommand{\scrX}{\mathscr{X}}
\newcommand{\normscrX}{\overline{\mathscr{X}}}
\newcommand{\scrCAS}{\mathscr{C}_\text{\tiny \rm AS}}
\newcommand{\scraCAS}{\mathscr{C}_\text{\tiny \rm aAS}}
\newcommand{\scrcCAS}{\mathscr{C}_\text{\tiny \rm $\lambda$AS}}
\newcommand{\CAS}{C_\text{\tiny \rm AS}}
\newcommand{\HAS}{H_\text{\tiny \rm AS}}
\newcommand{\sfG}{\mathsf{G}}

\newcommand{\sfQ}{\mathsf{Q}}

\newcommand{\sfX}{\mathsf{X}}

\newcommand{\frM}{\mathfrak{M}}
\newcommand{\A}{\mathbb{A}}
\newcommand{\C}{\mathbb{C}}
\newcommand{\R}{\mathbb{R}}
\newcommand{\W}{\mathbb{W}}
\newcommand{\N}{\mathbb{N}}
\newcommand{\Z}{\mathbb{Z}}
\newcommand{\Hom}{\operatorname{Hom}}
\newcommand{\End}{\operatorname{End}}
\newcommand{\Aut}{\operatorname{Aut}}
\newcommand{\g}{\mathfrak{g}}
\newcommand{\h}{\mathfrak{h}}
\newcommand{\gl}{\mathfrak{gl}}
\newcommand{\sh}{\operatorname{Sh}}
\newcommand{\sgn}{\operatorname{sgn}}
\newcommand{\ind}{\operatorname{ind}}
\newcommand{\Tr}{\operatorname{Tr}}
\newcommand{\tr}{\operatorname{tr}}

\newcommand{\Tot}{\operatorname{Tot}}
\newcommand{\pr}{\operatorname{pr}}
\newcommand{\id}{\operatorname{id}}
\newcommand{\spin}{\mathfrak{sp}}
\newcommand{\ev}{\operatorname{ev}}
\newcommand{\Ch}{\operatorname{Ch}}
\newcommand{\Chloc}{\operatorname{Ch}^\text{\rm\tiny loc}}

\newcommand{\normC}{\overline{C}}
\newcommand{\normB}{\overline{B}}
\newcommand{\normcalB}{\overline{\mathcal B}}
\newcommand{\normscrC}{\overline{\mathscr C}}
\newcommand{\Weyl}{\mathbb W^\text{\rm \tiny poly}}
\newcommand{\VWeyl}{\mathbb W^{\text{\rm \tiny poly},V}}
\newcommand{\fWeyl}{\mathbb W}
\newcommand{\pWeyl}{\mathbb W^+}
\newcommand{\Sym}{\operatorname{Sym}}
\newcommand{\ASym}{\operatorname{ASym}}
\newcommand{\JSym}{\operatorname{JSym}}
\newcommand{\PDO}{\Psi\!\operatorname{DO}}
\newcommand{\APDO}{\operatorname{A}\!\!\Psi\!\operatorname{DO}}
\newcommand{\Op}{\operatorname{Op}}
\newcommand{\Gammacpt}{\Gamma_\text{\rm\tiny cpt}}
\newcommand{\Exp}{\operatorname{Exp}}
\newcommand{\supp}{\operatorname{supp}}
\newcommand{\cstar}{\star_\text{\tiny\rm c}}

\numberwithin{equation}{section}
\theoremstyle{plain}
        \newtheorem{theorem}{Theorem}[section]
        \newtheorem{lemma}[theorem]{Lemma}
        \newtheorem{proposition}[theorem]{Proposition}
        \newtheorem{corollary}[theorem]{Corollary}

\theoremstyle{definition}
        \newtheorem{definition}[theorem]{Definition}
        \newtheorem{remark}[theorem]{Remark}
        \newtheorem{example}[theorem]{Example}

\title[Cyclic cocycles on deformation quantizations]{Cyclic cocycles on
       deformation quantizations and higher index theorems}
\author{M.J.~Pflaum, H. Posthuma,~\textrm{and} X.~Tang}
\begin{document}
\begin{abstract}
We construct a nontrivial cyclic cocycle on the Weyl algebra of a
symplectic vector space. Using this cyclic cocycle we construct an
explicit, local, quasi-isomorphism from the complex of differential
forms on a symplectic manifold to the complex of cyclic cochains of
any formal deformation quantization thereof. We give a new proof of Nest-Tsygan's algebraic
higher index theorem by computing the pairing between such cyclic
cocycles and the $K$-theory of the formal deformation quantization.
Furthermore, we extend this approach to derive an algebraic higher
index theorem on a symplectic orbifold. As an application, we obtain
the analytic higher index theorem of Connes--Moscovici and its
extension to orbifolds.
\end{abstract}
\address{\newline
Markus J. Pflaum, {\tt markus.pflaum@colorado.edu}\newline
         \indent {\rm Department of Mathematics, University of Colorado,
         Boulder, USA}\newline
        Hessel Posthuma, {\tt posthuma@math.uu.nl}\newline
         \indent {\rm Mathematical Institute, Utrecht University, Utrecht,
         The Netherlands} \newline
        Xiang Tang, {\tt xtang@math.wustl.edu}   \newline
         \indent {\rm  Department of Mathematics, Washington University,
         St.~Louis, USA}}
\maketitle

\tableofcontents
\input{Intro}
\input{CyclicWeyl}
\input{IndThms}
\input{GenOrbifolds}
\input{HAnaInThm}
\input{AnaHighIndOrb-new}
\appendix
\input{AppLocCycHom}
%\input{sec7}
%
%
\bibliographystyle{alpha}

\end{document}

%% file: Intro.tex
\section{Introduction}
Let $D$ be an elliptic differential operator on a compact manifold
$M$. As is well-known ellipticity implies that $D$ is a Fredholm
operator and the Atiyah-Singer index theorem \cite{AtiSin} expresses
the index of $D$ as a topological formula involving the Chern
character of the symbol $\sigma(D)$ and the Todd class of the
manifold $M$. In \cite{conmos}, {\sc Connes--Moscovici} proved a far
reaching generalization of the Atiyah--Singer index theorem, the
so-called higher index theorem. In subsequent work \cite{moswu},
{\sc Moscovici--Wu} provided an abstract setting to construct higher
indices. The essential idea hereby is as follows.

Let $\PDO^{-\infty}(M)$ be the algebra of smoothing
pseudodifferential operators on $M$. The operator $D$ defines an
element $e_D$ in the $K_0$-group of the algebra of smoothing
pseudodifferential operators $\PDO^{-\infty}(M)$,
and its image under the Chern-Connes character an element $\Ch(e_D)$
in the cyclic homology of $\PDO^{-\infty}(M)$. Since smoothing
operators act by trace class operators, the operator trace gives
rise to a cyclic cocycle $\tr$ on $\PDO^{-\infty}(M)$ of degree $0$.
Pairing this cocycle with the cycle $\Ch(e_D)$ one recovers the
analytic index of $D$ as $\ind (D)=\langle \tr ,\Ch(e_D)\rangle$.
As has been explained in \cite[\S 2]{conmos}, the local information
contained in $D$ respectively its symbol $\sigma (D)$ is not fully
captured by this index pairing.
To remedy this, {\sc Connes--Moscovici} constructed a localized index
which in the literature and also in this work is called the higher index.
According to \cite{moswu}, one can understand the higher index of $D$
as a pairing $\ind_{[f]}(D)= \langle [f] , \Chloc (e_D) \rangle$,
where $f$ is a given Alexander-Spanier cocycle on $M$
(on which one localizes the index), and $\Chloc (e_D)$ is an Alexander-Spanier
homology class associated $e_D$ which here is regarded as a difference of
projections in $\PDO^{-\infty}(M)$.
The higher index theorem in \cite{conmos} compute the localized index - which
no longer is integral - in terms of topological data generalizing the
Atiyah--Singer index theorem.

In this paper, we prove an algebraic generalization of the higher
index theorem to symplectic manifolds. Applying our theorem
to cotangent bundles, we recover the theorem of
{\sc Connes--Moscovici}. Furthermore, we extend our theorem to general
symplectic orbifolds and obtain an analog of the higher index
theorem on orbifolds generalizing {\sc Kawasaki}'s orbifold index theorem
\cite{Kaw} and also {\sc Marcolli--Mathai}'s higher index theorem for good
orbifolds \cite{MarMat}.\\

Our approach to an algebraic higher index theorem for symplectic
manifolds is inspired by the work \cite{ffs}. There, {\sc
Feigin--Felder--Shoikhet} proved an algebraic index theorem for
symplectic manifolds based on a formula for a Hochschild
$2n$-cocycle $\tau_{2n}$ on the Weyl algebra $\W_{2n}$ over
$\R^{2n}$ with its canonical symplectic structure. In this paper, we
construct an extension of the Hochschild cocycle $\tau_{2n}$ to a
sequence of cochains $(\tau_0,\tau_2,\ldots,\tau_{2n})$ which forms
a cocycle in the total cyclic bicomplex $\left( \Tot^{2n} \normcalB C^\bullet
(\Weyl_{2n}),b+B\right)$. Using this $(b+B)$-cocycle and Fedosov's
construction of a deformation quantization
$(\calA^{((\hbar))}_{\text{\tiny \rm cpt}}, \star)$ on a symplectic
manifold, we construct a quasi-isomorphism $\sfQ$ from the cyclic
de Rham complex to the $b+B$ total complex of
$(\calA^{((\hbar))}_{\text{\tiny \rm cpt}}, \star)$,
\[
\sfQ: \big( \Tot^\bullet \calB \Omega^\bullet (M)((\hbar)) , d \big)
\to
  \big( \Tot^\bullet \calB \normC^\bullet
  \big(\calA^{((\hbar))}_\text{\tiny \rm cpt}\big), b+B \big).
\]
If one views $(\calA^{((\hbar))}_{\text{\tiny \rm cpt}}, \star)$ as
the generalization of the algebra of  pseudodifferential
operators, one can try to compute the pairing between a cyclic
cocycle on $\calA^{((\hbar))}_{\text{\tiny \rm cpt}}$ with the
Chern-Connes character of an element in
$K_0(\calA^{((\hbar))}_{\text{\tiny \rm cpt}})$. Fedosov proved in
\cite{fe} that $K_0(\calA^{((\hbar))}_{\text{\tiny \rm cpt}})$ can
be represented by pairs of projectors $(P_1, P_2)$ on
$\calA^{((\hbar))}_{\text{\tiny \rm cpt}}$ with $P_1-P_2$ compactly
supported modulo stabilization.  Using methods from Lie algebra
cohomology, we obtain the following formula for this pairing,
\[
  \left\langle \sfQ(\alpha), P_1-P_2\right\rangle=
  \sum_{l=0}^{k}\frac{1}{(2\pi \sqrt{-1})^{l}}\int_M \alpha_{2l}
  \wedge \hat{A}(M)\Ch(V_1-V_2) \exp \Big( -\frac{\Omega}{2\pi \sqrt{-1}\hbar}\Big),
\]
where $\alpha=(\alpha_0, \cdots, \alpha_{2k})\in \Tot^{2k} \calB
\Omega^\bullet (M)((\hbar))$ is a sequence of closed differential
forms on $M$, and $V_1$ and $V_2$ are vector bundles on $M$
determined by the $0$-th order terms of $P_1$ and $P_2$, and
$\Omega\in \omega+\hbar H^2(M)[[\hbar]]$ is the characteristic class
of the deformation quantization $(\calA^{((\hbar))}_{\text{\tiny \rm
cpt}}, \star)$.

That the right hand side of the above algebraic higher index formula
coincides with  the algebraic localized index was originally proved
by {\sc Nest-Tsygan} in \cite{nets96}, \cite{bnt}, and {\sc Aastrup}
in \cite{aas} by a different approach. Using \v{C}ech-methods, {\sc
Nest-Tsygan} computed the Chern-Connes character of an element in
$K_0(\calA^{((\hbar))})$ by constructing a morphism from the cyclic
homology of $\calA^{((\hbar))}$ to the cohomology of $M$. Our
construction is exactly in the opposite direction and lifted to the
(co)chain level: By means of the formula for the $(b+B)$-cocycles
$(\tau_0, \ldots, \tau_{2n})$ we are able to construct an explicit
quasi-isomorphism $\sfQ$ from the sheaf complex  of differential
forms to the sheaf complex of cyclic cochains of
$\calA^{((\hbar))}$. This allows us to write down explicit
expressions for cyclic cocycles on $\calA^{((\hbar))}$. With this
new construction, we give a more transparent proof of the above
index theorem using differential forms and Lie algebra cohomology,
which is closer to {\sc Connes--Moscovici}'s original approach.

Let us  mention that the $b+B$ cycle $(\tau_0, \cdots,\tau_{2n})$ has been
discovered independently by {\sc Willwacher}
\cite{Wil}. He used this cocycle to compute a higher Riemann-Roch
formula.\\

By a similar idea as above, we extend the algebraic index theorem of
\cite{ppt} for orbifolds to the above higher version.
We represent an orbifold by a proper \'etale groupoid, and
consider $\calA^{((\hbar))}\rtimes \sfG$ as a deformation quantization
of a symplectic orbifold $M=(\sfG_0/\sfG, \omega)$, as it has been
constructed by the third author in \cite{ta}. Using Fedosov's idea
\cite{fe:g-index}, we generalize the above $b+B$ cocycle $(\tau_0,
\cdots, \tau_{2n})$ on the Weyl algebra $\W_{2n}$ to a
$\gamma$-twisted $b+B$ cocycle with $\gamma$ a linear symplectic
isomorphism on $V$ of finite order. Analogously to the manifold case,
we use the $\gamma$-twisted cocycle and Fedosov's connection to
define a S-quasi-isomorphism $Q$ from the cyclic de Rham
differential complex on the corresponding inertia orbifold
$\widetilde{M}$ to the $b+B$ total complex of the algebra
$\calA^{((\hbar))}\rtimes \sfG$. For $\alpha=(\alpha_{2k}, \cdots,
\alpha_0)\in \Tot^{2k}\calB \Omega^\bullet(\widetilde{M}) ((\hbar))$,
and $P_1, P_2$ two projectors in the matrix algebra over
$\calA^{((\hbar))}\rtimes \sfG$ with $P_1-P_2$ compactly supported,
we obtain the following formula as Thm.~\ref{thm:orb-alg-index}.
\[
  \left\langle Q(\alpha), P_1-P_2\right\rangle=
  \sum_{j=0}^{k}\int_{\widetilde{M}} \frac{1}{(2\pi \sqrt{-1})^jm}\frac{\alpha_{2j} \wedge
  \hat{A}(\widetilde{M}) \, \Ch_\theta(\iota^*V_1-\iota^*V_2) \,
  \exp(-\frac{\iota^*\Omega}{2\pi \sqrt{-1}\hbar})}{\Ch_\theta(\lambda_{-1}N)},
\]
where $V_1$ and $V_2$ are the orbifold vector bundles on $M$
determined by the $0$-th order terms of $P_1$ and $P_2$,
$\Omega$ is the characteristic class of $(\calA^{((\hbar))}\rtimes
\sfG, \star)$,  $\iota$ is the canonical map from $\widetilde{M}$ to
$M$, and $m$ is defined in terms of the order of the local isotopy groups.\\

As an application of our algebraic formulas, we derive higher analytic
index theorems for elliptic operators using an asymptotic symbol calculus.

To this end we first consider a cotangent bundle of a manifold $Q$. It was
shown by the first author \cite{P1} that the asymptotic symbol
calculus on pseudodifferential operators on $Q$ naturally defines a
deformation quantization
$(\calA^{((\hbar))}_{\text{\tiny \rm cpt}},\star_{\text{op}})$ of $T^*Q$ and the
operator trace induces a canonical trace on
$(\calA^{((\hbar))}_{\text{\tiny \rm cpt}},\star_{\text{op}})$. To derive
{\sc Connes--Moscovici}'s higher index from the higher algebraic index
theorem, we prove that the algebraic pairing
$\left\langle Q(\alpha), P_1-P_2\right\rangle$
coincides asymptotically with the pairing $\langle \scrX {[f]} ,
\Ch(e_D) \rangle$ defined in \cite{conmos}. More precisely, we prove
that the cyclic cocycles $Q(\alpha)$ and $\scrX {[f]}$ on
$\calA^{((\hbar))}_{\text{\tiny \rm cpt}}(T^*Q)$ are cohomologous, if
the Alexander--Spanier cocycle $f$ and the closed form $\alpha$
induce the same cohomology class on $Q$. We prove the claimed
relation by using sheaf theoretic methods and by applying inherent
properties of the calculus of asymptotic pseudodifferential
operators. Let us mention that a sketch of how to derive the
analytic higher index theorem from the algebraic one has already
been outlined in \cite{nets96}. Here, we take a  different approach
by elaborating more on the nature of Alexander--Spanier cohomology
and its use for constructing cyclic cocycles on a deformation
quantization in general. In particular, this enables us to directly
compare the algebraic higher index with the definition of the
localized index by {\sc Connes--Moscovici}.

Secondly, we consider the cotangent bundle of an orbifold $Q$. The
way we address this problem is similar to the above manifold case.
To define a higher index for an elliptic operator $D$ on $Q$, we
need to define a localized index of an elliptic operator $D$ on $Q$.
This leads us to introduce a new notion of orbifold
Alexander-Spanier cohomology, whose cohomology is equal to the
cohomology of the corresponding inertia orbifold $\widetilde{Q}$.
Next, we introduce a notion of localized $K$-theory of an orbifold,
and show that there is a well defined pairing between localized
$K$-theory and orbifold Alexander-Spanier cohomology of $Q$. With
these natural definitions and constructions, we follow the same
ideas as in the manifold case to prove a higher index theorem on a
reduced orbifold. We would like to remark that our definition of
orbifold Alexander--Spanier cohomology is new and different from the
standard definition of Alexander-Spanier cohomology of a topological
space. In particular, we have view an orbifold as a stack more than
just a topological space. For this reason, our higher index theorem
on orbifolds
detects the topological information of an orbifold as a stack. \\

Recall that {\sc Connes--Moscovici} \cite{conmos} used their higher
index theorem to prove a covering index theorem, which was used to
prove the Novikov conjecture in the case of hyperbolic groups. We
would like to view this paper as a seed for the study of covering
index theorems (cf.~\cite{MarMat}) for orbifolds and the equivariant
Novikov conjecture
\cite{roswei}. We plan to study these questions in the future.\\

This paper is organized as follows. In Section \ref{Sec:CycWeyl}, we
introduce and prove that $(\tau_0, \cdots, \tau_{2n})$ defines a $b+B$
cocycle on the Weyl algebra $\W_{2n}$. In Section \ref{Sec:CyCoSyMan},
we use a Fedosov connection to construct a quasi-isomorphism from the
sheaf complex of differential forms to the sheaf complex of cyclic cochains on the
algebra of the deformation quantization $(\calA^{((\hbar))}_{\text{\tiny
\rm cpt}}, \star)$ of a symplectic manifold corresponding to the Fedosov connection.
Then, in Section \ref{Sec:AlgIndThm}, we use Lie
algebra Chern-Weil theory technique to prove a algebraic higher
index theorem. Afterwards, in Section \ref{Sec:GenOrbi}, we extend the
constructions from Sections \ref{Sec:CycWeyl}-\ref{Sec:AlgIndThm} to
orbifolds and obtain a higher algebraic index theorem for orbifolds. In
Sections \ref{sec:HAITM} and \ref{sec:HAITO}, we discuss how to apply the
higher algebraic index theorem to prove {\sc Connes} and {\sc Moscovici}'s
higher index theorem on manifolds and its generalization to orbifolds.\\

\noindent{\bf {Acknowledgments}:} The authors would like to thank M.
Crainic, G. Felder, A. Gorokhovsky, and H.~Moscovici for helpful
discussions. H. Posthuma and X. Tang would like to thank the
Department of Mathematics, University of Colorado for hosting their
visits at CU Boulder, where part of the work has been completed. The
research of H.~Posthuma is supported by NWO, and X.~Tang is
partially supported by NSF grant 0703775.

%%% Local Variables:
%%% mode: latex
%%% TeX-master: "HigherIndex"
%%% End:

%% file: CyclicWeyl.tex
%
%
\section{Cyclic cohomology of the Weyl algebra}
\label{Sec:CycWeyl}
\subsection{The Weyl algebra}
\label{weyl_algebra}
Let $(V,\omega)$ be a finite dimensional symplectic vector space. In
canonical coordinates $(p_1,\ldots p_n,q_1,\ldots q_n)$ the
symplectic form simply reads $\omega=\sum_idp_i\wedge dq_i$. The
polynomial Weyl algebra $\Weyl (V)$ over the ring $\C[\hbar,
\hbar^{-1}]$ is the space of polynomials $\operatorname{S}
(V^*)\otimes\C[\hbar,\hbar^{-1}]$ with algebra structure given by
the Moyal--Weyl product
\[
\begin{split}
f\star g&=(m\circ\exp(\frac{\hbar}{2}\alpha))(f\otimes g)
\end{split}
\]
where $m$ is the commutative multiplication and
$\alpha\in \End \big( \Weyl (V) \otimes \Weyl (V) \big)$ is  basically the Poisson
bracket associated to $\omega$:
\[
  \alpha(f\otimes g)=\sum_{i=1}^n \left(\frac{\partial f}{\partial p_i}
  \otimes\frac{\partial g}{\partial q_i}-
  \frac{\partial f}{\partial q_i}\otimes\frac{\partial g}{\partial p_i}\right).
\]
In the formula for the Moyal product, the exponential is defined by means of its
power series expansion, which terminates after finitely many terms for the
polynomial Weyl algebra. For the particular case where $V =\R^{2n}$ with its
natural symplectic structure, we write $\Weyl_{2n}$ for $\Weyl (V)$.

The symplectic group $\operatorname{Sp} _{2n}$ acts on
$\Weyl_{2n}$ by automorphisms. Infinitesimally, this leads to an action of
the Lie algebra $\mathfrak{sp}_{2n}$ by derivations. It is known that all
derivations of $\Weyl_{2n}$ are inner, in fact there is a
short exact sequence of Lie algebras
\[
  0\rightarrow\C[\hbar,\hbar^{-1}]\rightarrow \Weyl_{2n}\rightarrow
  \operatorname{Der}\big( \Weyl_{2n} \big) \rightarrow 0.
\]
The action of $\mathfrak{sp}_{2n}$ is explicitly given by identifying
$\mathfrak{sp}_{2n}$ with the quadratic homogeneous polynomials in
$\operatorname{S} (V^*)$.

Finally, using the spectral sequence associated to the $\hbar$-filtration on $\Weyl_{2n}$,
%using the identification of $\Weyl_{2n}$ as a deformation of
%$\operatorname{S}(V^*)$,
one proves the following well-known result:
\begin{proposition}\cite{ft}
\label{cyclic-weyl}
The cyclic cohomology of the Weyl algebra is given by
\[
  HC^k (\Weyl_{2n})=
  \begin{cases}
    \C[\hbar,\hbar^{-1}]& \text{if $k=2n+2p$}, p\geq 0, \\
    0& \text{else}.
  \end{cases}
\]
\end{proposition}
\subsection{Cyclic cocycles on the Weyl algebra}
The aim of this section is to define an explicit cocycle in the
$(b,B)$-complex that generates the nontrivial cyclic cohomology
class at degree $2n$  as is suggested in Proposition \ref{cyclic-weyl} above. We first need a
couple of definitions. For $1\leq i\neq j\leq 2k\leq 2n$ we define
$\alpha_{ij}\in \End \big( (\Weyl_{2n})^{\otimes 2k+1}\big)$ by
\begin{displaymath}
\begin{split}
  \alpha_{ij}(a_0\otimes\ldots\otimes a_{2k})&=
  \sum_{s=1}^n \left(a_0 \otimes \ldots\otimes
  \frac{\partial a_i}{\partial p_s}\otimes\ldots\otimes
  \frac{\partial a_j}{\partial q_s}\otimes\ldots\otimes a_{2k}\right.\\
  &\hspace{3cm}
  \left.-a_0\ldots\otimes \frac{\partial a_i}{\partial q_s}\otimes\ldots
  \otimes\frac{\partial a_j}{\partial p_s}\otimes\ldots\otimes a_{2k}\right),
\end{split}
\end{displaymath}
i.e., the Poisson tensor acting on $i$'th and $j$'th slot of the tensor
product. We also need
\[
 \pi_{2k}=1\otimes (\hbar\alpha)^{\wedge k}\in
 \End \big( (\Weyl_{2n})^{\otimes (2k+1)} \big),
\]
and finally
$\mu_i: ( \Weyl_{2n})^{\otimes (i+1)}\rightarrow\C[\hbar,\hbar^{-1}]$
is given by
\[
  \mu_i (a_0\otimes\ldots\otimes a_i )=a_0(0)\cdots a_i (0).
\]
In the following, $\Delta^k\subset \R^k$ is the standard simplex given by $0\leq u_1\leq\ldots\leq u_k\leq 1$.
\begin{definition}
\label{dfn:tau}
  Let $\Weyl_{2n}$ be the Weyl algebra. For all $i$ with $0\leq i\leq 2n$ define
  the cochains $\tau_i\in \normC^i(\Weyl_{2n})$ as follows. For even degrees put
\[
\tau_{2k}(a)=(-1)^k\mu_{2k}\int_{\Delta^{2k}}\left.\prod_{0\leq
i<j\leq
2k}e^{\hbar(u_i-u_j+\frac{1}{2})\alpha_{ij}}\right|_{u_0=0}\pi_{2k}(a)
du_1\cdots du_{2k}.
\]
In the odd case, we put
\[
\tau_{2k-1}(a):=(-1)^{k-1}\mu_{2k-1}\int_{\Delta^{2k-1}}\left.\prod_{0\leq
i<j\leq
2k-1}e^{\hbar(u_i-u_j+\frac{1}{2})\alpha_{ij}}\right|_{u_0=0}(\hbar\alpha)^{\wedge
k}(a)du_1\cdots du_{2k-1},
\]
\end{definition}
\begin{remark}\label{rmk:sign}
The cocycle $\tau_{2n}\in \normC^{2n}( \Weyl_{2n})$ is
the Hochschild cocycle of \cite{ffs} up to a sign $(-1)^n$. The sign is needed for $(\tau_0, \cdots, \tau_{2n})$ to be $b+B$ closed, as in the theorem below.
\end{remark}
\begin{theorem}
\label{cw}
The cochains $\tau_i\in \normC^i(\Weyl_{2n})$ satisfy the relation
\[
  -B\tau_{2k}=\tau_{2k-1}=b\tau_{2k-2}.
\]
\end{theorem}
\begin{remark}
For $n=1$, the proof of this theorem is quite easy since the cocycles can be written down explicitly. We have
\[
\begin{split}
\tau_2(a_0\otimes a_1\otimes
a_2)&:=-\hbar\mu_2\int_{\Delta^2}e^{\hbar(\frac{1}{2}-u_1)\alpha_{01}}e^{\hbar(\frac{1}{2}-u_2)\alpha_{02}}
e^{\hbar(u_1-u_2+\frac{1}{2})\alpha_{12}}\\
&\hspace{5cm} (1\otimes \alpha)\left(a_0\otimes a_1\otimes a_2\right)du_1du_2\\
\tau_0(a_0)&:=a_0(0)
\end{split}
\]
With this we compute:
\[
\begin{split}
B\tau_2(a_0\otimes a_1)&=-\tau_2(1\otimes a_0\otimes a_1)+\tau_2(1\otimes a_1\otimes a_0)\\
&=-\hbar\mu_2\int_{\Delta^2}e^{\hbar(u_1-u_2+\frac{1}{2})\alpha}\alpha(a_0\otimes a_1-a_1\otimes a_0)du_1\wedge du_2\\
&=-\hbar\mu_2\int_0^1 du_2\int_0^{u_2}du_1e^{\hbar(u_1-u_2+\frac{1}{2})\alpha}\alpha(a_0\otimes a_1-a_1\otimes a_0)\\
&=-\mu_2\int_0^1 du_2\int_0^{u_2}du_1\left( \frac{d}{du_1}e^{\hbar(u_1-u_2+\frac{1}{2})\alpha} \right)(a_0\otimes a_1-a_1\otimes a_0)\\
&=-\mu_2\int_0^1 du_2 \left(e^{\frac{\hbar}{2}\alpha}-e^{\hbar(\frac{1}{2}-u_2)\alpha}\right)(a_0\otimes a_1-a_1\otimes a_0)\\
&=-\mu_2 e^{\frac{\hbar}{2}\alpha}(a_0\otimes a_1-a_1\otimes a_0)\\
&\qquad +\mu_2\int_0^1 du_2 \left( e^{\hbar(\frac{1}{2}-u_2)\alpha}-e^{-\hbar(\frac{1}{2}-u_2)\alpha}\right)(a_0\otimes a_1)\\
&=-\mu_2 e^{\frac{\hbar}{2}\alpha}(a_0\otimes a_1-a_1\otimes a_0)\\
&=-\left(a_0\star a_1(0)-a_1\star a_0(0)\right)\\
&=-b\tau_0(a_0\otimes a_1).
\end{split}
\]
The integral in the sixth line can be seen to be zero by using the antisymmetry of the integrand under reflection in the point $u_2=1/2$. This gives an easy proof of the $n=1$ case.
\end{remark}

\noindent
In the general case, the proof of Theorem \ref{cw} proceeds in two steps:
\begin{lemma}
One has
$\tau_{2k-1}(a_0\otimes\ldots\otimes
 a_{2k-1})=-B\tau_{2k}(a_0\otimes\ldots\otimes a_{2k-1})$.
\end{lemma}
\begin{proof} First we write out the left hand side:
\[
\begin{split}
\tau_{2k-1}& (a_0\otimes\ldots\otimes a_{2k-1})= \\
&= (-1)^{k-1}\int_{\Delta^{2k-1}}\prod_{0\leq i<j\leq 2k-1}e^{\hbar(u_i-u_j+\frac{1}{2})\alpha_{ij}}\\
&\hspace{3cm}(\hbar\alpha)^{\wedge k}(a_0\otimes\ldots\otimes a_{2k-1})du_1\cdots du_{2k-1}\\
&=(-1)^{k-1}\int_0^1 ds\int_{\Delta^{2k-1}}\prod_{0\leq i<j\leq 2k-1}e^{\hbar(u_i-u_j+\frac{1}{2})\alpha_{ij}}\\
&\hspace{3cm}(\hbar\alpha)^{\wedge k}(a_0\otimes\ldots\otimes a_{2k-1})du_1\cdots du_{2k-1}\\
&=(-1)^{k-1}\sum_{l=0}^{2k-1}\int_{u_l}^{u_{l+1}}ds\int_{\Delta^{2k-1}}\prod_{0\leq i<j\leq 2k-1}e^{\hbar(u_i-u_j+\frac{1}{2})\alpha_{ij}}\\
&\hspace{3cm}(\hbar\alpha)^{\wedge k}(a_0\otimes\ldots\otimes a_{2k-1})du_1\cdots du_{2k-1}\\
&=(-1)^{k-1}\sum_{l=0}^{2k-1}(-1)^l\int_{\Delta^{2k}}\prod_{0\leq i<j\leq 2k-1}e^{\hbar(u_i-u_j+\frac{1}{2})\alpha_{ij}}\\
&\hspace{3cm}(\hbar\alpha)^{\wedge k}(a_0\otimes\ldots\otimes
a_{2k-1})du_1\cdots du_ldsdu_{l+1}\cdots du_{2k-1}.
%&=\sum_{l=0}^{2k-1}(-1)^i\int_{\Delta^{2k}}dv_1\ldots dv_{2k}\prod_{0\leq i<j\leq 2k }e^{t(2v_i-2v_j+1)\alpha_{ij}}\pi^l_{2k}\\& (a_0\otimes\ldots a_{i-1}\otimes 1\otimes \ldots\otimes a_{2k-1}),
\end{split}
\]
In the $l$-th term of the sum we now change variables
\[
\begin{split}
v_1&=s\\
v_2&=u_{l+1}+s\\
&\,\,\,\vdots \\
v_{2k-l} &=u_{2k-1}+s\\
v_{2k-l+1}&=u_1+s\\
&\,\,\,\vdots \\
v_{2k}&=u_{l}+s.
\end{split}
\]
Now let $\sigma_l\in S_{2k}$ be the cyclic permutation
$\sigma_l(1,\ldots 2k)=(l,\ldots, 2k,1,\ldots l-1)$. With this, the
$l$-th term can be written as
\[
  (-1)^l \int_{\sigma_l(\Delta^{2k})}\prod_{1\leq i<j\leq 2k}
  e^{\hbar\psi(v_{\sigma_l(i)}-v_{\sigma_l(j)})\alpha_{ij}}(\hbar\alpha)^{\wedge k}
  (a_0\otimes\ldots\otimes a_{2k-1})dv_1\cdots dv_{2k},
\]
where $\psi:\R\rightarrow [-1,1]$ is the function introduced in \cite[\S 2.4]{ffs}.
As in the proof of Lemma 2.2.~of loc.~cit., the expression above is equal to
\[
  (-1)^l\int_{\Delta^{2k}}\prod_{1\leq i<j\leq 2k}
  e^{\hbar(v_{i}-v_{j}+\frac{1}{2})\alpha_{ij}}(\hbar\alpha)^{\wedge k}
  (a_l\otimes\ldots\otimes a_{2k-1}\otimes a_0\otimes\ldots a_{l-1})dv_1\cdots dv_{2k}.
\]
On the other hand we have
\[
\begin{split}
B\tau_{2k}(a_0\otimes\ldots\otimes
a_{2k-1})&=\sum_{l=0}^{2k-1}(-1)^l\tau_{2k}(1\otimes a_l\otimes
\ldots\otimes a_{2k-1}\otimes a_0\otimes \ldots\otimes a_{l-1})\\
&=(-1)^k\sum_{l=0}^{2k-1}(-1)^l\mu_{2k-1}\int_{\Delta^{2k}}\prod_{1\leq j<l\leq 2k}e^{\hbar(u_i-u_j+\frac{1}{2})\alpha_{ij}}\\
&\hspace{0.5cm}\times (\hbar\alpha)^{\wedge k}(a_l\otimes
\ldots\otimes a_{2k-1}\otimes a_0\otimes \ldots\otimes
a_{l-1})du_1\ldots du_{2k}.
\end{split}
\]
One finally concludes that the two sides of the claimed equality coincide.
\end{proof}

\begin{lemma}
One has $b\tau_{2k}=\tau_{2k+1}$.
\end{lemma}
\begin{proof}
The proof of the claim proceeds along the lines of the proof of Proposition 2.1.~in
\cite{ffs}. Introduce the differential form
$\eta\in\Omega^{2k}(\Delta^{2k+1},\normC^{2k+1}(\fWeyl_{2n}))$ by
\[
  \eta:=(-1)^k\mu_{2k+1}\sum_{i=1}^{2k+1}\left.\prod_{0\leq j<l\leq 2k+1}
  e^{\hbar(u_j-u_l+\frac{1}{2})\alpha_{jl}}\right|_{u_0=0}
  du_1\wedge\ldots\wedge \widehat{du_i}\wedge\ldots\wedge
  du_{2k+1}\pi^i_{2k},
\]
where $\pi^i_{2k}\in \End \big( \fWeyl_{2n}^{\otimes 2k+2}\big)$ is
$(\hbar\alpha)^{\wedge k}$ acting on all slots in the tensor product
except the zero-th and
 the $i$-th. It then follows that
\[
\begin{split}
  b\tau_{2k}&=\int_{\partial\Delta^{2k+1}}\eta =\int_{\Delta^{2k+1}}d\eta.
\end{split}
\]
We have
\[
 d\eta=(-1)^k\mu_{2k+1}\sum_{i=1}^{2k+1}(-1)^i\sum_{s=0}^{2k+1}\hbar\alpha_{is}\pi^i_{2k}
 \prod_{0\leq j<l\leq 2k+1}e^{\hbar(u_j-u_l+\frac{1}{2})\alpha_{jl}}du_1\wedge\ldots\wedge du_{2k+1}.
\]
We now claim that
\[
  \sum_{i=1}^{2k+1}(-1)^i\sum_{s=0}^{2k+1}\hbar\alpha_{is}\pi^i_{2k}=(\hbar\alpha)^{\wedge (k+1)}
  \in \operatorname{End}\big( \fWeyl_{2n}^{\otimes(2k+2)} \big).
\]
Indeed one can split the sum as
\[
  \sum_{i=1}^{2k+1}(-1)^i\sum_{s=0}^{2k+1}\hbar\alpha_{is}\pi^i_{2k}=
  \sum_{i=1}^{2k+1}(-1)^i\hbar\alpha_{i0}\pi^i_{2k}+
  \sum_{i=1}^{2k+1}(-1)^i\hbar\sum_{s=1}^{2k+1}\alpha_{is}\pi^i_{2k}.
\]
The first part equals $(\hbar\alpha)^{\wedge(k+1)}$, whereas the
second equals zero: the $\alpha_{ij}$ all commute among each other,
the number of terms $2k(2k-1)$ is even, they cancel pairwise.
\end{proof}
\noindent
This completes the proof of Theorem \ref{cw}. As a corollary we have of course:
\begin{corollary}
\label{cor:cyclic-tau}
The cochains $\tau_{2k}$, $0\leq k\leq n$
combine to define a cocycle
\[
 (\tau_0,\tau_2,\ldots,\tau_{2n})\in
 \left( \Tot^{2n} \calB \overline{C}^\bullet (\Weyl_{2n}),b+B\right),
\]
\end{corollary}
\begin{remark}
In particular, $b\tau_{2n}=0$, which is the statement in \cite{ffs} that $\tau_{2n}$
is a Hochschild cocycle, generating the Hochschild cohomology in degree $2n$.
In other words, we have completed this Hochschild cocycle $\tau_{2n}$ to a full
cyclic cocycle
$(\tau_0,\tau_2,\ldots,\tau_{2n})$ in the $(b,B)$-complex. Notice the similarity of
this cocycle with the so-called JLO-cocycle \cite{jlo}.
\end{remark}
\subsection{The $\mathfrak{sp}_{2n}$-action}
For any algebra $A$, denote by $\mathfrak{gl}(A)$ the associated Lie algebra given by
$A$ equipped with the Lie bracket $[a_1,a_2]=a_1a_2-a_2a_1$. This Lie algebra acts on
the Hochschild chains by
\[
  L_a(a_0\otimes\ldots\otimes a_k)=
  \sum_{i=0}^k(a_0\otimes\ldots\otimes [a,a_i]\otimes\ldots\otimes a_k).
\]
The Cartan formula $L_a=b\circ \iota_a+\iota_a\circ b$ holds with respect to the
Hochschild differential, if we define $\iota_a:C_k(A)\to C_{k+1}(A)$ by
\[
 \iota_a(a_0\otimes\ldots\otimes a_k) =
 \sum_{i=0}^k(-1)^{i+1}(a_0\otimes \ldots \otimes a_{i}\otimes a\otimes
 a_{i+1}\otimes\ldots\otimes a_k).
\]
Dually, these formulas induce Lie algebra actions of $\mathfrak{gl}(A)$ on
$C^\bullet(A)$ and $\normC^\bullet(A)$.

Recall that $\mathfrak{sp}_{2n}$ acts on $\Weyl_{2n}$ by inner derivations where we
identify $\mathfrak{sp}_{2n}$ with the homogeneous quadratic polynomials in
$\Weyl_{2n}$.
\begin{proposition}
\label{prop:relative-tau}
The cochains $\tau_{2k}\in \normC^{2k}(\Weyl_{2n})$, $0\leq k\leq n$ are invariant and
basic with respect to $\mathfrak{sp}_{2n}$, i.e.,
\[
  L_a\tau_{2k}=0 \quad \text{and} \quad \iota_a\tau_{2k}=0 \quad
  \text{for all $a\in \mathfrak{sp}_{2n}$}.
\]
\end{proposition}
\begin{proof}
 The proof is literally the same as for the Hochschild cocycle $\tau_{2n}$,
 cf.~\cite[Thm 2.2.]{ffs}.
\end{proof}
\noindent This property of the cocycle $(\tau_0,\ldots,\tau_{2n})\in
\Tot^{2n} \calB \overline{C}^\bullet (\Weyl_{2n})$ is important in
the next section where we apply the Fedosov construction to
globalize these cocycles to deformed algebras over arbitrary
symplectic manifolds.
\section{Cyclic cocycles on symplectic manifolds}
\label{Sec:CyCoSyMan}
Let $(M,\omega)$ be a symplectic manifold with symplectic form $\omega$. We study in
this section the cyclic cohomology of a deformation quantization
$\calA^\hbar$ of $(M, \omega)$. In particular, we construct an
explicit chain map from the space of differential forms on $M$ to
the space of cyclic cochains on the quantum algebra $\calA^\hbar$.
\subsection{Deformation quantization of symplectic manifolds}
For the convenience of the reader let us briefly review Fedosov's construction of a
deformation quantization of a symplectic manifold $(M, \omega)$.

We first extend the Weyl algebra $\Weyl (V)$ for a symplectic vector space
$(V, \omega)$ to $\pWeyl (V)$ and $\fWeyl (V)$. Let
$y^1, \cdots, y^{2n}$ be a symplectic basis of $V$ with
$y^{2i-1}=p_i$, $y^{2i}=q_i$ for $1\leq i\leq 2n$. Then $\pWeyl (V)$ consists of
elements of the form
\[
  \sum_{i_1, \cdots, i_2n, i\geq 0}\hbar^i a_{i, i_1,\cdots, i_{2n}}
  y^{i_1}\cdots y^{2n}\quad \text{with $a_{i, i_1, \cdots, i_{2n}}$ constant}.
\]
It is easy to check that the product $\star$ on $\fWeyl$ extends to a well
defined associative product on $\pWeyl(V)$. Furthermore, we define
$\fWeyl (V)$ to be $\pWeyl (V)[\hbar^{-1}]$.

Observe that the standard symplectic Lie group
$\operatorname{Sp}(2n, \mathbb{R})$ lifts to act on $\fWeyl (V)$ and $\pWeyl
(V)$. Let $FM$ be the symplectic frame bundle of $TM$, which is a
principal $\operatorname{Sp}(2n, \mathbb{R})$-bundle. We consider the following
associated bundle $\calW=FM\times _{\operatorname{Sp}_{2n}} \pWeyl
{V}$, which is usually called the Weyl algebra bundle. We fix a
symplectic connection $\nabla$ on $TM$, which lifts to a connection
$\tilde{\nabla}$ on $\calW$. Let $R\in \Omega^2\big( M; \mathfrak{sp}(TM)
\big)$ be the curvature of $\nabla$. Then $\tilde{\nabla}^2$ is
equal to $\frac{1}{\hbar}[\tilde{R}, -]\in \Omega^2 \big( M;
\End(\calW) \big)$, where $\tilde{R}$ is
obtained from $R$ via the embedding $\mathfrak{sp}_{2n}\hookrightarrow
\pWeyl_{2n}$.

Assign $deg(y^i)=1$, and $deg(\hbar)=2$, and denote $\fWeyl_{\geq
k}$ be subset of $\fWeyl$ with degree greater than or equal $k$.
Fedosov proved in \cite{fe} that there exists a smooth section
$\tilde{A}\in\Omega^1(M; \calW_{\geq 3})$ such that $D=
\tilde{\nabla}+ \frac{1}{\hbar}[A, -]$ defines a flat connection on
$\calW$, which means that $D^2=0\in \Omega^2 \big( M; \End(\calW)
\big)$. This implies that the Weyl curvature $\Omega$ of $D$, which
is defined by $\Omega=\tilde{R}+\tilde{\nabla}(A)+\frac{1}{2
\hbar}[A,A]$ is in the center of $\calW$ since $D^2=
\frac{1}{\hbar}[\Omega,-]$. Since the center of $\mathbb{W}^+_{2n}$
is given by $\C[[\hbar]]$, $\Omega= -\omega+\hbar\omega_1+\cdots$ is
a closed form in $\Omega^2 \big( M; \C[[\hbar]] \big)$. By \cite{fe}
it follows that the sheaf $\calA^\hbar_D$ of flat sections with
respect to $D$ is isomorphic to $\calC^\infty_M[[\hbar]]$ as a
$\C[[\hbar]]$-module sheaf. Moreover, the induced product on
$\calC^\infty(M)[[\hbar]]$ defines a star product on $M$. The
connection $D$ is usually called a Fedosov connection on $\calW$. In
the following we will refer to $\calA^\hbar_D (M)$ as the  quantum
algebra  associated to $D$, and will often denote it for short by
$\calA^\hbar_D$ or $\calA^\hbar$ if no confusion can arise. The
algebra of sections with compact support of the sheaf
$\calA^\hbar_D$ will be denoted by $\calA^\hbar_\text{\rm\tiny
cpt}$. Finally, let us remark that gauge equivalent $D$ and $D'$
define isomorphic sheaves of algebras $\calA^\hbar_D$ and
$\calA^\hbar_{D'}$, and that any formal deformation quantization of
$M$ can be obtained in this way.
\subsection{Shuffle product on Hochschild chains}
In this part, we review the construction of shuffle product on
Hochschild chains. Let $A$ be a graded algebra with a degree $1$
derivation $\nabla$. Recall that the shuffle product between $a_0
\otimes \cdots \otimes a_p \in \normC_p (A) $ and $b_0 \otimes
\cdots \otimes  b_q \in \normC_q (A)$ is defined to be
\[
\begin{split}
  (a_0 \otimes & \, \cdots \otimes a_p)  \times (b_0 \otimes \cdots \otimes b_q) = \\
  & = (-1)^{\deg(b_0)(\sum_j \deg(a_j))} \sh_{p,q} (a_0b_0 \otimes
  a_1 \otimes \cdots \otimes a_p \otimes b_1 \otimes  \cdots \otimes b_q),
\end{split}
\]
where
\[
  \sh_{p,q}(c_0 \otimes \cdots \otimes c_{p+q})=
  \sum_{\sigma\in \operatorname{S}_{p,q}}\sgn(\sigma) \, c_0 \otimes c_{\sigma (1)}
  \otimes \cdots \otimes c_{\sigma (p+q)}
\]
with sum over all $(p,q)$-shuffles in $\operatorname{S}_{p+q}$.

In \cite[Sec.~2]{ef}, {\sc Engeli--Felder} considered differential
graded algebras, and studied the properties of the shuffle product
of a Hochschild chain with a Maurer-Cartan element in the
differential graded algebra. Due to the needs of our application
here to deformation quantization, we consider a generalized
Maurer-Cartan element $\omega$ which means a degree $1$ element of
$A$ such that $\nabla\omega + \omega^2/\hbar+\tilde{R}=\Omega$ is in
the center of $A$ and $\tilde{R}$ is a degree 2 element. We prove
the following analogous properties of shuffle products with $\omega$
as in \cite{ef}.

\begin{lemma}
\label{lem:shuffle}
  Let $\omega\in A$ be such that $\Omega-\tilde{R}=\nabla\omega+\omega^2/\hbar$. Put
  $(\omega)_k :=1 \otimes \omega \otimes \cdots \otimes \omega\in
  \normC_k(A)$. Then one has for all
  $a=a_0 \otimes \cdots \otimes a_p\in \normC_p(A)$
\begin{equation}
\label{eq:b-shuffle}
\begin{split}
  b(a\times (\omega)_k)= \; &b(a)\times (\omega)_k+(-1)^pa\times
  b(\omega)_k\\
  &-(-1)^p\sum_{i=0}^k (a_0\otimes \cdots \otimes [\omega, a_i] \otimes \cdots
  \otimes a_p)\times (\omega)_{k-1},
\end{split}
\end{equation}
where $[a, a']$ for $a,a'\in A$ is the graded commutator between $a$ and $a'$.
\end{lemma}
\begin{proof}
This is literally the same as the proof of \cite[Lemma 2.6]{ef}.
\end{proof}

\begin{lemma}
\label{lem:b-omega} For $\omega$ as in Lemma \ref{lem:shuffle} and $k\geq 1$
\[
b(\omega)_0=0 \quad \text{ and } \quad b (\omega)_k=\hbar
\nabla((\omega)_{k-1})+\hbar\sum_{j=1}^{k-1}(-1)^j1 \otimes \omega
\otimes \cdots\otimes  (\Omega-\tilde{R}) \otimes \cdots \otimes
\omega .
\]
\end{lemma}
\noindent
Let us remark at this point that Lemma \ref{lem:b-omega} is slightly different
from
\cite[Lemma 2.5]{ef} because of the existence of $\Omega$.
\begin{proof}
  First check $b((\omega)_0)=b(1)=0$. Then observe that for
  $k\geq 1$
  \[
  \begin{split}
  b (\omega)_k  = \, & b(1\otimes \omega \otimes \cdots \otimes \omega)=
  \omega\otimes \cdots \otimes \omega+ \\
  & +\sum_{j=1}^{k-1}(-1)^j 1\otimes \omega \otimes
  \cdots \otimes \omega^2 \otimes \cdots \otimes \omega+
  (-1)^k(-1)^{k-1}\omega \otimes \cdots \otimes \omega\\
  = \, &
  \sum_{j=1}^{k-1}(-1)^j1 \otimes \omega \otimes \cdots \otimes\hbar (\Omega-\tilde{R}-\nabla\omega)
  \otimes \cdots \otimes \omega\\
%  = \, &\sum_{i=1}^{k-1}(-1)^{i-1}(1\otimes \omega\otimes\cdots\otimes\nabla\omega
%    \otimes\cdots\otimes\omega)+ \\
%  &+  \sum_{i=1}^{k-1}(-1)^{i}(1\otimes \omega\otimes\cdots\otimes \Omega\otimes
%    \cdots\otimes \omega)\\
  = \, &\hbar \nabla((\omega)_{k-1})+\hbar\sum_{j=1}^{k-1}(-1)^j
  1\otimes \omega\otimes \cdots \otimes (\Omega-\tilde{R})\otimes \cdots \otimes \omega.
\end{split}
\]
\end{proof}

\begin{lemma}
\label{lem:B-shuffle}
 For $\omega$ as in Lemma \ref{lem:shuffle} and every  $a\in \normC_l(A)$
one has
 \[ \normB (a\times (\omega)_k)= \normB a \times (\omega)_k .\]
\end{lemma}
\begin{proof}
The claim follows by a straightforward computation:
\[
\begin{split}
 \normB & \, (a\times (\omega)_k) = \\
 & = \normB \big( (-1)^{\deg(b_0)(\sum_j \deg(a_j))}
 \sh_{l,k}(a_0 \otimes a_1\otimes \cdots \otimes
 a_l \otimes \omega\otimes \cdots\otimes  \omega) \big)\\
 & =\sum_{i=l+1}^{k+l}(-1)^{i(k+l)}1\otimes
 \sh_{l,k}(\omega\otimes\cdots\otimes
 \omega\otimes a_0 \otimes a_1\otimes \cdots\otimes  a_l \otimes\omega\otimes
 \cdots\otimes  \omega)\\
 &\hspace{1em} +\sum_{i=1}^{l}(-1)^{i(k+l)}1\otimes \sh_{l,k}
 (a_i\otimes\cdots\otimes a_l\otimes \omega\otimes\cdots\otimes \omega\otimes
 a_0\otimes a_1\otimes \cdots \otimes  a_{i-1})\\
 & =\sum_{i=0}^{l}(-1)^{il} \sh_{l+1,k}(1\otimes a_i \otimes\cdots\otimes
 a_{l}\otimes a_0\otimes \cdots\otimes  a_{i+1}\otimes\omega\otimes\cdots
 \otimes  \omega) \\ & = \normB a\times (\omega)_k .
\end{split}
\]
\end{proof}

\subsection{Cyclic cocycles on deformation quantizations of symplectic manifolds}
\label{subsec:cyccohdefquant}
In this section, we study the cyclic cohomology of the quantum algebra
\[
  \calA^{((\hbar))}_D :=\calA^{\hbar}_D[\hbar^{-1}],
\]
which is  the kernel of a Fedosov connection $D=d+
\frac{1}{\hbar} [A, - ]$ on $\calW[\hbar^{-1}]$.
 Note that since $D$
is a local operator we in fact obtain a sheaf of quantum algebras on
$M$, which we also denote by $\calA^{((\hbar))}_D$. Let
$\calA^{((\hbar))}_\text{\rm \tiny cpt}$ be its space of sections
with compact support. We will define in this section an S-morphism
$\sfQ$ between the mixed complexes
\[
  \big( \Omega^\bullet (M)  , d , 0 \big) \quad \text{and} \quad
  \Big( \normC^\bullet \big( \calA^{((\hbar))}_\text{\tiny\rm cpt}\big),b,B\Big).
\]
In the construction of $\sfQ$  we will use the mixed sheaf complex
$\big( \normscrC^\bullet \big( \calA^{((\hbar))} \big),b,B \big)$
defined in Appendix \ref{subsec:loc} and Theorem \ref{thm:locqism}
which tells that the complex of its global section spaces is
quasi-isomorphic to the mixed complex
$\big(\normC^\bullet\big(\calA^{((\hbar))}_\text{\tiny\rm cpt}
\big),b,B\big)$.

In the following definitions, we consider the shuffle product on the
Hochschild chains of the graded algebra $\calW\otimes_{\calC^\infty
(M)}\Omega^\bullet(M)((\hbar))$ with a degree $1$ derivation
$\nabla$, the symplectic connection,  and a generalized
Maurer-Cartan element $A$, the Fedosov connection.

\begin{remark}The cyclic cocycle $(\tau_0,\ldots,\tau_{2n})\in \Tot^{2n}
\calB \overline{C}^\bullet (\Weyl_{2n})$ defined in Def.
\ref{dfn:tau} extends uniquely to a continuous cyclic cocycle on the
algebra $\fWeyl$ with the same properties as Prop.
\ref{prop:relative-tau}.
\end{remark}
\begin{definition}
\label{dfn:psi}
  Define $\Psi_{2k}^i\in \Omega^i (M)\otimes_{\calC^\infty (M)}
  \big(\calW^{\otimes (2k-i+1)} \big)^* (M) $ by putting
\[
 \Psi_{2k}^i\big( a_0 \otimes \cdots \otimes a_{2k-i} \big) :=
 \left( \frac{1}{\hbar} \right)^i \!\!
 \tau_{2k}\big( (a_0\otimes \cdots \otimes a_{2k-i})
 \times (A)_i \big).
\]
To explain this definition a bit more: for a given point $x\in M$,
we have used the natural identification of the fiber of
$\calW[\hbar^{-1}]$ over $x$ with the Weyl algebra $\fWeyl_{2n}$ in
the formula above. The cochain $\tau_{2k}$ is defined as in
Definition \ref{dfn:tau}, $(-)^*$ denotes the dual bundle functor,
and $a_0, \cdots , a_{2k-i}$ are germs of smooth sections of $\calW$
at $x$. It is important to remark that the definition above does not
depend on the decomposition $D=\nabla+A$ of the Fedosov connection:
a different choice amounts to adding a $\mathfrak{sp}_{2n}$ valued
one-form to $A$. By Proposition \ref{prop:relative-tau}, this yields
the same result.
\end{definition}
\begin{proposition}
\label{prop:psi}
For every chain $a_0 \otimes \cdots \otimes a_{2k-i}\in C_{2k+1-i}\big(
\calA^{((\hbar))}_\text{\tiny\rm cpt} \big)$
the above defined $\Psi^i_{2k}$ satisfies the following equality:
\begin{equation}
\label{eq:psi}
\begin{split}
(-1)^i \; & d\Psi_{2n-2k}^{i}(a_0\otimes
\cdots\otimes  a_{2k+1-i})\\
=&\Psi^{i+1}_{2n-2k}(b(a_0\otimes  \cdots\otimes
a_{2k+1-i}))+\Psi^{i+1}_{2n-2k+2}( \normB (a_0\otimes \cdots\otimes
a_{2k+1-i})) .
\end{split}
\end{equation}
\end{proposition}
\begin{proof}
To prove Eq.~(\ref{eq:psi}), apply $\tau_{2k}$ to
Eq.~(\ref{eq:b-shuffle}) with $\omega=  A$ and check
that
\begin{equation}\label{eq:tau-shuffle}
\begin{split}
  \left( \frac{1}{\hbar} \right)^i \, & \tau_{2k} (b(a\times (A)_i))= \\
  = & \left( \frac{1}{\hbar} \right)^i\tau_{2k}(b(a)\times (A)_i)+
  (-1)^{2k-i+1}\left( \frac{1}{\hbar} \right)^i\tau_{2k}(a\times b(A)_i)\\
  &-(-1)^{2k-i+1}\sum_{j=0}^{2k+1-i}
  \left( \frac{1}{\hbar} \right)^{i-1}\tau_{2k} \big(
  (a_0\otimes  \cdots\otimes  \frac{1}{\hbar}[A, a_j] \otimes \cdots \\
  & \hspace{60mm} \cdots\otimes  a_{2k+1-i})\times (A)_{i-1}\big)
\end{split}
\end{equation}
for every $a=a_0 \otimes \cdots \otimes a_{2k+1-i}\in
C_{2k+1-i}\big( \calA^{((\hbar))}_\text{\tiny\rm cpt} \big)$.
Recall that by definition, every $a_j\in \calA^{((\hbar))}$
satisfies the equality
\[
 \nabla a_j+ \frac{1}{\hbar} [A, a_j]=0.
\]
Therefore, we have
\[
\begin{split}
\tau_{2k}& \;((a_0\otimes  \cdots\otimes  \frac{1}{\hbar}[A, a_j]
\otimes \cdots\otimes  a_{2k+1-i})\times (A)_{i-1}) = \\
&=-\tau_{2k}((a_0\otimes  \cdots\otimes  \nabla a_j\otimes
\cdots\otimes a_{2k+1-i})\times (A)_{i-1}) .
\end{split}
\]
By Lemma \ref{lem:b-omega}, we obtain
\begin{equation}
\label{eq:b-A}
\begin{split}
\frac{1}{\hbar} \; & a\times b((A)_i)= \\
& = a\times \nabla((A)_{i-1})+ \sum_{j=1}^{i-1}(-1)^j a\times
(1\otimes A\otimes \cdots\otimes  (\Omega-\tilde{R})\otimes
\cdots\otimes A).
\end{split}
\end{equation}
Recall that $\Omega\in\Omega^2(M,\C[[\hbar]])$ is in the center of
$\calW$ and $\tilde{R}$ is in the image of $\mathfrak{sp}_{2n}$ in
$\calW$. Therefore, since the $\tau_{2k}$ are reduced
$\mathfrak{sp}_{2n}$ basic cochains by Prop.
\ref{prop:relative-tau},
\[
 \tau_{2k}(a\times (1 \otimes A \otimes
 \cdots\otimes  (\Omega-\tilde{R})\otimes \cdots\otimes A))=0.
\]
Applying $\tau_{2k}$ to Eq.~(\ref{eq:b-A}), one gets
\[
  \frac{1}{\hbar}\tau_{2k}(a\times b(A)_i)=\tau_{2k}(a\times
  \nabla((A)_{i-1})).
\]
\noindent Therefore, we have that
\begin{displaymath}
\begin{split}
 (-1)^{2k+1-i}& \left( \frac{1}{\hbar} \right)^i\tau_{2k}(a\times b(A)_i)\\
 &-(-1)^{2k+1-i}\sum_{j=0}^{2k+1-i}
 \left( \frac{1}{\hbar} \right)^{i-1}\tau_{2k}
 ((a_0\otimes  \cdots\otimes  \frac{1}{\hbar}[A,a_j]\otimes \cdots\\
  & \hspace{60mm} \cdots \otimes a_{2k+1-i})\times (A)_{i-1})\\
 = \, &(-1)^{2k+1-i}\left( \frac{1}{\hbar} \right)^{i-1}\tau_{2k}(a\times
\nabla((A)_{i-1}))\\
 &+(-1)^{2k+1-i}\sum_{j=0}^{2k+1-i}\left( \frac{1}{\hbar} \right)^{i-1}\tau_{2k}
 ((a_0\otimes \cdots\otimes  \nabla a_j\otimes \cdots\otimes  a_{2k+1-l})\times
 (A)_{i-1})\\  = \, &(-1)^{2k+1-i}
 \left( \frac{1}{\hbar} \right)^{i-1}d\tau_{2k}(a\times (A)_{i-1}) .
\end{split}
\end{displaymath}
Applying Corollary \ref{cor:cyclic-tau}, we have that
\begin{equation}
\label{eq:b-tau}
\begin{split}
b\tau_{2k} (a\otimes (A)_i)= \, & -B\tau_{2k+2}(a\times (A)_i) \\
= \, &-\tau_{2k+2}(\normB (a\times (A)_i))\\
= \, &-\tau_{2k+2}(\normB(a)\times (A)_i)\ \ \ \ \ \
(\text{by Lemma \ref{lem:B-shuffle}}).\\
\end{split}
\end{equation}
Eq.~(\ref{eq:b-tau}) entails
\[
\begin{split}
\left( \frac{1}{\hbar} \right)^i\tau_{2k}(b(a\times (A)_i))= -\left(
\frac{1}{\hbar} \right)^i\tau_{2k+2}(\normB (a)\times (A)_i).
\end{split}
\]
Now going back to Eq. (\ref{eq:tau-shuffle}), we obtain
\begin{equation}
\label{eq:t-shuffle-new}
\begin{split}
-\left( \frac{1}{\hbar} \right)^i \, &
\tau_{2k+2}(\normB(a)\times (A)_i)= \\
 = \, & \left( \frac{1}{\hbar} \right)^i\tau_{2k}(b(a)\times (A)_i)+
(-1)^{i-1}\left( \frac{1}{\hbar} \right)^{i-1}d\tau_{2k}(a\times
(A)_{i-1}).
\end{split}
\end{equation}
But this is equivalent to
\[
\begin{split}
(-1)^{i-1} \, & \left( \frac{1}{\hbar} \right)^i
d\tau_{2k}(a\times (A)_i)= \\
& =  \left( \frac{1}{\hbar} \right)^{i+1}\tau_{2k+2}(\normB
(a)\times (A)_{i+1})+\left( \frac{1}{\hbar} \right)^{i+1}
\tau_{2k}(b(a)\times (A)_{i+1}),
\end{split}
\]
which by the definition of $\Psi_{2k}^i$ entails Eq.~(\ref{eq:psi}).
\end{proof}
\begin{definition}
\label{dfn:chi}
For every $i, r$ with $2r\leq i$ and every open $U\subset M$ define a morphism
$\chi_{i,U}^{i-2r}:\Omega^i(U)((\hbar))\to \normscrC^{i-2r}
\big(\calA^{((\hbar))}_{\rm\tiny cpt}\big) (U)$ by
\[
 \chi_{i,U}^{i-2r}(\alpha)(a_0\otimes  \cdots\otimes  a_{i-2r})=
 \int_U\alpha\wedge \Psi_{2n-2r}^{2n-i}(a_0\otimes  \cdots\otimes  a_{i-2r}),
\]
where $\alpha \in \Omega^i(U)((\hbar))$ and
$a_0 , \cdots , a_{i-2r} \in \calA_\text{\tiny \rm cpt}^{((\hbar))} (U)$.
The integral converges because $a_0,\ldots,a_{i-2r}$ have compact support.
Obviously, the $\chi_{i,U}^{i-2r}$ form the local components of sheaf morphisms
$\chi_i^{i-2r}:\Omega^i_M
((\hbar))\to \normscrC^{i-2r}\big(\calA^{((\hbar))}\big)$.
Using these, define further sheaf morphisms
$\chi_i:\Omega^i_M ((\hbar))\to \Tot^i \calB \normscrC^\bullet
(\calA^{((\hbar))})$ by
\begin{equation}
\label{eq:consmor1}
 \chi_i=\sum _{2r\leq i}\chi_i^{i-2r}.
\end{equation}
\end{definition}
\noindent
The $\chi_i$ have the following crucial property.
\begin{proposition}
\label{prop:mixedchainmap}
  For every $\alpha\in \Omega^\bullet(U)((\hbar))$ with $U\subset M$ open
  one has
  \[
    (b+B)\chi_\bullet(\alpha)=\chi_\bullet(d\alpha).
  \]
\end{proposition}
\begin{proof}
Writing out the definition of $\chi$, we have to show that
\[
\begin{split}
\int_M \; & d\alpha\wedge \sum_{2r\leq i+1}\Psi_{2n-2r}^{2n-i-1}(a_0\otimes
\cdots\otimes  a_{i+1-2r}) = \\
= \, &\int_M \alpha \wedge \sum_{2r\leq i+1}\Psi^{2n-i}_{2n-2r}(b(a_0\otimes
\cdots\otimes
a_{i+1-2r}))+\\
& +\int_M\alpha\wedge \sum_{2r\leq i+1}\Psi^{2n-i}_{2n-2r+2}(\normB(a_0\otimes
\cdots\otimes  a_{i+1-2r}))
\end{split}
\]
holds true for all chains $a_0\otimes  \cdots\otimes a_{i-2r+1}
\in C_{2k+1-i}\big( \calA^{((\hbar))}_\text{\tiny \rm cpt} \big)$.
Since $M$ is a closed manifold, by integration by parts, this equality
is equivalent to
\[
\begin{split}
(-1)^i\int_M & \; \alpha\wedge \sum_{2r\leq i+1}d\Psi_{2n-2r}^{2n-i-1}
(a_0\otimes \cdots\otimes  a_{i+1-2r}) = \\
=\,&\int_M \alpha \wedge \sum_{2r\leq i+1}\Psi^{2n-i}_{2n-2r}(b(a_0\otimes
\cdots\otimes
a_{i+1-2r}))+\\
& +\int_M\alpha\wedge \sum_{2r\leq i+1}\Psi^{2n-i}_{2n-2r+2}(\normB (a_0\otimes
\cdots\otimes  a_{i+1-2r})).
\end{split}
\]
This is a corollary of Prop. \ref{prop:psi}
\end{proof}
\noindent
As a corollary of Proposition \ref{prop:mixedchainmap}, we obtain for every $i$
a sheaf morphism
\[
 \scrQ^i: \Tot^i\calB \Omega^\bullet_M ((\hbar)):=
 \bigoplus_{2r \leq i }\Omega^{i-2r}_M((\hbar))
 \to \Tot^i \calB \normscrC^\bullet \big( \calA^{((\hbar))} \big)
\]
which over $U\subset M$ open evaluated on forms
$\alpha_{i-2r} \in \Omega^{i-2r} (U) ((\hbar)) $ gives
\begin{equation}
\label{eq:consmor2}
\scrQ_U^i \Big( \sum_{2r \leq i} \alpha_{i-2r} \Big)=
\frac{1}{(2\pi\sqrt{-1})^n} \sum_{2r \leq i} \chi_{i-2r,U}(\alpha_{i-2r}),
\end{equation}
where we have viewed $\chi_{i-2r,U}(\alpha_{i-2r})$ as an element in
$\Tot^i \calB \normscrC^\bullet \big( \calA^{((\hbar))} \big) (U)$
via the embedding
$\Tot^{i-2r} \calB \normscrC^\bullet \big( \calA^{((\hbar))} \big)\hookrightarrow
\Tot^i \calB \normscrC^\bullet \big( \calA^{((\hbar))} \big)$.

\begin{theorem}
\label{thm:quasi}
The above defined sheaf morphism
\[
  \scrQ: \big( \Tot^\bullet \calB \Omega^\bullet_M ((\hbar)), d \big) \to
 \big( \Tot^\bullet \calB \normscrC^\bullet(\calA^{((\hbar))}), b+B \big)
\]
is an S-morphism between mixed cochain complexes of sheaves and a
quasi-isomorphism of the sheaves of cyclic cochains.
\end{theorem}
\begin{proof}
By Proposition~\ref{prop:mixedchainmap}, $\scrQ$ is a morphism of sheaf
complexes. Together with Eqs.~\eqref{eq:consmor1} and \eqref{eq:consmor2} this
entails that $\scrQ$ is an S-morphism.
To prove the second claim it therefore suffices by \cite[Prop.~2.5.15]{loday}
that the $\big( \chi^i_i\big)_{i\in \N}$
% \Omega^i_M \rightarrow \scrC^i \big( \calA^{((\hbar))} \big)$
form a quasi-isomorphism of sheaf complexes
$\chi : \Omega^\bullet_M ((\hbar)) \rightarrow
\scrC^\bullet \big( \calA^{((\hbar))} \big)$.

This follows from a spectral sequence argument provided in the
following. We remark that $\chi$ does not preserve the
$\hbar$-filtration on both complexes. Therefore, we need to modify
$\chi_i^i$ by $\frac{1}{\hbar^{i-n}}$ without changing the final
conclusion. Under this change, we will have a cochain map $\chi :
(\Omega^\bullet_M ((\hbar)), \hbar d) \rightarrow (\scrC^\bullet
\big( \calA^{((\hbar))}\big), b)$ compatible with the
$\hbar$-filtration. Then we consider the induced morphism on the
corresponding spectral sequences. The $E_0$ terms of $\scrC^\bullet
\big( \calA^{((\hbar))} \big)$ is equal to the localized Hochschild
cochain sheaf complex $\scrC^\bullet \big( \calC^\infty(M)((\hbar))
\big)$, which is quasi-isomorphic to the sheaf of de Rham currents
on $M$, cf.\ \cite{c:book}. The induced differential on $E_0$ under
this quasi-isomorphism is dual to the Poisson differential on the
sheaf of differential forms on $M$

As all the above sheaves are fine, it sufficient to prove the claim
over each element of an open cover of $M$ where each of its open sets is
symplectic diffeomorphic to an open contractible subset of $\R^{2n}$
equipped with the standard symplectic form: a Darboux chart. We check that the induced
$\chi^i_i$ on $E_0$ over such open set $U$ is a quasi-isomorphism.
Over $U$, the $E_0$ component $\tilde{\chi}^i_i$ of $\chi^i_{i}$ in
Def.~\ref{dfn:chi} is computed to be
\begin{equation}\label{eq:chi-e-0}
 \tilde{\chi}_i^i(\alpha)(a_0\otimes a_1\otimes \cdots \otimes
 a_i)=
 \int_{U}\alpha \wedge \ast(a_0da_1\wedge\cdots \wedge da_i),
\end{equation}
where $\ast:\Omega^{i}_M\rightarrow \Omega^{2n-i}_M$ is the symplectic Hodge star operator on $M$ introduced by Brylinski \cite{bry}.

By the identity $d^\pi=(-1)^i\ast d \ast $ for the Poisson homology differential
$d^\pi$ on $\Omega^{i+1}_M$:
\begin{equation}\label{eq:poisson-d}
\begin{split}
 \int_{U} \; & d\alpha \wedge \ast (a_0da_1\wedge\cdots \wedge da_{i+1})= \\
 & = (-1)^{i}\int_{U}\alpha \wedge d \ast (a_0da_1\wedge\cdots \wedge
 da_{i+1})\\
 &=\int_U\alpha \wedge \ast\left(d^\pi(a_0da_1\wedge \cdots \wedge da_{i+1})\right).
\end{split}
\end{equation}
\noindent Combining Eq.~(\ref{eq:chi-e-0})-(\ref{eq:poisson-d}), we
see that $\tilde{\chi}_i^i$ maps the de Rham differential on
$\Omega^\bullet_M((\hbar))$ to a differential on the cohomology of
$\scrC^\bullet \big( \calC^\infty(M)((\hbar)) \big)$, which is dual
to the Poisson differential $d^\pi$. Therefore, we conclude that on
each $U$, the chain map $(\tilde{\chi}_i^i)_{i\in \N}$ is a
quasi-isomorphism at the $E_0$ level. This proves that
$(\chi_i^i)_{i\in \N}$ is a quasi-isomorphism.
\end{proof}
\begin{corollary}
Over global sections, $\scrQ$ induces an S-quasi-isomorphism
\[
  \sfQ: \big( \Tot^\bullet \calB \Omega^\bullet (M)((\hbar)) , d \big) \to
  \big( \Tot^\bullet \calB \normC^\bullet
  \big(\calA^{((\hbar))}_\text{\tiny \rm cpt}\big), b+B \big).
\]
\end{corollary}
%%% Local Variables:
%%% mode: latex
%%% TeX-master: "HigherIndex"
%%% End:

%% file: IndThms.tex
%
%
\section{Algebraic index theorems}
\label{Sec:AlgIndThm}
In this section we study Connes' pairing between the $K$-theory of
$\calA^{((\hbar))}_{\rm\tiny cpt}$ and a cocycle
$\sfQ(c)\in \Tot^\bullet \calB\normC^\bullet
(\calA^{((\hbar))}_\text{\tiny\rm cpt}(M))$, where $c$ is an element in
$\Tot^\bullet \calB \Omega^\bullet(M)=
\bigoplus_{2l\leq \bullet}\Omega^{\bullet-2l}((\hbar))$.
This results in an algebraic index theorem which computes this pairing
in terms of topological data of the underlying manifold $M$.

\subsection{The pairing between cyclic cohomology and $K$-theory}
We start with briefly reviewing the general theory \cite[Sec.
8.3]{loday} of a pairing between cyclic cohomology and $K_0$ group of a
unital algebra.

Let $A$ be a unital algebra over a field $\Bbbk$ and let $e$ be an idempotent
of $A$. The Chern character $\Ch_k(e)$ is a cocycle in
\[
  \overline{\calB}_{2k}(A)= A\otimes \overline{A}^{\otimes(2k)}\, \oplus  \,
  A\otimes \overline{A}^{\otimes(2k-2)}\, \oplus \cdots \oplus \, A
\]
defined by the following formulas
\begin{equation}
\label{eq:defchern}
 \begin{split}
  \Ch_k(e)&=(c_k, c_{k-1},\cdots, c_0)\in \overline{\calB}_{2k}(A), \quad \text{where} \\
   c_i&=(-1)^i\frac{(2i)!}{i!}
   (e-\frac{1}{2})\otimes e^{\otimes (2i)}\in A\otimes \overline{A}^{2i}\quad \text{for $i=1,\ldots , k$}\\
   c_0&=e\in A.
 \end{split}
\end{equation}
It is easy to check that $\Ch_k(e)$ is $b+B$ closed.
One then defines a pairing between a $(b+B)$-cocycle $\phi=(\phi_{2k},
\cdots, \phi_{0})$ and a projection $e\in A$ by the canonical
pairing between $C_k(A)$ and $\normC^k(A)$,
\[
 \langle \phi, e \rangle :=\langle \phi, \Ch_k(e)\rangle=
 \sum_{l=0}^k(-1)^l\frac{(2l)!}{l!}\phi_{2l}\left((e-\frac{1}{2}) \otimes e\otimes \cdots \otimes e\right).
\]
This construction descends to cohomology and yields the desired pairing
\[
HC^k(A)\times K_0(A)\to \Bbbk.
\]

Now let $M$ be a symplectic manifold, and $\calA^{((\hbar))}(M)$ be a Fedosov deformation
quantization of $M$ as constructed in the previous section.
We apply the above to obtain a pairing between the cyclic cohomology
$HC^\bullet(\calA^{((\hbar))}_{\rm\tiny cpt})$ and the $K_0$ group of $\calA^{((\hbar))}_{\rm\tiny cpt}(M)$. (To define the Chern character like (\ref{eq:defchern}), we usually adjoin a unit to the algebra $\calA^{((\hbar))}_{\rm\tiny cpt}(M)$. )

Recall from \cite[6.1]{fe} that an element in $K_0 \big(\calA^{((\hbar))}_{\rm\tiny cpt}\big)$
can be represented by a pairing of projections $P_0, P_1$ in
$\frM_k \big( \calA^{((\hbar))} \big)$ for some $k\geq 0$ such that $P_0-P_1$ is
compactly supported. (By $\frM_k \big( \calA^{((\hbar))} \big)$ we mean the algebra of $k\times k$-matrices with coefficient in $\calA^{((\hbar))}$.) The set of all such pairs of projections forms
a semi-group. It is proved in \cite[6.1]{fe} that modulo stabilization
this semi-group is isomorphic to the $K$-group of $M$.
Now let $\phi$ be a $(b+B)$-cocycle of $\calA^{((\hbar))}$ which has degree $2k$.
Then the pairing between $\phi=(\phi_0, \cdots, \phi_{2k})$ and $e=(P_1,P_2)$
a representative of a $K$-group element of $\calA^{((\hbar))}$ is defined as
\[
  \langle\phi, e\rangle : =\langle\phi, P_1\rangle-\langle\phi, P_2\rangle.
\]

\subsection{Quantization twisted by a vector bundle}
In the following, we explain how to reduce the computation of the
above pairing to the trivial case that $e=1$ in $\calA^{((\hbar))}$.
Define $p_1=P_1|_{\hbar=0}$ and $p_2=P_2|_{\hbar=0}$. Since $P_1\star P_1=P_1$ and
$P_2\star P_2=P_2$, the matrices $p_1$ and $p_2$ are projections in
$\frM_n(C^\infty(M))$ and therefore define vector bundles $V_1$ and $V_2$ on
$M$. Furthermore,  $V_1$ and $V_2$ are isomorphic outside a compact
of $M$.

Following \cite{fe}, we can twist the quantum algebra
$\calA^{\hbar}$ by the bundles $V_1$ and $V_2$. We consider the
twisted Weyl algebra bundles
$\calW_{V_1}=\calW\otimes\End(V_1)$ and
$\calW_{V_2}=\calW\otimes\End(V_2)$. Fixing
connections $\nabla_{1}$ and $\nabla_{2}$ on $V_1$ resp.~$V_2$, we
obtain connections $\nabla_{V_1}=\nabla\otimes1+1\otimes \nabla_{1}$
and $\nabla_{V_2}=\nabla\otimes1+1\otimes \nabla_{2}$ on
$\calW_{V_1}$ resp.~$\calW_{V_2}$. {\sc Fedosov} proved in \cite{fe}
that there are flat connections $D_{V_1}=\nabla_{V_1}+
\frac{1}{\hbar} [A_{V_1}, -]$ and $D_{V_2}=\nabla_{V_2}+
\frac{1}{\hbar} [A_{V_2},-]$ on $\calW_{V_1}$ resp.~$\calW_{V_2}$
such that the algebra of flat sections forms a deformation
quantization twisted by $V_1$ resp.~$V_2$. The corresponding
deformation quantization sheaf is denoted by
$\calA^{((\hbar))}_{V_1}$ resp.~$\calA^{((\hbar))}_{V_2}$.

Observe that the cocycle $(\tau_0, \cdots, \tau_{2n})$ on $\Weyl_{2n}$
can be extended to the algebra $\VWeyl_{2n}:=\Weyl_{2n}\otimes \End(V)$ for
any finite dimensional vector space $V$ by putting
\[
  \tau_{2k}^V\big( (a_0\otimes M_0) \otimes \cdots \otimes
  (a_{2k} \otimes M_{2k}) \big) := \tau(a_0 \otimes \cdots \otimes a_{2k})
  \tr(M_0M_1\cdots M_{2k}).
\]
As in the proof of Corollary \ref{cor:cyclic-tau} one checks that
$(\tau_0^V, \cdots, \tau_{2n}^V)$ is a $(b+B)$-cocycle on $\VWeyl_{2n}$. Hence
we can extend Def.~\ref{dfn:psi} to define twisted
${\Psi}_{V_j,2k}^i$ for $j=1,2$ by
\[
\begin{split}
 {\Psi}_{V_j,2k}^i & \, \big( a_0 \otimes \cdots \otimes a_{2k-i}\big) = \\
 & = \left( \frac{1}{\hbar} \right)^{i} \tau^{V_j}_{2k} \big(
 (a_0 \otimes \cdots \otimes a_{2k-i})\times (A_{V_j})_{i} \big),
\end{split}
\]
where $a_0, \cdots , a_{2k-i}$ are germs of smooth sections of $\calW_{V_j}$ at $x$.
Moreover, we define sheaf morphisms
$\chi_{V_j,i}^{i-2l} :\Omega^i_M((\hbar))\to
\normscrC^{i-2l}\big(\calA^{((\hbar))}_{V_j}\big)$ by setting over $U \subset M$ open
\[
  \chi_{V_j,i,U}^{i-2l}(\alpha)(a_0 \otimes \cdots \otimes a_{i-2l}) :=
  \int_M\alpha\wedge {\Psi}_{V_j,2n-2l}^{2n-i}\big( a_0 (x)\cdots\otimes a_{i-2l} (x)\big),
\]
where $\alpha \in \Omega^i (U)((\hbar))$ and where
$a_0, \cdots , a_{2k-i} \in \calA^{((\hbar))}_{V_j,\text{\tiny \rm cpt}} (U)$ are sections
of the twisted deformation quantization sheaf with compact support in $U$.
Like in Section \ref{subsec:cyccohdefquant} we then obtain S-quasi-isomorphisms
of mixed sheaf complexes
\[
 \scrQ_{V_j}: \big( \Tot^\bullet \calB \Omega^\bullet_M ((\hbar)), d \big) \to
 \big( \Tot^\bullet \calB \normscrC^\bullet(\calA^{((\hbar))}_{V_j}), b+B \big),
 \quad j=1,2.
\]
Over global sections, $\scrQ_{V_j}$ then induces an S-quasi-isomorphism
\[
  \sfQ_{V_j}: \big( \Tot^\bullet \calB \Omega^\bullet (M)((\hbar)) , d \big) \to
  \big( \Tot^\bullet \calB \normC^\bullet
  \big( \calA^{((\hbar))}_{V_j,\text{\tiny \rm cpt}}\big), b+B \big).
\]
Generalizing \cite[Thm.~3]{cd}, we have the following proposition:
\begin{proposition} \label{prop:reduction}
For any a closed differential
$\alpha\in \Tot^\bullet \calB \Omega^\bullet (M)((\hbar))$ and
projections $P_1$ and $P_2$ of $\calA^{((\hbar))}$ with $P_1-P_2$
compactly supported, one has
\[
  \langle \sfQ(\alpha), P_1-P_2\rangle=
  \langle \sfQ_{V_1}(\alpha), 1\rangle - \langle \sfQ_{V_2}(\alpha), 1\rangle.
\]
\end{proposition}
\begin{proof}
The proof of \cite[Thm.~3]{cd} applies verbatim.
\end{proof}
\subsection{Lie algebra cohomology}\label{sec:lie-alg}
In the following paragraphs, we use Lie algebra cohomology to
determine the pairing $\langle \sfQ_{V}(\alpha),1\rangle$ locally for a vector
bundle $V$ on $M$.
By definition, the pairing $\langle \sfQ_V(\alpha),1\rangle$ for an
element $\alpha=(\alpha_0, \cdots, \alpha_{2k} )\in
\Tot^{2k} \calB \Omega^\bullet (M)((\hbar))$ is equal to
\[
\begin{split}
  &\Big\langle\, \frac{1}{(2\pi \sqrt{-1})^n}\sum_{l\leq k} \,  \chi_{V,2k-2l}(\alpha_{2k-2l}),
  1\Big\rangle\\
  =& \Big\langle\frac{1}{(2\pi \sqrt{-1})^n}\sum_{l\leq k, j\leq k-l}\chi_{V, 2k-2l}^{2k-2l-2j}
  (\alpha_{2k-2l}),1\Big\rangle\\
   = &\frac{1}{(2\pi \sqrt{-1})^n}\sum_{l\leq k, j\leq k-l}\frac{(-1)^{k-l-j}(2k-2l-2j)!}{(k-l-j)!}
  \int_M\alpha_{2k-2l}\wedge \Psi_{V,2n-2j}^{2n-2k+2l}
  (1 \otimes \cdots \otimes 1).
\end{split}
\]
Now observe that $\Psi_{V, 2n-2j}^{2n-2k+2l}(1 \otimes \cdots \otimes 1)$
vanishes when $j<k-l$ since $\tau_{2n-2j}$ is a normalized cochain. Hence
\[
\begin{split}
  \langle \sfQ_V(\alpha), 1\rangle \, & =
  \sum_{l\leq k}\frac{1}{(2\pi \sqrt{-1})^n}
  \int_M\alpha_{2k-2l}\wedge \Psi_{V,2n-2k+2l}^{2n-2k+2l}(1) = \\
  & = \sum_{l=0}^k\frac{1}{(2\pi \sqrt{-1})^n}
  \int_M\alpha_{2l}\wedge\Psi^{2n-2l}_{V, 2n-2l}(1).
\end{split}
\]
These considerations show that for the computation of the pairing between
an element $\alpha \in \Tot^\bullet \calB \Omega^\bullet (M)((\hbar))$ and
a class in $K_0 \big( \calA^{((\hbar))} \big)$ it is sufficient to determine
$\Psi_{V, 2n-2l}^{2n-2l}(1)$ for all $l\leq n$.

To achieve this goal we will apply methods from Lie algebra cohomology, namely the Chern--Weil homomorphism.
To this end let us first review the standard map from the Hochschild cochain
complex to the corresponding Lie algebra cochain complex, which can be
found in \cite{loday}.

Let $A$ be a unital algebra. Consider Lie algebra $\mathfrak{gl}_N(A)$ of $N\times N$-matrices with coefficients in $A$. There is a
chain map $\phi_N$ from the Hochschild cochain complex
$C^\bullet(A)$ to the Lie algebra cochain complex
$C^\bullet\big(\mathfrak{gl}_N(A); \mathfrak{gl}_N(A)^* \big)$:
\begin{equation}
\label{HoLa}
\begin{split}
 \phi^N(c)\, & \big( (M_1\otimes a_1)\otimes \cdots \otimes (M_k\otimes a_k)
 \big)(M_1\otimes a_1)\\
  &=\sum_{\sigma\in S_k}\sgn(\sigma)c(a_0\otimes a_{\sigma(1)}\otimes
  \cdots \otimes a_{\sigma(k)})\tr(M_0M_{\sigma(1)}\cdots M_{\sigma(k)}).
\end{split}
\end{equation}
We define $\Theta_{V, N, 2k}$ to be  $\phi^N(\tau^V_{2k})\in C^{2k}
\big( \gl_N(\fWeyl_{2n}^V);\gl_N(\fWeyl_{2n}^V)^* \big)$. It is easy
to check that $\Psi_{V, 2n-2k}^{2n-2k}(1)=
\left(\frac{1}{\hbar}\right)^{2n-2k}\frac{1}{(2n-2k)!}
\Theta_{V,2n-2k}(A\wedge\cdots \wedge A)(1)$.

\begin{proposition}
\label{prop:theta-lie-algebra}
For any $k\leq n$, $\Theta_{V, N, 2k}(1)$ is a cocycle in the relative
Lie algebra cohomology complex
$C^{2k}\big( \gl_N( \fWeyl_{2n}^V), \gl_N \oplus \gl_V \oplus
\mathfrak{sp}_{2n} \big)$ and satisfies
\[
 \Theta^N_{V,2n}(p_1\wedge q_1\wedge\cdots \wedge p_n\wedge q_n)=N\dim(V).
\]
\end{proposition}
\begin{proof}
Since $1$ is in the center of $\fWeyl_{2n}^V$, we have the following
equation
\[
  \partial_\text{\tiny\rm Lie}
  ((\Theta_{V,N,2k})(1))=\partial_\text{\tiny\rm Lie}(\Theta_{V,N,2k})(1).
\]
On the right hand side of the above equation, $\Theta_{V,N,2k}$ is
viewed as a Lie algebra cochain in $C^{2k}\big(
\gl_N(\fWeyl_{2n}^V);\gl_N(\fWeyl_{2n}^V)^* \big)$. Furthermore,
since $\phi^N$ is a morphism of cochain complexes, we have that $\partial_\text{\tiny\rm
Lie} \Theta_{V,N,2k}(1)=
\partial_\text{\tiny\rm Lie}\phi^N(\tau^V_{2k})=\phi^N(b(\tau^V_{2k}))$.
Since $(\tau^V_0, \cdots, \tau^V_{2n})$ is a $(b+B)$-cocycle, we
have $b(\tau^V_{2k})=-B(\tau^V_{2k+2})$ and $\partial_\text{\tiny\rm
Lie}\Theta_{V,N,2k}(1)=-\phi^N(B(\tau^V_{2k+2}))(1)$. Now we compute
\[
\begin{split}
 \phi^N & (B(\tau^V_{2k+2}))(1)\big( (a_1\otimes M_1) \otimes \cdots \otimes
 (a_{2k+1}\otimes M_{2k+1})\big)\\
 =&\sum_{\sigma\in S_{2k+1}}\sgn(\sigma)B(\tau^V_{2k+2})(1\otimes a_{\sigma(1)}\otimes
 \cdots \otimes a_{\sigma(2k+1)})\cdot \tr(M_{\sigma(1)}\cdots M_{\sigma(2k+1)})\\
 =&\sum_{\sigma\in S_{2k+1}}\sum_{i}\sgn(\sigma)\tau^V_{2k+2}(1\otimes a_{\sigma(i)}\otimes
 \cdots \otimes a_{\sigma(2k+1)}\otimes 1\otimes a_{\sigma(1)}\otimes \cdots \otimes
 a_{\sigma(i-1)}) \cdot\\
 &\hspace{50mm} \cdot \tr \big( M_{\sigma(1)}\cdots M_{\sigma(2k+1)} \big) = 0.
\end{split}
\]
One concludes that $\Theta_{V,N,2k}(1)$ is a closed $2k$-cocycle in
$C^{2k} \big( \gl_N(\fWeyl_{2n}^V);\C((\hbar)) \big)$.
Since  $\tau_{2k}$ is a normalized cochain, one can easily check that
$\Theta_{V,N,2k}$ is in fact a cocycle relative to the Lie subalgebra
$\gl_N\oplus\gl_V$ of $\gl_N(\fWeyl^V_{2n})$. The fact that
$\Theta_{V,N,2k}$ is a cocycle relative to $\mathfrak{sp}_{2n}$ is a
corollary of Proposition \ref{prop:relative-tau}. Thus the claim is proven.
\end{proof}

\subsection{Local Riemann-Roch theorem}
In this subsection, we use Chern-Weil theory to compute
the Lie algebra cocycle $\Theta_{V,N, 2k}$, using the strategy in the proof of
\cite[Thm.~5.1]{ffs}.

We start with recalling the construction of the Chern-Weil
homomorphism. Let $\g$ be a Lie algebra and $\h$ a Lie subalgebra
with an $\h$-invariant projection $\pr:\g\to \h$. The curvature
$C\in \Hom(\wedge^2 \g, \h)$ of $\pr$  is defined by
\[
  C(u\wedge v):=[\pr(u),\pr(v)]-\pr([u,v]).
\]
Let $(S^\bullet\h^*)^\h$ be the algebra of $\h$-invariant polynomials on $\h$
graded by polynomial degree. Define the homomorphism $\rho:(S^\bullet\h^*)^\h\to
C^{2\bullet}(\g,\h)$ by
\begin{displaymath}
  \rho (P)(v_1\wedge\cdots \wedge v_{2q})=
  \frac{1}{q!}\!\!\!\!\sum_{\sigma\in S_2q,\atop \
  \sigma(2i-1)<\sigma(2i)}\!\!\!\!(-1)^{\sigma}P\big( C(v_{\sigma(1)},\
  v_{\sigma(2)}), \cdots, C(v_{\sigma(2q-1)},\ v_{\sigma(2q)})\big).
\end{displaymath}
The right hand side of this equation defines a cocycle, and the
induced map in cohomology $\rho :(S^\bullet\h^*)^\h\to
H^{2\bullet}(\g,\h)$ is independent of the choice of the projection
$\pr$. This is the Chern--Weil homomorphism.

In our case, we consider $\g=\gl_N(\fWeyl_{2n}^V)$ and
$\h=\gl_N\oplus\gl_V\oplus\spin_{2n}$. The projection $\pr:\g\rightarrow\h$ is
defined by
\begin{displaymath}
\begin{split}
  \pr \, & (M_1\otimes M_2\otimes a) :=\\
   &:=\frac{1}{N}a_0\tr(M_2)M_1
  +\frac{1}{\dim(V)}a_0\tr(M_1)M_2+\frac{1}{N\dim(V)}\tr(M_1\otimes M_2)a_2,
\end{split}
\end{displaymath}
 where $a_j$ is the component of $a$ homogeneous of degree $j$ in $y$, $M_1\in \gl_N$, and $M_2\in \gl_V$.
The essential point about the Chern-Weil homomorphism in this case
is contained in the following result.
\begin{proposition}
\label{prop:chern-weil}
  For $N\gg n$ and $q\leq 2k$, the Chern-Weil homomorphism
  \begin{displaymath}
    \rho:(S^q\h^*)^{\h}\to H^{2q}(\g,\h)
  \end{displaymath}
  is an isomorphism.
\end{proposition}
\begin{proof}
  The proof of this result goes along the same lines as the proof of
  Proposition 4.2 in \cite{ffs}.
\end{proof}
Recall the following invariant polynomials on the Lie algebras $\gl_N$ and $\spin_{2n}$:
First on $\gl_N$ we have the Chern character
\[
\Ch(X):=\tr\left(\exp X\right),~\mbox{for}~X\in\gl_N.
\]
On $\spin_{2n}$, we have the $\hat{A}$-genus:
\[
\hat{A}(Y):=\det\left(\frac{Y/2}{\sinh(Y/2)}\right)^{1/2},~\mbox{for}~Y\in\spin_{2n}.
\]
We will need the rescaled version $\hat{A}_\hbar(Y):=\hat{A}(\hbar
Y)$. With this, we can now state:
\begin{theorem}\label{thm:local-rr}
In $H^{2k} \big( \gl_N(\fWeyl_{2n}^V),
\gl_{N}\oplus\gl_V\oplus\spin_{2n} \big)$ we have the identity
\[
  [\Theta_{V, N,2k}]=\rho((\hat{A}_\hbar\Ch_V\Ch)_k)
\]
for $k\leq n$ and $N\gg 0$.
\end{theorem}
\begin{proof}When $k=n$ the equality is proved in \cite[Thm.~5.1]{ffs}. Actually, one
can literally repeat the constructions and arguments in the proof of \cite[Thm.~5.1]{ffs}
for all $k\leq n$. We remark that we have different sign convention with respect to \cite[Thm. 5.1]{ffs} due to the change of sign in the cocycle $\tau_{2n}$, cf.\ Remark \ref{rmk:sign}.
\end{proof}
\subsection{Higher algebraic index theorem}
In this section, we use Theorem \ref{thm:local-rr} to compute the
pairing $\left\langle \sfQ(\alpha), P_1-P_2\right\rangle$.

\begin{theorem}\label{thm:higher-algind}
For a sequence of closed forms $\alpha=(\alpha_0, \cdots, \alpha_{2k})\in
\Tot^{2k} \calB \Omega^\bullet (M)((\hbar))$
and two projectors $P_1, P_2$ in $\calA^{((\hbar))}$ with $P_1-P_2$
compactly supported, one has
\[
  \left\langle \sfQ(\alpha), P_1-P_2\right\rangle=
  \sum_{l=0}^{k}\frac{1}{(2\pi \sqrt{-1})^{l}}\int_M \alpha_{2l}
  \wedge \hat{A}(M)\Ch(V_1-V_2) \exp \Big( -\frac{\Omega}{2\pi \sqrt{-1}\hbar}\Big),
\]
where $V_1$ and $V_2$ are vector bundles on $M$ determined by the
zero-th order terms of $P_1$ and $P_2$.
\end{theorem}
\begin{proof}
According to Proposition \ref{prop:reduction}, $\langle \sfQ(\alpha),
P_1-P_2\rangle = \langle \sfQ_{V_1}(\alpha),1\rangle-
\langle \sfQ_{V_2}(\alpha),1\rangle$.
Furthermore, by the arguments at the beginning of Section \ref{sec:lie-alg},
$\langle \sfQ_{V_i}(\alpha), 1\rangle$, $i=1,2$ is given  by
\begin{equation}\label{eq:pairing}
\sum_{l\leq k}\frac{1}{(2\pi \sqrt{-1})^n}\int_M\alpha_{2l}\wedge
\Psi^{2n-2l}_{V, 2n-2l}(1).
\end{equation}
Moreover, recall that $\Psi^{2n-2l}_{V,2n-2l}(1)$ is equal to
$\frac{1}{(2n-2l)!\hbar^{n-l}}\Theta_{V,2n-2l} (A\wedge\cdots \wedge
A)(1)$. Note that the direct sum with trivial bundles does not
change the value of the pairing. Therefore, we can add a large
enough trivial bundle to both $V_1$ and $V_2$ so that we can apply
Theorem \ref{thm:local-rr} to compute $\Theta_{V, 2n-2l}$.
For vector fields $\xi_1, \cdots, \xi_{2n-2l}$ on $M$ we have
\[
\begin{split}
\Theta_{V,N,2n-2l}(1) & (A\wedge\cdots A)(\xi_1, \cdots, \xi_{2n-2l}) = \\
  = \, & (2n-2l)!\, \rho((\hat{A}_\hbar \Ch_V\Ch)_{2n-2l})(A(\xi_1)
  \wedge \cdots \wedge A(\xi_{2n-2l}))\\
  = \, &\frac{(2n-2l)!}{(n-l)!}\sum_{\sigma(2j-1)<\sigma(2j)}\sgn(\sigma)\\
  & P_{2n-2l}\big( C(A(\xi_{\sigma(1)}), A(\xi_{\sigma(2)})), \cdots,
  C(A(\xi_{\sigma(2n-2l-1)}), A(\xi_{\sigma(2n-2l)}))\big),
\end{split}
\]
where $P_{2n-2l}=(\hat{A}_\hbar\Ch_V\Ch)_{n-l}\in
{(S^{n-l}\h)^*}^\h$. By \cite[Thm.~5.2]{ffs}, for two any vector
fields $\xi, \eta$ on $M$, $C(A(\xi), A(\eta))$ is equal to
$\tilde{R}_V(\xi, \eta)+\tilde{R}(\xi, \eta)-\Omega(\xi,\eta)$,
where $\tilde{R}$ (and $\tilde{R}_V$) is the lifting of the
curvature of the bundle $TM$ (and $V$) and $\Omega$ is the curvature
for the Fedosov connection. Therefore, we have
\[
\begin{split}
 \Theta_{V, N,2n-2l}(1)(A\wedge\cdots A)&(\xi_1, \cdots,
 \xi_{2n-2l})\\
 = & (2n-2l)!\, \rho(P_{2n-2l})((R_V+R-\Omega)^{2n-2l}).
\end{split}
\]
Replacing $\Psi_{V,2n-2l}^{2n-2l}(1)$ by $
\frac{1}{\hbar^{n-l}(2n-2l)!}\Theta_{V,N,2k}$ in Equation
\ref{eq:pairing}, we obtain
\[
\begin{split}
  &\left\langle \sfQ_{V_1}(\alpha)-\sfQ_{V_2}(\alpha), 1\right\rangle\\
  =&
  \sum_{l\leq k}\frac{1}{(2\pi \sqrt{-1})^l}\int_M\alpha_{2l}\wedge
  \hat{A}(M)\Ch(V_1-V_2)\exp\Big(-\frac{\Omega}{2\pi \sqrt{-1}\hbar} \Big) .
\end{split}
\]
This completes the proof.
\end{proof}

%%% Local Variables:
%%% mode: latex
%%% TeX-master: "HigherIndex"
%%% End:

%% file: GenOrbifolds.tex
%
%
\section{Generalization to orbifolds}
\label{Sec:GenOrbi}
In this section we show how the previous constructions can be generalized to
orbifolds. The result is an algebraic index theorem for $(b+B)$-cocycles on
certain formal deformations of proper \'etale groupoids, which in turns
generalizes the index formula for traces in \cite{ppt}.
\subsection{Preliminaries}
Let $(M,\omega)$ be a symplectic orbifold, i.e., a paracompact Hausdorff space
locally modeled on a quotient of an open subset of $\R^{2n}$, equipped with
the standard symplectic form, by a finite subgroup
$\Gamma\subset \operatorname{Sp}(2n,\R)$. As an abstract notion of an atlas,
we fix  a proper \'etale groupoid $\sfG_1\rightrightarrows \sfG_0$ with the
property that $\sfG_0/\sfG_1\cong M$, and $\sfG_0$ is equipped with a
$\sfG$-invariant symplectic form $\omega$. Denoting the two structure
maps of the groupoid by $s,t:\sfG_1\rightarrow \sfG_0$,
this means that $s^*\omega=t^*\omega$. Remark that for any symplectic
orbifold, such a groupoid always exists and is unique up to Morita
equivalence. Associated to the groupoid $\sfG$ is its convolution algebra
$\calA \rtimes \sfG:=\calC^\infty_\text{\tiny\rm cpt}(\sfG_1)$ with product
given by convolution:
\[
  (f_1*f_2)(g):=\sum_{g_1 g_2=g}f_1(g_1)f_2(g_2), \quad
  \text{where $f_1,f_2 \in \calC^\infty_\text{\tiny\rm cpt}(\sfG_1)$
  and $g\in \sfG_1$}.
\]
The symplectic structure on $\sfG$ equips $\calA \rtimes \sfG$ with
a noncommutative Poisson structure, that is, a degree $2$ Hochschild
cocycle whose Gerstenhaber bracket with itself is a coboundary. Let
$\calA^\hbar$ be a $\sfG$-invariant deformation quantization of
$(\sfG_0,\omega)$, for example given by Fedosov's method, using an
invariant connection as is explained in \cite{fe:g-index}. This
means that $\calA^\hbar$ forms a $\sfG$-sheaf of algebras over
$\sfG_0$, and we can take the crossed product $\calA^\hbar\rtimes
\sfG:= \Gamma_\text{\tiny\rm cpt} \big( \sfG_1, s^{-1}\calA^\hbar
\big)$ with algebra structure
\[
  [a_1\star_ca_2]_g=\sum_{g_1g_2=g}([a_1]_{g_1}g_2)[a_2]_{g_2},
  \quad \text{for $a_1,a_2\in s^{-1}\calA^\hbar (\sfG_1)$ and $g\in \sfG_1$}.
\]
This is a noncommutative algebra deforming the convolution algebra of the
underlying groupoid.

In \cite{nppt}, the  cyclic cohomology of $\calA^{((\hbar))} \rtimes
\sfG$ was computed to be given by
\begin{equation}
\label{iso-cyclic}
  HC^\bullet(\calA^{((\hbar))} \rtimes \sfG)=
  \bigoplus_{r\geq 0}H^{\bullet-2r}\big( \widetilde{M},\C((\hbar)) \big),
\end{equation}
where $\widetilde{M}$ is the so-called inertia orbifold which we will now
describe. Introduce the ``space of loops'' $B^{(0)}$ given by
\[
  B^{(0)}:=\{g\in \sfG_1  \mid s(g)=t(g) \}.
\]
In the sequel, we denote by $\sigma_0$ the local embedding obtained as the
composition of the canonical embedding $B^{(0)}\hookrightarrow \sfG_1$ with the
source map $s$. If no confusion can arise, we also denote the embedding
$B^{(0)}\hookrightarrow \sfG_1$ by $\sigma_0$.
The groupoid $\sfG$ acts on $B^{(0)}$ and the associate action groupoid
$\Lambda \sfG:=B^{(0)}\rtimes \sfG$ turns out to be proper and \'etale as well.
It therefore models another orbifold $\widetilde{M}:=B^{(0)}/\sfG$ called the
inertia orbifold.

As done in the previous sections for smooth manifolds, we will lift
the isomorphism \eqref{iso-cyclic} to a morphism of cochain
complexes where on one side we have a complex of differential forms
and on the other side Connes' $(b,B)$-complex. There are two natural
choices for a de Rham-type of complex that computes the cohomology
of $\widetilde{M}$. One is to use a simplicial resolution of $B^{(0)}$ given
by the so-called ``higher Burghelea spaces'' $B_k$ and the
associated simplicial de Rham complex. The other one, and this is
the complex we use, is to use $\sfG$-invariant differential forms on
$B^{(0)}$. The fact that the two models compute the same cohomology is
true because $\sfG$ is a proper groupoid.

It was observed in \cite{crainic} that $\Lambda \sfG$ is a so-called cyclic
groupoid, that is, comes equipped with a canonical nontrivial section
$\theta:B^{(0)}\rightarrow \Lambda \sfG_1$ of both source and target map. In this
case $\theta$ is given by $\theta(g)=g$, $g\in B^{(0)}$. As a consequence of this,
when we pull-back the sheaf $\calA^\hbar$ to $B^{(0)}$, it comes equipped with a
canonical section
\[
  \theta\in\underline{\Aut}(\iota^{-1}(\calA^\hbar_{G})).
\]
As we have seen, for a smooth symplectic manifold, the local model for a deformation
quantization was given by the Weyl algebra. In this case, it is given by the Weyl
algebra together with an automorphism.

\subsection{Adding an automorphism to the Weyl algebra}
As remarked in \S \ref{weyl_algebra}, the symplectic group
$\operatorname{Sp}(2n,\R)$ acts on the Weyl algebra $\Weyl_{2n}$ by
automorphisms. Let us fix an element $\gamma\in \operatorname{Sp} (2n,\R)$ of
finite order. It induces a decomposition of $V=\R^{2n}$ into two components,
$V=V^\perp\oplus V^\gamma$, where $V^\gamma$ is the subspace of fixed points.
Since $\gamma$ is a linear symplectic transformation, this decomposition is
symplectic, and we put $l := \dim(V^\perp)/2$.
Adding the automorphism $\gamma$ to the definition of cyclic cohomology has quite
an effect in the sense that we now have
\begin{proposition}
The twisted cyclic cohomology of the Weyl algebra is given by
\[
  HC^k_\gamma \big(\Weyl_{2n} \big)=
 \begin{cases}
   \C[\hbar,\hbar^{-1}],&\text{if $k=2n-2l+2p$},p\geq 0\\
   0,&\text{else}.
 \end{cases}
\]
\end{proposition}
We will now give an explicit generator for the nonzero class in cyclic
cohomology. Let $A$ and $\tilde A$ be algebras over a field $\Bbbk$, possibly
equipped with automorphisms $\gamma \in\Aut (A)$ and
$\tilde\gamma \in \Aut (\tilde A)$. The Alexander--Whitney map defines a cochain map
\[
  \#: C^\bullet(A)\otimes C^\bullet(\tilde A)\rightarrow C^\bullet(A\otimes \tilde A),
\]
where the Hochschild differentials are twisted by resp. $\gamma$, $\tilde\gamma$ and
$\gamma\otimes \tilde\gamma$. According to the Eilenberg--Zilber theorem, this is in
fact a quasi-isomorphism.  The cyclic version of this theorem,
cf.~\cite[\S 4.3]{loday}, states that the map above can be completed to a
quasi-isomorphism of the cochain complexes $\Tot^\bullet\calB C^\bullet$.
Below we will only be interested in the case where one of the two cochains
is of degree $0$, that means a twisted trace. Recall that a $\tilde\gamma$-twisted
trace on $\tilde A$ is a linear functional
$\tr_{\tilde \gamma} :\tilde A\rightarrow \Bbbk$ satisfying
\begin{equation}
\label{twisted_trace}
 \tr_{\tilde \gamma}(\tilde a_1 \tilde a_2)=\tr_{\tilde \gamma}(\tilde\gamma(\tilde a_2)
 \tilde a_1)\quad \text{ for all $\tilde a_1, \tilde a_2 \in \tilde A$}.
\end{equation}
\begin{lemma}
Let $\psi=(\psi_0,\ldots,\psi_{2k})\in \Tot^{2k}\calB C^\bullet(A)$ be a
$\gamma$-twisted $(b+B)$-cocycle on $A$ and $\tr$ a $\tilde \gamma$-twisted trace on
$\tilde A$. Then the cochain
\[
  \psi\# \tr=(\psi_0\#\tr ,\ldots,\psi_{2k}\# \tr)
\]
is a $\gamma\otimes\tilde\gamma$-twisted cocycle of degree $2k$ in
$\Tot^\bullet\calB C^\bullet(A\otimes \tilde A)$.
 \end{lemma}
\begin{proof}
Explicit computation.
\end{proof}
In our case, we have $\Weyl_{2n}=\Weyl_{2l}\otimes \Weyl_{2n-2l}$ according to the
decomposition $V=V^\perp\oplus V^\gamma$ of the underlying symplectic vector space.
Notice that by definition, the automorphism $\gamma\in \operatorname{Sp}_{2n}$
is trivial on $\Weyl_{2n-2l}$. Therefore we can simply use the cyclic cocycle
$(\tau_0,\ldots,\tau_{2n-2l})$ of degree $2n-2l$ on this part of the tensor product.
On the transversal part, i.e., associated to $V^\perp=\R^{2l}$ we use the
twisted trace $\tr_\gamma:\Weyl_{2l}  \rightarrow\C[\hbar,\hbar^{-1}]$
constructed by Fedosov in \cite{fe:g-index}: For this, we choose a
$\gamma$-invariant complex structure on $V^\perp$, identifying
$V^\perp\cong\C^{l}$ so that $\gamma\in \operatorname{U}(l)$. The inverse
Caley transform
\begin{displaymath}
  c(\gamma)=\frac{1-\gamma}{1+\gamma}
\end{displaymath}
is an anti-hermitian matrix, i.e., $c(\gamma)^*=-c(\gamma)$. With this,
define
\begin{displaymath}
  \tr_\gamma(a):=\mu_{2l}\left({\rm det}^{-1}(1-\gamma^{-1})
  \exp\left(\hbar \, c(\gamma^{-1})^{ij}\frac{\partial}{\partial z^i}
  \frac{\partial}{\partial \bar{z}^j}\right)a\right),
\end{displaymath}
where $c(\gamma^{-1})^{ij}$ is the inverse matrix of $c(\gamma^{-1})$ and where
we sum over the repeated indices $i,j=1,\ldots,l$. It is proved in
\cite[Thm.~1.1]{fe:g-index}, that this
functional is a $\gamma$-twisted trace density, i.e., satisfies equation
\eqref{twisted_trace}. Clearly, $\tr_\gamma(1)={\rm det}^{-1}(1-\gamma^{-1})$,
so the cohomology class of $\tr_\gamma$ is independent
of the chosen polarization. With this we have:
\begin{proposition}
Let $\gamma\in \operatorname{Sp}(2n,\R)$. Then the $\#$-product
\[
  (\tau_0\#\tr_\gamma,\ldots,\tau_{2n-2l}\#\tr_\gamma)\in
  \Tot^{2n-2l}\calB \normC^\bullet(\Weyl_{2n})
\]
defines a nontrivial $\gamma$-twisted cocycle of degree $2n-2l$ on the Weyl algebra.
\end{proposition}
\subsection{Cyclic cocycles on formal deformations of proper \'etale groupoids}
\label{sec:cyclic-gpd}
In this section we will show how to use the twisted $(b+B)$-cocycle of the previous
section to construct arbitrary $(b+B)$-cocycles on formal deformations of proper
\'etale groupoids. Consider again the Burghelea space $B^{(0)}$. Generically, this space
will not be connected, and has components of different dimensions. Introduce the
locally constant function $\ell:B^{(0)}\rightarrow\mathbb{N}$ by putting $\ell(g)$ equal
to half the codimension of the fixed point set of $g$ in a local orbifold chart.
\begin{definition}\label{def:psi}
Define $\Psi^{i}_{2k}\in\Omega^i(B^{(0)}) \otimes_{\calC^\infty (B^{(0)})}
\big( (\sigma_0^* \calW)^{\otimes (2k-2\ell-i+1)}\big)^* $ by
\[
\begin{split}
  \Psi^i_{2k}\, & \big( a_0\otimes \ldots \otimes a_{2k-2\ell -i}  \big):= \\
  & := \left(\frac{1}{\hbar}\right)^i\tau_{2k-2\ell}^\theta
  \big( (a_0 \otimes \ldots \otimes a_{2k-2\ell -i}) \times
  (\sigma_0^* A )_i \big).
\end{split}
\]
Hereby, $\calW$ is the Weyl algebra bundle on $\sfG_0$ for the $\sfG$-invariant
Fedosov deformation quantization $\calA^{((\hbar))}$, $A$ is the corresponding
connection 1-form on $\sfG_0$, and $a_0 , \ldots , a_{2k-2\ell -i}$ are germs
of smooth sections of $\sigma_0^* \calW$ at a point $g\in B^{(0)}$.
Notice that as a cochain on $\sigma_0^* \calW$, the degree of $\Psi^{i}_{2k}$
varies over the connected components of $B^{(0)}$ according to the function $\ell$
introduced above.
\end{definition}
\begin{proposition}
The $\Psi^i_{2k}$ are $\sfG$-equivariant and satisfy the equalities
\[
\begin{split}
  (-1)^id\Psi^{i-1}_{2k}&=\Psi^{i}_{2k}\circ b_\theta+\Psi^{i}_{2k+2}\circ
  B_\theta.
\end{split}
\]
\end{proposition}
\begin{proof}
Since the Fedosov connection on $\sfG_0$ is assumed to be
$\sfG$-invariant, $\Psi^i_{2k}$ is easily checked to be
$\sfG$-equivariant. We observe that
$b_\theta(\sigma_0^*A)_k=b(\sigma_0^*A)_k$ and
$B_\theta(\sigma_0^*A)_k=B(\sigma_0^*A)_k$ on $\sfG_0$. The proof of
the equality follows the same lines as the proof of its untwisted
version Proposition \ref{prop:psi}.
\end{proof}
\begin{remark}
\label{rem:trdens}
In particular, for $g\in B^{(0)}$, $i=2n-2\ell (g)$ and $k=n$, we find that
over each connected neighborhood of $g\in B^{(0)}$
\[
  d\Psi^{2n-2\ell(g) -1}_{2n}=\Psi_{2n}^{2n-2\ell (g)}\circ b_\theta.
\]
Thus the form $\Psi_{2n}^{2n - 2\ell}$ is a ``twisted trace
density'' in the notation of \cite[Def.~2.1]{ppt}. In fact
unraveling the definitions, the identity above is exactly
\cite[Prop.~4.2]{ppt}.
\end{remark}
\begin{definition}\label{def:chi}
For $2r\leq i$, define sheaf morphisms
\[
  \chi_i^{i-2r}:\Omega^i_{B^{(0)}}((\hbar))\rightarrow
  \normscrC^{i-2r}(\sigma_0^* \calA^{((\hbar))}),
\]
by the formula
\[
  \chi_{i,U}^{i-2r}(\alpha)(a_0,\ldots,a_{i-2r}):=
  \int_{B^{(0)}}\alpha\wedge\Psi^{2n-2\ell-i}_{2n-2r}
  (\sigma_0^{-1}a_0,\ldots, \sigma_0^{-1}a_{i-2r}),
\]
where $U\subset B^{(0)}$ is open, $\alpha \in \Omega^i (U)((\hbar))$,
and $a_0 , \ldots , a_{2k-2r} \in
\Gamma_\text{\tiny\rm cpt} \big( \sigma^* \calA^{((\hbar))}\big)$.
Together these morphisms define sheaf morphisms
\[
 \chi_i:\Omega^i_{B^{(0)}}((\hbar))\rightarrow \Tot^i\calB \normscrC^\bullet
 \big( \sigma_0^* \calA^{((\hbar))}\big), \quad
 \chi_i:=\sum_{2r\leq i}\chi^{i-2r}_i.
\]
\end{definition}
\noindent
By an argument similar to the untwisted case we obtain
\begin{theorem}
The morphism $\chi_\bullet$ is a morphism of sheaves of cochain complexes, i.e.,
\[
  (b+B)\chi_\bullet(\alpha)=\chi_\bullet(d\alpha),
\]
for all $\alpha\in\Omega^\bullet(U)$ and $U\subset B_\bullet$ open.
\end{theorem}
\noindent
With this we can now define an S-morphism of mixed $\sfG$-sheaf complexes
over the inertia orbifold $\tilde M$ as follows:
\[
  \scrQ^i:\Tot^i \calB \Omega^\bullet_{B^{(0)}} ((\hbar)) =
  \bigoplus_{2r\leq i}\Omega^{i-2r}_{B^{(0)}}((\hbar))\rightarrow
  \Tot^i \calB \normscrC^\bullet \big( \calA^{((\hbar))} \big)
\]
by
\[
  \scrQ^i \big( \sum_{2r\leq i}\alpha_{i-2r} \big):=
  \frac{1}{(2\pi \sqrt{-1})^\ell}\sum_{2r\leq i}\chi_{i-2r}(\alpha_{i-2r}).
\]
Forming global invariant sections we finally obtain the S-morphism
\[
  \sfQ : \Tot^i \calB \Omega^\bullet ({\tilde M})((\hbar)) =
  \bigoplus_{2r\leq i}\Omega^{i-2r}  ({\tilde M})((\hbar)) \rightarrow
  \Tot^i \calB \normC^\bullet \big( \calA^{((\hbar))}\rtimes \sfG \big) .
\]
\begin{proposition}\label{prop:quasi}
The map $\sfQ$ is an S-quasi-isomorphism establishing the
isomorphism \eqref{iso-cyclic}.
\end{proposition}
\subsection{Twisting by vector bundles}
It is our aim to compute the pairing of the cocycles in Connes' $(b+B)$-complex
obtained by the map $\sfQ$ above with $K$-theory classes on
$\calA^{((\hbar))} \rtimes \sfG$. Let us first explain how orbifold vector bundles
define elements in $K_0\big( \calA^{((\hbar))} \rtimes \sfG)$. Recall that an
orbifold vector bundle is  a vector bundle $V\rightarrow \sfG_0$ together with an
action of $\sfG$. Taking formal differences of isomorphism classes, these define
the orbifold $K$-group $K^0_\text{\tiny\rm orb}(M)$. An orbifold vector bundle defines a
projective $\calA \rtimes \sfG$-module $\Gamma_\text{\tiny\rm cpt}(\sfG_0,V)$, where
$f\in \calC^\infty_\text{\tiny\rm cpt}( \sfG_1)$ acts on
$\xi\in\Gamma_\text{\tiny\rm cpt}(\sfG_0,V)$ by
\begin{displaymath}
  (f\cdot \xi)(x)=\sum_{t(g)=x} f(g)\, \xi(s(g)), \quad \text{for $x\in \sfG_0$}.
\end{displaymath}
On the other hand, $K$-theory is stable under formal deformations, which means that
\[
  K_0(\calA^\hbar \rtimes \sfG)\cong K_0(\calA \rtimes \sfG),
\]
where the isomorphism is induced by taking the zero'th order term of a projector
in a matrix algebra over $\calA^\hbar \rtimes \sfG$. Altogether, we have defined
a map
\[
  K^0_\text{\tiny\rm orb}(M)\rightarrow K_0(\calA^\hbar \rtimes \sfG).
\]
It therefore makes sense to pair our cyclic cocycles with formal differences of
isomorphism classes of vector bundles. To compute this pairing we again use
quantization with values in the vector bundle to extend our cyclic cocycles.
For this, notice that when we pull-back an orbifold vector bundle
$V\rightarrow \sfG_0$ to $B^{(0)}$, the cyclic structure $\theta$ acts on $\sigma^*_0V$.
We therefore consider the algebra ${\VWeyl_{2n}}=\Weyl_{2n}\otimes \End (V)$
equipped with the automorphism $\gamma$ acting both via $\operatorname{Sp}(2n)$
on $\Weyl_{2n}$ and on the second factor by an element in $\End (V)$, also denoted
by $\gamma$. This leads to cochains
\[
  \tau_{2k}^{V,\gamma}\big( (a_0\otimes M_0) \otimes \ldots \otimes
  (a_{2k}\otimes M_{2k}) \big) := \tau^\gamma_{2k}
  (a_0 \otimes \ldots \otimes a_{2k})\tr_V(\gamma M_0\cdots M_{2k}).
\]
for $0\leq k\leq 2n-2l$. Together these cochains constitute a $\gamma$-twisted
$(b+B)$-cocycle
$(\tau^{V,\gamma}_0,\tau_2^{V,\gamma},\ldots, \tau^{V,\gamma}_{2n-2l})\in
\Tot^{2n-2l}\calB C^\bullet (\fWeyl_{2n}^V)$.
With this, one generalizes the definition of $\Psi^i_{2k}$, $\chi_\bullet$ and
$\scrQ^\bullet$ in the obvious manner to
$\Psi^i_{V,2k}$, $\chi_{V,i}$ and $Q_V^i$.
\begin{proposition}
  Let $\alpha=(\alpha_0,\ldots,\alpha_{2k}) \in \Tot^{2k}\calB \Omega^\bullet (\widetilde{M})$
  be a closed differential form, and $P_1$ and $P_2$ projection in the matrix algebras
  over $\calA^{((\hbar))} \rtimes \sfG$ with $P_1-P_2$ compactly supported on $\widetilde{M}$.
  Then we have
\[
\begin{split}
  \left\langle \sfQ(\alpha),P_{V_1}-P_{V_2}\right\rangle&=
  \left\langle \sfQ_{V_1}(\alpha)-\sfQ_{V_2}(\alpha), 1\right\rangle\\
  &= \sum_{i=0}^{k}\int_{\widetilde{M}}\frac{1}{(2\pi \sqrt{-1})^\ell m}
  \alpha_{2i}\wedge\left(\Psi^{2n-2\ell-2i}_{V_1,2n-2i}(1)-
  \Psi^{2n-2\ell-2i}_{V_2,2n-2i}(1)\right).
\end{split}
\]
Here the function $m:\widetilde{M}\to \mathbb{N}$ is the locally constant function which coincides for each sector $\calO\subset B^{(0)}$ with $m_\calO$, the order of the isotopy group  of the principal stratum of $\calO/\sfG \subset \widetilde{M}$.
\end{proposition}
\begin{proof}
The first equality is just as in Prop.~\ref{prop:reduction}. For the second,
again observe that the twisted cyclic cocycles are normalized, so we can throw
away all terms that contain more than one $1$. Finally, the reduction to
an integral over $\widetilde{M}$ is as in \cite[Prop.~4.4]{ppt}.
\end{proof}
\subsection{A twisted Riemann--Roch theorem}
By the previous proposition, it remains to evaluate $\Psi^{2n-2\ell-2i}_{V,2n-2i}(1)$,
which is of course done by interpreting it as a cocycle in Lie algebra cohomology.
Define the inclusion of Lie algebras $\h\subset\g$ by setting
\[
  \g:=\mathfrak{gl}_N \big( \fWeyl_{2n}^{V,\gamma} \big), \quad
  \h:=\mathfrak{gl}_N\oplus\mathfrak{gl}_V\oplus \mathfrak{sp}^\gamma_{2n},
\]
where the superscript $\gamma$ means taking $\gamma$-invariants. We will now construct
Lie algebra cocycles of $\g$ relative to $\h$  in $C^\bullet(\g;\h)$ as follows. First
the standard morphism from Hochschild cochains to Lie algebra cochains,
cf.~Eq.~\eqref{HoLa}, is still a morphism of cochain complexes when we twist the
differentials:
\[
  \phi^N:\big( C^\bullet(\frM_N( \fWeyl_{2n}),\frM_N( \fWeyl_{2n})^*),
  b_\gamma\big)\rightarrow
  \big( C^\bullet(\mathfrak{gl}_N( \fWeyl_{2n}),\frM_N(\fWeyl_{2n})^*),
  \partial_{\text{\tiny\rm Lie},\gamma}\big).
\]
Here the twisted Lie algebra cochain complex is as defined in \cite[\S 4.1.]{ppt}.
Second, evaluation at $1\in  \frM_n( \fWeyl_{2n})$ induces a morphism
\[
 \ev_1: \big( C^\bullet(\mathfrak{gl}_N(W^\gamma_{2n}), M_N(W_{2n})^*),
 \partial_{\text{\tiny\rm Lie},\gamma}\big)
 \rightarrow \big( C^\bullet(\g,\C((\hbar)),\partial_\text{\tiny\rm Lie}\big).
\]
Notice that this is only a morphism of cochain complexes when
restricted to the $\gamma$-invariant part of
$\mathfrak{gl}_N(\fWeyl_{2n}^V)$, because the evaluation morphism
above only respects the module structure of this sub-Lie algebra.
With this we now have:
\begin{proposition}
  For $k\leq n$ the cochain
  \[
    \Theta^{N,\gamma}_{V,2k}:= \frac{1}{\hbar^k}\ev_1\big( \phi^N(\tau_{2k}^\gamma)\big)\in
    C^{2k}\left(\g;\h,\C((\hbar))\right),
  \]
  is a Lie algebra cocycle relative to $\h$, which means
  $\partial_\text{\tiny\rm Lie}\Theta^{N,\gamma}_{V,2k}=0.$
\end{proposition}

\noindent
With this we have
\[
  \Psi^{2n-2\ell-2r}_{V,2n-2r}(1)=
  \left(\frac{1}{\hbar}\right)^{n-\ell-r}\frac{1}{(2n-2\ell-2r)!}
  \Theta^{N,\theta}_{V,2n-2\ell-2r}(A\wedge\ldots\wedge A)(1)
\]
To explicitly compute the class
$[\Theta^{N,\gamma}_{V,2k}]\in H^{2k}(\g;\h,\C((\hbar)))$,
we use the Chern--Weil homomorphism
\[
 \rho:\left(S^k\h^*\right)^\h \rightarrow H^{2k}(\g;\h,\C((\hbar))),
\]
which, by \cite[Prop.~5.1]{ppt}, is again an isomorphism for $k\leq n-l$ and
$N\gg n$ as in the untwisted case, cf.~Prop.~\ref{prop:chern-weil}.
Let us now describe the ingredients of the unique polynomial in
$\operatorname{S}^k\h^*$ that is defined by $\Theta^{N,\gamma}_{V,2k}$. For
this we split
\[
 \h=\mathfrak{sp}_{2n-2\ell}\oplus\mathfrak{sp}^\gamma_{2\ell}\oplus
 \mathfrak{gl}^\gamma_V\oplus\mathfrak{gl}_N,
\]
and write $X=(X_1,X_2,X_3,X_4)$ for an element in $\h$. Define
\[
  (\hat{A}_\hbar J_\gamma\Ch_{V,\gamma}\Ch)(X):=
  \hat{A}_\hbar(X_1)J_\gamma(X_2)\Ch_{V,\gamma}(\hbar X_3)\Ch(X_4),
\]
where $\Ch$ and $\hat{A}_\hbar$ are as before, $\Ch_{V,\gamma}$ is the Chern character
twisted by $\gamma$. Concretely, this means  $\Ch_{V,\gamma}(X_3)=\tr_V(\gamma\exp( X_3))$.
Finally, $J_\gamma$ is defined by
\[
  J_\gamma(X_2):=
  \sum_{i=0}^\infty\frac{1}{i!}\tr_\gamma(\underbrace{X_2\star\dots\star X_2}_{i}),
\]
where we use the embedding of $\mathfrak{sp}^\gamma_{2\ell}\subset
\mathfrak{sp}_{2n}$ as degree two polynomials in the Weyl algebra.
Strictly speaking, this is not an element of
$\operatorname{S}^\bullet(\h^*)^\h$, but we will only need a finite
number of terms in the expansion in the theorem below. In fact, in
the application to the higher index theorem, the specific element
$X_2$ turns out to be pro-nilpotent.
\begin{theorem}
 In $H^{2k}(\g;\h)$ we have the equality
 \[
  [\Theta^{N,\gamma}_{V,2k}]=
  \rho\left((\hat{A}_\hbar J_\gamma\Ch_{V,\gamma}\Ch)_k\right).
 \]
\end{theorem}
\begin{proof}
Given Theorem \ref{thm:local-rr}, this follows as in \cite[Thm.~5.3]{ppt}.
\end{proof}
\subsection{The higher index theorem for proper \'etale groupoids}
We finally arrive at our main result. To state it properly, we need
to introduce a few characteristic classes. Let $V$ be an orbifold
vector bundle. Using the cyclic structure $\theta$, we can twist the
Chern character of the pull-back $\sigma_0^{-1}V$ to define
$\Ch_\theta(\iota^{-1}V)$ by
\[
\Ch_\theta(\iota^{-1}V):=\tr\left(\theta\exp\left(\frac{R_V}{2\pi
\sqrt{-1}}\right)\right)\in H^{ev}(\widetilde{M})
\]
where $R_V$ denotes the curvature of a connection on $V$. Denote by $N$, the normal
bundle over $B^{(0)}$ coming from the embedding into $G_1$. It is easy
to see that the element
\[
  \Ch_\theta(\lambda_{-1}N):=
  \sum_{i=0}^{2\ell}(-1)^i\Ch_\theta(\Lambda^iN)\in H^{ev}(\widetilde{M}),
\]
is invertible. If we use $R^\perp$ to denote the curvature on $N$,
then
\[
\sum_{i=0}^{2\ell}(-1)^i\Ch_\theta(\Lambda^iN)=\det(1-\theta^{-1}\exp(-\frac{R^{\perp}}{2\pi
\sqrt{-1}})).
\]
With this observation, we can now state:
\begin{theorem}\label{thm:orb-alg-index}
Let $\alpha=(\alpha_{2k}, \cdots, \alpha_0)\in \Tot^{2k}\calB
\Omega^\bullet(\widetilde{M}) ((\hbar))$ be a sequence of closed forms
on the inertia orbifold, and $P_1, P_2$ be two projectors in the
matrix algebra over $\calA^{((\hbar))}$ with $P_1-P_2$ compactly
supported. Then we have
\[
  \left\langle \sfQ(\alpha), P_1-P_2\right\rangle=
  \sum_{j=0}^{k}\int_{\widetilde{M}} \frac{1}{(2\pi \sqrt{-1})^j m}\frac{\alpha_{2j} \wedge
  \hat{A}(\widetilde{M}) \, \Ch_\theta(\iota^*V_1-\iota^*V_2) \,
  \exp(-\frac{\iota^*\Omega}{2\pi \sqrt{-1}\hbar})}{\Ch_\theta(\lambda_{-1}N)},
\]
where $V_1$ and $V_2$ are the orbifold vector bundles on $M$ determined by the
zero-th order terms of $P_1$ and $P_2$, and $m$ is a local constant function defined by the order of the isotopy group of the principal stratum of a sector $\calO/\sfG\subset \widetilde{M}$.
\end{theorem}

%%% Local Variables:
%%% mode: latex
%%% TeX-master: "HigherIndex"
%%% End:

%% file: HAnaInThm.tex
%
%
\section{The Higher analytic index theorem on manifolds}
\label{sec:HAITM}
The higher algebraic index theorems proved in Section
\ref{Sec:AlgIndThm} gives us the means to derive Connes--Moscovici's
higher index theorem in a deformation theoretic framework. To this
end we first recall Alexander--Spanier cohomology which is needed to
define a higher analytic index for elliptic operators on manifolds
and then determine the cyclic Alexander--Spanier cohomology. An
$\hbar$-dependent symbol calculus for pseudodifferential operators
gives rise to a deformation quantization on the cotangent bundle.
This together with the computation of the cyclic Alexander--Spanier
cohomology enable us to relate the analytic with the algebraic
higher index. The higher algebraic index theorems can then be
derived from Thm.~\ref{thm:higher-algind}.
\subsection{Alexander--Spanier cohomology}
\label{subsec:ASc}
Assume to be given a smooth manifold $M$. Like in App.~\ref{subsec:loc}
denote by $\Bbbk$ one of the commutative rings $\R$, $\R[[\hbar]]$ and
$\R((\hbar))$, and let $\calO_{M,\Bbbk}$ be one of the sheaves
$\calC^\infty_M$, $\calC^\infty_M[[\hbar]]$ and  $\calC^\infty_M ((\hbar ))$,
respectively. In other words,
$\calO_{M,\Bbbk} (U) := \calC^\infty (U) \hat\otimes \Bbbk$ with $U\subset M$
open consists of all smooth functions on $U$ with values in $\Bbbk$.
If no confusion can arise, we shortly write $\calO$ instead of
$\calO_{M,\Bbbk}$. For $k\in \N$ denote by $\calO^{\hat\boxtimes k}$ the
completed exterior tensor product sheaf which is a sheaf on $M^k$ and which
is defined by the property
\[
  \calO^{\hat\boxtimes k} (U_1 \times \cdots \times U_k ) \cong
  \calO (U_1) \hat\otimes  \cdots \hat\otimes \calO (U_k)
  \quad \text{for all $U_1,\cdots , U_k \subset M$ open},
\]
where $\hat\otimes $ means the completed bornological tensor product. Put now
$\scrCAS^k(\calO) := \Delta^*_{k+1} \big( \calO^{\hat\boxtimes k+1}\big)$ and
define sheaf maps
 $\deltaAS: \scrCAS^{k-1} (\calO) \rightarrow \scrCAS^k (\calO)$
as follows. First observe that
\begin{displaymath}
  \scrCAS^k(\calO) (U) \cong \calO^{\hat\boxtimes k+1} \big( U^{k+1} \big) /
  \calJ \big( \Delta_{k+1} (U) , U^{k+1} \big) ,
\end{displaymath}
where $\calJ \big( \Delta_{k+1} (U) , U^{k+1} \big)$ denotes the ideal of
sections of $\calO^{\hat\boxtimes k+1}$ over $U^{k+1}$ which vanish on the
diagonal $\Delta_{k+1} (U)$. Then define
$\delta f \in \calO^{\hat\boxtimes k+1} \big( U^{k+1} \big)$ for
$f \in \calO^{\hat\boxtimes k} \big( U^k \big)$ by the formula
\begin{equation}
\nonumber
  \begin{split}
  & \delta f = \sum_{i=0}^k (-1)^i \, \delta^i f, \quad \text{ where} \\
  & \delta^i f (x_0, \ldots , x_{k+1}) =
  f (x_0, \ldots , x_{i-1}, x_{i+1},\ldots ,
  x_{k+1} ), \quad x_0,\ldots ,x_{k+1} \in U.
  \end{split}
\end{equation}
Additionally, put
\begin{displaymath}
  \delta' f = \sum_{i=0}^{k-1} (-1)^i \, \delta^i f .
\end{displaymath}
By construction, $\delta f $ and $\delta'f$ lie in
$\calJ \big( \Delta_{k+2} (U) , U^{k+2} \big)$,
if  $f \in \calJ \big( \Delta_{k+1} (U) , U^{k+1} \big)$. Hence one can pass
to the quotients and obtains maps
\[
  \deltaAS :\scrCAS^{k-1}(\calO) (U) \rightarrow \scrCAS^k(\calO) (U)
  \quad \text{and} \quad
  \deltaAS' :\scrCAS^{k-1} (\calO) (U) \rightarrow \scrCAS^k(\calO) (U)
\]
which are the components of sheaf maps.
Since $\delta^2 = (\delta')^2=0$, we have  two sheaf cochain complexes
$\big( \scrCAS^\bullet (\calO), \deltaAS \big)$ and
$\big( \scrCAS^\bullet (\calO), \deltaAS' \big)$. Denote by
$ \CAS^\bullet (\calO)  := \Gamma (M, \scrCAS^\bullet(\calO) )$ the complex of
global sections with differential given by $\deltaAS$.
This is the {\it Alexander--Spanier cochain complex}
of $\calO$. Its cohomology is denoted by $\HAS^\bullet (\calO) $ and called the
{\it Alexander--Spanier cohomology} of $\calO$.
In the particular case, where $\Bbbk =\R$ and $\calO =\calC^\infty_M$,
one recovers the {\it Alexander--Spanier cohomology}
$\HAS^\bullet (M)$ of $M$.
\begin{proposition}
\label{prop:resAS}
 Let $\iota : \underline{\Bbbk} \rightarrow \scrCAS^0(\calO)$ be the canonical
 embedding of the locally constant sheaf $\underline{\Bbbk}$ into
 $\scrCAS^0(\calO)$. Then
 \[
   \underline{\Bbbk} \overset{\iota}{\rightarrow} \scrCAS^0 (\calO)
   \overset{\deltaAS}{\rightarrow}\scrCAS^1(\calO)
   \overset{\deltaAS}{\rightarrow} \ldots \overset{\deltaAS}{\rightarrow}
   \scrCAS^k(\calO) \overset{\deltaAS}{\rightarrow}
 \]
 is a fine resolution of the locally constant sheaf
 $\underline{\Bbbk}$. Moreover, $\big( \scrCAS^\bullet(\calO) ,
 \deltaAS'\big)$ is contractible.
\end{proposition}
\begin{proof}
 Obviously, each of the sheaves $ \scrCAS^k (\calO)$ is fine. So it remains to
 show that for each $x\in M$ the sequence of stalks
 \begin{displaymath}
  0 \hookrightarrow \Bbbk \overset{\iota}{\rightarrow} \scrCAS^0 (\calO)_x
  \overset{\deltaAS}{\rightarrow} \ldots \overset{\deltaAS}{\rightarrow}
  \scrCAS^k (\calO)_x \overset{\deltaAS}{\rightarrow}  \ldots
 \end{displaymath}
 is exact. To this end note first that the stalk $\scrCAS^k (\calO)_x$ is given
 as the inductive limit of quotients
 $\calO^{\hat\boxtimes (k+1)} \big(U^{k+1} \big) /
 \calJ \big( \Delta_{k+1}  (U) , U^{k+1} \big)$, where $U$ runs through the open
 neighborhoods of $x$. Define now for
 $k\in \N$ so-called \textit{extra degeneracy maps}
 \begin{displaymath}
 \begin{split}
   s_x^k : \: & \calO^{\hat\boxtimes (k+2)} \big( U^{k+2} \big) \rightarrow
   \calO^{\hat\boxtimes (k+1)} \big(U^{k+1}\big) , \:
   f \mapsto f (x, - ),  \text{ and } \\
   s^{k+1,k} :\: & \calO^{\hat\boxtimes (k+2)} \big( U^{k+2} \big) \rightarrow
   \calO^{\hat\boxtimes (k+1)} \big(U^{k+1}\big) , \:
   f_0 \otimes \ldots \otimes f_{k+1} \mapsto
   f_1 \otimes \ldots \otimes f_k \otimes f_{k+1} f_0 .
 \end{split}
 \end{displaymath}
Additionally  put
\begin{displaymath}
   \varepsilon_x : \calO ( U ) \rightarrow
   \Bbbk , \: f \mapsto f (x ).
\end{displaymath}
Then one checks immediately that
\begin{equation}
  s_x^{k+1} \delta + \delta s_x^k  = \id \quad \text{for all $k \in \N$} \quad
  \text{and} \quad s_x^0 \delta + \iota \varepsilon   = \id .
\end{equation}
This proves the first claim. For the proof of the second it suffices
to verify that
\begin{equation}
\label{eq:exdegcomm}
  s^{k+1,k} \delta^0 = \id \quad \text{and} \quad
   s^{k+1,k} \delta^i = \delta^{i-1} s^{k,k-1} ,
\end{equation}
since then
\begin{displaymath}
   s^{k+1,k} \delta' + \delta' s^{k,k -1}= \id
   \quad \text{for all $k \in \N^*$ and} \quad s^{1,0} \delta' = \id.
\end{displaymath}
But Eq.~(\ref{eq:exdegcomm}) is obtained by straightforward
computation, and the proposition follows.
\end{proof}

\begin{remark}
\label{Rem:qicdr} By the preceding result the Alexander--Spanier cohomology has
to coincide both with the \v{C}ech cohomology of the locally constant sheaf
$\underline{\Bbbk}$ and the de Rham cohomology of $M$ with values in
$\Bbbk$ (cf.~\cite{sp:algtop,conmos}). Let us sketch the
construction of the corresponding quasi-isomorphisms. To this end
choose an open covering $\calU$ of $M$ and a subordinate smooth
partition of unity $(\varphi_U)_{U\in \calU}$. Consider a \v{C}ech
cochain $c = (c_{U_0, \ldots, U_k})_{(U_0, \ldots, U_k)\in \calN^k
(\calU)}$ with  values in the ring $\Bbbk$, where
\[
  \calN^k (\calU):= \{ (U_0 , \ldots , U_k) \in \calU^{k+1} \mid
  U_0 \cap  \ldots \cap U_k  \neq \emptyset\}
\]
is the nerve of the covering.
Associate to $c$ the Alexander--Spanier cochain
\begin{displaymath}
  \rho_\calU (c) (x_0, \ldots , x_k) = \sum_{U_0 \ldots U_k} \,
  c_{U_0 \ldots U_k} \,
  \varphi_{U_0} (x_0) \cdot \ldots \cdot \varphi_{U_k} (x_k) .
\end{displaymath}
One checks easily that the resulting map
$\rho_\calU : \check{C}^\bullet_\calU (M,\Bbbk) \rightarrow \CAS^\bullet (\calO)$
is a chain map. Moreover, if $\calU$ is a good covering, i.e.~if it is locally
finite and if the intersection of each finite family of elements of $\calU$
is contractible, then $\rho_\calU$ is even a quasi-isomorphism. To define a
quasi-isomorphism
$\lambdaAS : \CAS^\bullet (\calO) \rightarrow \Omega^\bullet (M,\Bbbk)$
first choose a complete riemannian metric on $M$, and denote by $\exp$ the
corresponding exponential function. For
$f\in \calO^{\hat\boxtimes (k+1)} (M^{k+1})$, $x\in M$ and
$v_1,\ldots , v_k\in T_xM$ then put
\begin{displaymath}
\begin{split}
  \lambda (f)_x (v_1, \ldots,&v_k) :=\\
  =&\frac{1}{k!} \sum_{\sigma \in S_k} \, \sgn (\sigma)
  \frac{\partial}{\partial s_1} \ldots \frac{\partial}{\partial s_k} f
  \big( x, \exp_x (s_1 v_{\sigma (1)}),\ldots ,\exp_x (s_k v_{\sigma (k)}) \big)|_{s_i=0}.
\end{split}
\end{displaymath}
Clearly, this defines a $\Bbbk$-valued smooth $k$-form $\lambda (f)$, which
vanishes, if one has $f \in \calJ\big( \Delta_{k+1} (M), M^{k+1} \big) $.
Moreover, one checks easily that $\lambda \delta (f ) = d \lambda (f)$.
By passing to the quotient $\CAS^k (M)=
\calO^{\hat\boxtimes (k+1)} (M^{k+1})/\calJ\big( \Delta_{k+1} (M), M^{k+1}\big)$
we thus obtain the desired chain map which is denoted by $\lambdaAS$.
By \cite{conmos}, $\lambdaAS$ is a quasi-isomorphism.
\end{remark}
\begin{remark}
\label{Rem:AShom}
  For later purposes let us present here another representation of Ale\-xan\-der-Spanier
  cochains in case $\calO$ is the sheaf of smooth functions on $M$.
  This representation allows also for a dualization, i.e.~the construction of
  Alexander-Spanier homology groups.
  To this end consider an open covering $\calU$ of $M$, and denote by
  $\calU^k$ the neighborhood $\bigcup_{U\in \calU} U^k$ of the diagonal $\Delta_k (M)$
  in $M^k$. Then put
  \begin{equation}
    \label{Eq:DefASCochains}
    C^k_{\rm\tiny AS} (M,\calU) := \calC^\infty \big( \calU^k \big)
  \end{equation}
  Obviously, $\big( C^\bullet_{\rm\tiny AS} (M,\calU), \delta \big)$ then forms
  a complex where, in degree $k$, $\delta$   denotes here the Alexander-Spanier
  differential restricted to $\calC^\infty\big( \calU^k \big)$.
  Moreover, for every refinement
  $\calV \hookrightarrow \calU$ of open coverings, one has a canonical
  chain map
  $ C^\bullet_{\rm\tiny AS} (M,\calU) \rightarrow C^\bullet_{\rm\tiny AS} (M,\calV)$.
  The direct limit of these chain complexes with respect to $\calU$
  running through the directed set $\operatorname{Cov} (M)$ of open coverings
  of $M$ coincides
  naturally with the Alexander-Spanier cochain complex over $M$:
  \begin{equation}
    \lim_{\longrightarrow \atop \calU \in \operatorname{Cov} (M)}
     C^\bullet_{\rm\tiny AS} (M,\calU) \cong
     \scrCAS^\bullet(\calC^\infty) (M) .
  \end{equation}
  Hence the direct limit of the cochain complexes
  $C^\bullet_{\rm\tiny AS} (M,\calU)$ computes
  the Ale\-xan\-der-Spanier cohomology of $M$. Note that since homology functors
  commute with direct limits, Alexander-Spanier cohomology also coincides naturally
  with the direct limit
  \[
   \lim_{\longrightarrow \atop \calU \in \operatorname{Cov} (M)}
   H^\bullet_{\rm\tiny AS} (M,\calU).
  \]

  Now let $C_k^{\rm\tiny AS} (M,\calU)$ be the topological dual of
  $C^k_{\rm\tiny AS} (M,\calU)$, i.e.~the space of compactly supported
  distributions on $M^{k+1}$. Transposing $\delta$ gives rise
  to a chain complex $\big( C_\bullet^{\rm\tiny AS} (M,\calU) , \delta^* \big) $,
  the homology of which is denoted by $H_\bullet^\text{\rm\tiny AS} (M,\calU)$.
  The inverse limit
  \begin{equation}
     H_\bullet^{\rm\tiny AS} (M) :=
     \lim_{\longleftarrow \atop \calU \in \operatorname{Cov} (M)}
     H_\bullet^{\rm\tiny AS} (M,\calU)
  \end{equation}
  is called the Ale\-xan\-der-Spanier homology of $M$.
  By \cite[Prop.~1.2]{moswu} one has for every open covering
  $\calU$ of $M$ a natural isomorphism between the Ale\-xan\-der-Spanier homology and
  \v{C}ech homology
  \begin{equation}
    \label{Eq:IsoAScech}
    H_\bullet^{\rm\tiny AS} (M,\calU) \cong \check{H}_\bullet (M,\calU).
  \end{equation}
  This implies in particular, that Alexander-Spanier homology coincides naturally with
  \v{C}ech homology. Moreover, for a good open cover $\calU$ of $M$, i.e~an open cover
  such that all finite nonempty intersections of elements of $\calU$ are contracible,
  the homology $H_\bullet^{\rm\tiny AS} (M,\calU) $ of the cover $\calU$
  then has to coindide with the Alexander-Spanier homology $H_\bullet^{\rm\tiny AS} (M)$
  of the total space (cf.~\cite[\S 15]{BotTu}).

  By duality of the defining
  complexes, Alexander-Spanier homology and cohomology pair naturally,
  which means that in each degree $k$ one has a natural map
  \begin{equation}
     \langle - , -\rangle :
     H_k^{\rm\tiny AS} (M) \times  H^k_{\rm\tiny AS} (M) \rightarrow \R .
  \end{equation}
  Let us describe this pairing in some more detail, since we will later need it.
  Let $[f]$ be an Alexander-Spanier cohomology class represented
  by some cochain $f \in C^k_{\rm\tiny AS} (M,\calU)$. Let
  $\mu = \big( [\mu_\calV] \big)_{\calV \in \operatorname{Cov} (M)}$ be an
  Alexander-Spanier homology class, where the $\mu_\calV $ are
  appropriate cycles in $C_k^{\rm\tiny AS} (M, \calV)$.
  Then, one puts
  \begin{equation}
    \label{Eq:DefASPair}
    \langle \mu , [f] \rangle := \mu_\calU (f) .
  \end{equation}
  It is straightforward to check that this definition of the pairing
  $\langle \mu , [f] \rangle$ does not depend on the choice of
  representatives for the homology classes $[\mu_\calV]$ respectively for the
  cohomology class $[f]$.
\end{remark}

Besides the above defined sheaf complex
$\big( \scrCAS^\bullet (\calO) , \deltaAS\big)$, one can
define the sheaf complex $\big( \scraCAS^\bullet (\calO), \deltaAS\big)$ of
antisymmetric Alexander--Spanier cochains and the  sheaf complex
$\big( \scrcCAS^\bullet (\calO) , \deltaAS\big)$ of cyclic Alexander--Spanier
cochains. A section of $\scrCAS^k (\calO)$ over $U\subset M$ open which is
represented by some
$f \in \calO^{\hat\boxtimes k+1} \big( U^{k+1} \big)$ is called
{\it antisymmetric} resp.~{\it cyclic}, if
\[
  f ( x_{\sigma (0)} , \ldots , x_{\sigma (k+1} ) =
  \sgn (\sigma) \, f(x_0 , \ldots , x_k )
\]
for all $(x_0 , \ldots , x_k )$ close to the diagonal and every permutation
resp.~every cyclic permutation $\sigma$ in $k+1$ variables. In the following
we show how to determine the cohomology of these sheaf complexes.
To this end we first define
\textit{degeneracy  maps} $s^{i,k}$ for $0\leq i \leq k$ as follows:
\begin{displaymath}
  \begin{split}
  s^{i,k} : \: & \calO^{\hat\boxtimes (k+2)} \big( U^{k+2} \big) \rightarrow
  \calO^{\hat\boxtimes (k+1)} \big(U^{k+1}\big) , \\
  & f_0 \otimes \ldots \otimes f_{k+1} \mapsto
    f_0 \otimes \ldots \otimes f_i \, f_{i+1} \otimes \ldots  \otimes f_{k+1} .
  \end{split}
\end{displaymath}
Obviously, these maps $s^{i,k}$ induce sheaf morphisms
$\sAS^{i,k}: \scrCAS^{k+1} (\calO)  \rightarrow  \scrCAS^k (\calO)$.
Moreover, one checks immediately that the following cosimplicial identities
are satisfied:
\begin{align}
 \deltaAS^j \deltaAS^i & = \deltaAS^i \deltaAS^{j-1}, \quad \text{if $i<j$}\\
 \sAS^{j,k-1} \sAS^{i,k} & = \sAS^{i,k-1} \sAS^{j+1,k}, \quad
 \text{if $i\leq j$} \\
 \sAS^{j,k} \deltaAS^i & =
 \begin{cases}
   \deltaAS^i \sAS^{j-1,k-1} & \text{for $i <j$},\\
   \id & \text{for $i=j$ or $i=j+1$},\\
   \deltaAS^{i-1} \sAS^{j,k-1} & \text{for $i >j+1$}.
 \end{cases}
\end{align}
Next we introduce the cyclic operators
\begin{equation}
  \label{eq:cycop}
  t_x^k : \scrCAS^k (\calO)_x \rightarrow \scrCAS^k (\calO)_x , \:
  [f_0 \otimes \ldots \otimes f_k]_x \mapsto (-1)^k
  [f_1 \otimes \ldots \otimes f_k \otimes f_0]_x
\end{equation}
Note that the cyclic operator $t^k_x$ is induced by a globally defined
sheaf morphism
$t^k : \scrCAS^k (\calO) \rightarrow \scrCAS^k (\calO)$. One easily checks
that the $t^k$ satisfies  the following cyclic identities:
\begin{align}
  t^k\deltaAS^i & = \deltaAS^{i-1} t^{k-1}, \quad \text{if $1\leq i\leq k$}\\
  t^k\sAS^{i,k} & = \sAS^{i-1,k} t^{k+1}, \quad \text{if $1\leq i\leq k$}\\
  \big( t^k\big)^{k+1} & = \id .
\end{align}
This means that the tuple $\big( \scrCAS^k (\calO) , \deltaAS^i,
\sAS^{i,k} , t^k\big)$ is a cyclic cosimplicial sheaf over $M$. Its
cyclic cohomology can be computed as the cohomology of either one of
the following complexes:
\begin{enumerate}
\item
  the total complex of the associated cyclic bicomplex
  with vertical differentials given by $\deltaAS$ in even degree
  resp.~by $- \deltaAS'$ in odd degree, and horizontal differentials given by
  $\id - t^k$ in even degree  resp.~by
  $N^k := \sum_{l=0}^k \big( t^k\big)^l$ in odd degree;
\item
  the complex obtained as the $0$-th cohomology of the horizontal
  differentials in the cyclic bicomplex; in other words
  this is the cyclic Alexander--Spanier complex $\scrcCAS^\bullet (\calO)$
  with differential $\deltaAS$;
\item
  the total complex of the associated mixed cochain complex
  with differentials $\deltaAS$ and $\BAS$, where
  $\BAS^k := N^k \, s^{0,k} \big( \id - t^{k-1} \big) $ with $s^{0,k}$
  the extra degeneracy defined above.
\end{enumerate}
By Proposition \ref{prop:resAS} the Hochschild cohomology of the mixed complex
(3) is given by $\Bbbk$ in degree $0$ and by $0$ in all other degrees.
Hence the cyclic cohomology of this mixed complex coincides with $\Bbbk$ in
even degree and with $0$ else. Since the cyclic cohomology is also computed by
$\scrcCAS^\bullet (\calO)_x$ one obtains the claim about the cyclic
Alexander--Spanier cohomology in the following result.
\begin{proposition}
\label{Prop:dercatisoAS}
  In the derived category of sheaves on $M$, both sheaf complexes
  $\scrCAS^\bullet (\calO)$ and $ \scraCAS^\bullet (\calO)$ are isomorphic to
  $\underline{\Bbbk}$, whereas $ \scrcCAS^\bullet (\calO)$
  is isomorphic to the cyclic sheaf complex
  \[
    \underline{\Bbbk} \rightarrow 0 \rightarrow \underline{\Bbbk}
    \rightarrow \ldots \rightarrow 0 \rightarrow \underline{\Bbbk}
    \rightarrow 0 \rightarrow \ldots \: .
  \]
  Moreover, the antisymmetrization
  $\varepsilon^\bullet : \scrCAS^\bullet (\calO)\rightarrow
  \scraCAS^\bullet (\calO) $,
  \[
    \varepsilon^k \big( [f_0 \otimes \ldots \otimes f_k ]\big) =
    \sum\limits_{\sigma\in S_{k+1}}\frac{\sgn{\sigma}}{(k+1)!}
    [f_{\sigma (0)} \otimes \ldots \otimes f_{\sigma (k)} ]
  \]
  is a quasi-isomorphism.
\end{proposition}
\begin{proof}
  By the previous considerations it remains only to prove that
  $\varepsilon^\bullet$ is a quasi-isomorphism. To this end one checks
  along the lines of the proof of Proposition \ref{prop:resAS} and
  by using the maps $s_x^k$ that $ \scraCAS^\bullet (\calO)$
  is a fine resolution of $\underline{\Bbbk}$, and that $\varepsilon^\bullet$
  is a sheaf morphism between these fine resolutions over the identity
  of $\underline{\Bbbk}$.
\end{proof}

\begin{remark}
Abstractly, a cyclic object in a category $\mathcal{C}$ is a
contravening functor from Connes' cyclic category $\Delta C$ to
$\mathcal{C}$, cf.\ \cite[\S6.1]{loday}. The cyclic category $\Delta
C$ has the remarkable property of being isomorphic to its opposite
$\Delta C^{op}$ via an explicit functor as in Prop. 6.1.11. in
\cite{loday}. Therefore, out of any cyclic object, one constructs a
\textit{co}cyclic object -that is, a covariant functor  $\Delta
C\to\mathcal{C}$- by precomposing with this isomorphism, called the
dual. With this, one recognizes the cocyclic sheaf $\scrCAS^\bullet
(\calO)$ as the dual of the cyclic sheaf $\calO^\natural_\bullet$
associated to $\calO$ as a sheaf of algebras.
\end{remark}

Next we construct a quasi-isomorphism from the sheaf complex $\big(
\scrcCAS^k (\calO) , \deltaAS \big)$ to the total complex of the
mixed sheaf complex $\big( \Omega^\bullet(-,\Bbbk), d , 0\big)$. To
this end define for $2r \leq k$ and $U\subset M $ open a morphism
\[
  \lambdaAS_{k,U}^{k-2r} : \Gamma \big(U , \scrCAS^k (\calO)\big)
  \rightarrow \Omega^{k-2r}(U,\Bbbk)
\]
as follows. First let $f \in  \calO^{\hat\boxtimes k+1} (U^{k+1})$ be a
representative of a section of $\scrCAS^k (\calO)$ over $U$, let $x \in U$ and
$v_1, \ldots , v_{k-2r} \in T_xM$. Then put
\begin{displaymath}
\begin{split}
  \lambda_{k,U}^{k-2r} & (f)_x (v_1, \ldots , v_{k-2r}) :=
  \frac{(k-2r)!}{(k+1)!} \sum_{\nu \in S_{2r+1,k-2r}}
  \sum_{\sigma \in S_{k-2r}} \, \sgn (\nu) \, \sgn (\sigma) \\
  & \frac{\partial}{\partial s_1} \ldots \frac{\partial}{\partial s_{k-2r}}
  (\nu f) \big( x, x, \ldots , x,
  \exp_x (s_1 v_{\sigma (1)}),\ldots ,\exp_x (s_{k-2r} v_{\sigma (k-2r)}) \big)|_{s_i=0}
\end{split}
\end{displaymath}
Hereby, $S_{p,q}$ denotes the set of $(p,q)$-shuffles of the set
$\{ 0, \ldots , p+q\}$, and $\nu f $ for $\nu \in S_{k+1}$ is defined by
\[
  \nu f (x_0, x_1, \ldots , x_k) :=
  f (x_{\nu (0)}, x_{\nu (1)}, \ldots , x_{\nu (k)}).
\]
Obviously, $\lambda_{k,U}^{k-2r} (f)$ vanishes, if $f$ vanishes
around the diagonal of $U^{k+1}$. Hence one can define
\[
  \lambdaAS_{k,U}^{k-2r} \big( f + \calJ (\Delta_{k+1}, U^{k+1} ) \big) :=
  \lambda_{k,U}^{k-2r} (f)
\]
which provides us with the desired morphism. By an immediate computation one
checks that for $f_0, \ldots , f_k \in \calO (U)$
\begin{equation}
\label{eq:RepSmorASdeRham}
\begin{split}
  \lambda_{k,U}^{k-2r} & \, (f_0 \otimes \ldots  \otimes f_k ) = \\
  = \, & \frac{(k-2r)!}{(k+1)!}\!
  \sum_{\nu \in S_{2r+1,k-2r}}\!\!\!\!
  \sgn (\nu) \, f_{\nu (0)} \cdot \ldots \cdot f_{\nu (2r)}
  \, df_{\nu (2r+1)} \wedge \ldots \wedge df_{\nu (k)} .
\end{split}
\end{equation}
\begin{proposition}
\label{prop:deflambda}
  Let
  $\lambdaAS_k:\scrcCAS^k (\calO) \rightarrow \Tot^k \calB
  \Omega^\bullet (- , \Bbbk)$
  be the sheaf morphism
  defined by $\lambdaAS_k := \sum_{2r\leq k} \lambdaAS_k^{k-2r}$.
  Then the following relation is satisfied:
  \[
    \lambdaAS_{k+1} \deltaAS  = d \lambdaAS_k .
  \]
\end{proposition}
\begin{proof}
First check that for $0\leq i \leq k$
\begin{equation}
\nonumber
\label{eq:CompLambDeli1}
  \begin{split}
    \lambda_{k,U}^{k-2r}&
    \,\big( \delta^i (f_0 \otimes \ldots \otimes f_{k-1})\big)=\\
    = \, & (-1)^i \, \lambda_{k,U}^{k-2r}
    \big( 1 \otimes f_0 \otimes \ldots \otimes f_{k-1}\big) \\
    = \, &(-1)^i \frac{(k-2r)!}{(k+1)!} \sum_{\nu \in S_{2r,k-2r}}
    \sgn (\nu)\, f_{\nu (0)} \otimes \ldots \otimes f_{\nu (2r-1)} \,
    df_{\nu (2r)} \wedge \ldots \wedge df_{\nu (k-1)} ,
  \end{split}
\end{equation}
and then that
\begin{equation}
\nonumber
\label{eq:CompLambDeli2}
  \begin{split}
    d &\, \lambda_{k-1,U}^{k-1-2r} (f_0 \otimes \ldots \otimes f_{k-1}) = \\
    & = \frac{(k-2r-1)!}{k!} \sum_{\nu \in S_{2r+1,k-2r-1}} \sgn (\nu) \,
    d ( f_{\nu (0)} \otimes \ldots \otimes f_{\nu (2r}) ) \wedge
    df_{\nu (2r+1)} \wedge \ldots \wedge df_{\nu (k-1)} \\
    & = \frac{(k-2r)!}{k!} \sum_{\nu \in S_{2r,k-2r}} \sgn (\nu) \,
    f_{\nu (0)} \otimes \ldots \otimes f_{\nu (2r-1)} \,
    df_{\nu (2r+1)} \wedge \ldots \wedge df_{\nu (k-1)} .
  \end{split}
\end{equation}
By the definition of $\delta$,  these two equations  entail the claimed
equality.
\end{proof}

\subsection{Higher indices}
Alexander--Spanier cohomology has been used by {\sc
Connes--Moscovici} \cite{conmos} to define higher (analytic) indices
of an elliptic operator acting on the space of smooth sections of a
(hermitian) vector bundle over a closed (riemannian) manifold. More
precisely, the Connes--Moscovici higher indices can be understood as
a pairing of the Chern character of a K-theory class defined by an
elliptic operator with the cyclic cohomology class defined by an
Alexander--Spanier cohomology class (cf.~\cite{moswu}). Unlike for
the K-theoretic formulation of the Atiyah--Singer index formula,
where the K-theory of the algebra of smooth sections over the
cosphere bundle of the underlying manifold is considered, it turns
out that for the K-theoretic formulation of higher index theorems
the appropriate algebra is the algebra of trace class operators
acting on the Hilbert space of square integrable sections of the
given vector bundle. This point of view and the fact that the
pseudo-differential calculus on the underlying manifold gives rise
to a deformation quantization enable us to compare the higher
analytic index with the higher algebraic index and then derive the
Connes--Moscovici higher index formula. In the following we provide
the details and proceed in several steps.

{\bf Step 1.} Assume that
$\Psi \in \Omega^{2n} (M) \otimes_{\calC^\infty (M)} \calW^*$ is a trace density
for the star product algebra $\calA^{((\hbar))}_\text{\tiny\rm cpt}$ on $M$.
In other words this means that
\[
  \Tr : \calA^{((\hbar))}_\text{\tiny\rm cpt}  \rightarrow \Bbbk, \quad
  a\mapsto \int_M \Psi (a)
\]
is a trace functional on
$\calA^{((\hbar))}_\text{\tiny\rm cpt}$. Then we define a chain map
\[
  \normscrX_{\Tr} : \scrCAS^\bullet \big( \calC^\infty_M((\hbar)) \big)
  \rightarrow \scrC^\bullet \big( \calA^{((\hbar))} \big)
\]
as follows. For $f_0, f_1, \ldots , f_k \in \calC^\infty (U)((\hbar))$ with
$U\subset M$ open and
$a_0 , \ldots , a_k \in \calA^{((\hbar))}_\text{\rm\tiny cpt} (U)$ put
\begin{equation}
\label{eq:charxdef1}
  \scrX_{\Tr} (f_0 \otimes f_1 \otimes \ldots \otimes f_k )\,
  (a_0 \otimes \ldots \otimes a_k ) :=
  \Tr_k \big(  f_0 \star a_0 , \ldots , f_k \star a_k \big) ,
\end{equation}
where
\begin{equation}
\label{eq:charxdef2}
  \Tr_k \big( a_0 , \ldots , a_k) := \Tr \big( a_0 \star \ldots \star a_k\big).
\end{equation}
Since the star product is local and the trace functional $\Tr$ is given as
an integral over the trace density, which also is local in its argument, one
concludes that the cochain
$\scrX_{\Tr} (f )$ vanishes, if $f\in \calC^\infty (U^{k+1})((\hbar))$
vanishes around the diagonal $\Delta_{k+1} (U)$. By passing to the quotient
we obtain the desired maps
$\normscrX_{\Tr} : \scrCAS^k \big( \calC^\infty_M((\hbar)) \big) (U)
  \rightarrow \scrC^k \big( \calA^{((\hbar))} \big) (U)$.
By straightforward computation one checks that
\[
  b \normscrX_{\Tr} = \normscrX_{\Tr} \deltaAS \quad
  \text{and} \quad B \normscrX_{\Tr} (f) = 0,
  \text{ if $f \in \scrcCAS^k \big( \calC^\infty_M((\hbar)) \big) (U)$}.
\]
Hence $\normscrX_{\Tr}$ provides a chain map from the
cyclic Alexander--Spanier complex to the cyclic complex of the deformed algebra.

\begin{remark}
\label{Rem:ExtScrX}
  Let  $\scrA$ be a sheaf of $\Bbbk$-algebras. Assume
  that on $\calO$ a local product denoted by $\cdot$ is defined, and
  that $\scrA$ carries an $\calO$-module structure.
  Finally let $\tau : \scrA (M) \rightarrow \Bbbk$ be a trace. Then
  Eqs.~\eqref{eq:charxdef1} and \eqref{eq:charxdef2} define
  a map
  \[
   \normscrX_\tau : \scrcCAS^k \big( \calO \big) \rightarrow
   C_\lambda^k \big( \scrA (M) \big).
  \]
  Later in this section we will make use of this observation.
\end{remark}

We now want to compare the morphism $\normscrX_{\Tr}$ with
$\scrQ \circ \lambdaAS$. To this end, let $\Psi$ denote the trace density
$\Psi_{2n}^{2n}$ defined in \ref{dfn:psi}. Note that by Proposition
\ref{prop:psi}, $\Psi_{2n}^{2n}$ is a trace density, indeed.
Furthermore, let $U\subset M$ be a contractible open Darboux domain.
By Theorem \ref{thm:quasi} one knows that
\[
  \scrQ_U^{2k} (1) =\Tr
\]
is a generator of the cyclic cohomology group
$H^{2k} \big( \Tot^\bullet \calB \normscrC^\bullet(\calA^{((\hbar))}) (U)\big)$
for $ k\geq 0 $, and that all other cyclic cohomology groups
$H^l \big( \Tot^\bullet \calB \normscrC^\bullet(\calA^{((\hbar))}) (U)\big)$,
with $l$ odd. Moreover, observe that for all $k\in \N$
\[
  \normscrX_{\Tr} (1^{2k+1} ) = \Tr_{2k}
  \quad \text{and} \quad
  \scrQ_U^{2k} \lambda (1^{2k+1} ) =  \scrQ_U^{2k} (1) =
  \Tr  .
\]
But since
\[
  \frac{1}{2k} (b+\normB ) \big( \Tr_1 , -\Tr_3, \ldots ,
  (-1)^{k-1} \Tr_{2k-1} \big)
  = \Tr_0 + (-1)^{k-1} \Tr_{2k} \quad \text{for $k>0$},
\]
both $\normscrX_{\Tr} (1^{2k+1} )$ and $\scrQ_U^{2k} \lambda (1^{2k+1} )$ are
generators of the cyclic cohomology groups
$H^{2k} \big( \Tot^\bullet \calB \normscrC^\bullet(\calA^{((\hbar))}) (U)\big)$
for $k\geq 0 $. Hence one concludes
\begin{proposition}
\label{prop:QISM-ASCyc}
  The sheaf morphisms $\normscrX_{\Tr} :
  \scrcCAS^\bullet \big( \calC^\infty_M((\hbar)) \big) \rightarrow
  \Tot^\bullet \calB \scrC^\bullet\big(\calA^{((\hbar))}\big)$
  and $\scrQ \circ \lambdaAS :
  \scrcCAS^\bullet \big( \calC^\infty_M((\hbar)) \big) \rightarrow
  \Tot^\bullet \calB \normscrC^\bullet\big(\calA^{((\hbar))}\big)
  \hookrightarrow \Tot^\bullet \calB \scrC^\bullet\big(\calA^{((\hbar))}\big)$
  coincide in the derived category of sheaves on $M$. In particular,
  $\normscrX_{\Tr}: \scrcCAS^\bullet \big(\calC^\infty_M((\hbar))\big) \rightarrow
  \Tot^\bullet \calB \scrC^\bullet\big(\calA^{((\hbar))}\big)$ is a quasi-isomorphism.
\end{proposition}

{\bf Step 2.}
Next we explain how a global symbol calculus for pseudodifferential operators
on a riemannian manifold $Q$ gives rise to a deformation quantization on the
cotangent bundle $T^*Q$.
Given an open subset $U\subset Q$ denote by $\Sym^m(U)$, $m\in \Z$, the
space of symbols of order $m$ on $U$, that means the space of
smooth functions $a$ on $T^*U$ such that in each local coordinate system
of $U$ and each compact set $K$ in the domain of the local coordinate
system there is an estimate of the form
\begin{displaymath}
  \big| \partial_x^{\alpha}\partial_{\xi}^{\beta} a (x,\xi)\big| \leq
  C_{K, \alpha, \beta}(1+|\xi|^2)^{\frac{m-|\beta|}{2}}, \quad
  \text{$x\in K$, $\xi \in T^*_x Q$, $\alpha,\beta \in \N^n$},
\end{displaymath}
for some $C_{K,\alpha,\beta} >0$. Moreover, put
\begin{displaymath}
 \Sym^\infty (U) := \bigcup_{m\in \Z}\Sym^m (U), \quad
 \Sym^{-\infty} (U) := \bigcap_{m\in \Z}\Sym^m (U).
\end{displaymath}
Obviously, the spaces $\Sym^m(U)$ with $m \in \Z \cup \{\pm \infty\}$
form the section spaces of a sheaf $\Sym^m$ on $Q$.
Similarly, one constructs the presheaves $\PDO^m$
of pseudodifferential operators of order
$m\in \Z \cup \{ \pm \infty \}$ on $Q$.
Next let us recall the definition of the symbol map
$\sigma $ and its quasi-inverse, the quantization map $\Op$.
The symbol map associates to every operator $A\in \PDO^m (U)$
a symbol $a \in \Sym^m (U)$ by setting
\begin{equation}
\label{def-symbol}
  a (x,\xi) : = A \big( \chi (\cdot, x)
  e^{i \langle \xi , \Exp_x^{-1} (\cdot )\rangle } \big) \, (x) ,
\end{equation}
where $\Exp_x^{-1}$ is the inverse map of the exponential map on $T_xQ$,
and
\begin{equation}
\label{cut-off}
  \chi : Q\times Q \rightarrow [0,1]
\end{equation}
is a smooth cut-off function such that $\chi = 1$ on a neighborhood
of the diagonal, $\chi (x,y) =\chi (y,x)$ for all $x,y \in Q$,
$\supp \chi (\cdot , x)$ is compact for each $x\in Q$, and finally
such that the restriction of $\Exp_x$ to an open neighborhood of
$\Exp_x^{-1} \big( \supp \chi (\cdot , x)\big)$ is a diffeomorphism
onto its image. The quantization map
\begin{equation}
\label{op-quant}
   \Op :\:  \Sym^m (U) \rightarrow \PDO^m (U) \subset \Hom \big(
  \mathcal{C}^{\infty}_\text{\tiny\rm cpt} (U),
  \mathcal{C}^{\infty}(U)\big),
  \end{equation}
  is then given by
\begin{equation}
\label{form_op-quant}
 \big( \Op (a) f\big) (x) :=
  \int_{T^*_x Q} \int_{Q} e^{-i \langle \xi , \Exp_x^{-1} (y)\rangle }
  \chi(x, y)  a (x , \xi) f(y) \, dy\,  d\xi , \quad
  f \in \calC^\infty_\text{\tiny\rm cpt} (U).
\end{equation}
The maps $\sigma$ and $\Op$ are now inverse to each other up to elements
$\PDO^{-\infty}$ respectively $\Sym^{-\infty}$.
Note that by definition of the operator map, the Schwartz kernel $K_{\Op (a)}$
of $\Op (a)$ is given by
\begin{equation}
\label{eq:schwkerop}
   K_{\Op (a)} (x,y) = \int_{T^*_xQ} \, e^{i \langle \xi , \exp^{-1}_x (y) \rangle}
   \chi (x,y) \, a (\xi) \, d\xi .
\end{equation}

By the space $\ASym^m (U)$, $m\in \Z$ of {\it asymptotic symbols} over an open
$U\subset Q$ one understands the space of all
$q \in \calC^\infty (T^*U \times [0 ,\infty))$ such that for each
$\hbar \in [0 , \infty)$ the function $q(-,\hbar)$ is in $\Sym^m (U)$
and such that $q$ has an asymptotic expansion of the form
\begin{displaymath}
  q \sim \sum_{k\in \N} \hbar^k a_{m-k} ,
\end{displaymath}
where each $a_{m-k}$ is a symbol in $\Sym^{m-k} (U)$. More precisely,
this means that one has for all $N \in \N$
\begin{displaymath}
  \lim_{\hbar \searrow \, 0} \Big( q (-,\hbar)  - \hbar^{-N}
  \sum_{k=0}^N \hbar^k a_{m-k}\Big) = 0 \quad \text{in $\Sym^{m-N} (U)$}.
\end{displaymath}
Like above one then obtains sheaves $\ASym^m$ for $m\in \Z \cup \{ \pm\infty \}$.
Now consider the subsheaves
$\JSym^m  \subset \ASym^m$ consisting of
all asymptotic symbols which vanish to infinite order at
$\hbar =0$. The quotient sheaves  $\A^m := \ASym^m / \JSym^m$
can then be identified with the formal power
series sheaves $\Sym^m [[\hbar]]$.

The operator product on $\PDO^\infty$ induces an
asymptotically associative product on $\ASym^\infty (Q)$ by
defining for $q,p \in \ASym^\infty (Q)$
\begin{equation}
\label{as-product}
   q \circledast p   :=
   \begin{cases}
      \sigma_\hbar \big( \Op_\hbar (q) \circ \Op_\hbar (p) \big) & \text{if $\hbar > 0$} ,\\
      q(-,\hbar ) \cdot p (-,\hbar )& \text{if $\hbar =0$}.
   \end{cases}
\end{equation}
Hereby, $\Op_\hbar = \Op \circ \iota_\hbar$ and $\sigma_\hbar = \iota_{\hbar^{-1}}\circ\sigma$,
where $\iota_\hbar : \Sym^\infty (Q) \rightarrow \Sym^\infty (Q)$
is the map which maps a symbol $a$ to the symbol
$(x,\xi) \mapsto a(x,\hbar \xi)$.
By standard techniques of pseudodifferential calculus (cf.~\cite{P1}),
one checks that $\circledast$ has an asymptotic expansion
of the following form:
\begin{equation}
\label{Eq:AsExOpStar}
  q \circledast p \sim q \cdot p + \sum_{k =1}^\infty c_k (q,p) \, \hbar^k ,
\end{equation}
where the $c_k$ are bidifferential operators on $T^*Q$ such that
\begin{displaymath}
    c_1 (a,b) -c_1 (b,a) = -i \{a,b\} \quad
    \text{for all symbols $a,b \in \Sym^\infty (Q)$}.
\end{displaymath}
Hence, $\circledast$ is a star product on the quotient sheaf $\A^\infty$,
which gives rise to a deformation quantization for the sheaf $\calA_{T^*Q}$
of smooth functions on the cotangent bundle $T^*Q$. By definition of
the product $\circledast$  it is clear that for the Schwartz kernels
of two operators $\Op_\hbar (q ) $ and $\Op_\hbar (q)$ one has the
following relation:
\begin{equation}
  \label{eq:prodschwker}
  K_{\Op_\hbar \big( q \circledast p  \big)} (x,y) =
  \int_Q K_{\Op_\hbar (q )} (x,z) \, K_{\Op_\hbar (p )} (z,y) \,dz .
\end{equation}

Even though $\circledast$ is not
obtained by a Fedosov construction, it is equivalent to a Fedosov star product
$\star$ on $T^*Q$ by \cite{nets95}. In the following, we fix $\star$ to be such a
Fedosov star product, and assume that it is obtained by a Fedosov connection $A$ constant
along the fibers of $T^*Q$. Note that by the equivalence of  $\circledast$ and $\star$,
each trace functional for $\circledast$ is one for $\star$ and vice versa.

Using the riemannian metric on $Q$ one even obtains a trace functional $\Tr$
on $\A^\infty$ by the following construction. Pseudodifferential operators
$\PDO^{- \dim Q}_\text{\tiny \rm cpt} (Q)$ act as trace class operators on the Hilbert
space $L^2 (Q)$. Thus there is a map
\begin{displaymath}
  \Tr : \A^{- \infty}_\text{\tiny \rm cpt} (Q)
  \rightarrow \C[\hbar^{-1} ,\hbar]], \quad
  q \mapsto \tr \big( \Op_\hbar (q) \big),
\end{displaymath}
where $\tr$ is the operator trace.
By construction, $\Tr$ has to be a trace with respect to $\circledast$
and is $\operatorname{ad} (\A^\infty)$-invariant.
Using the global symbol calculus for pseudodifferential operators
\cite{Wid:CSCPO,P1} the following formula can be derived:
\begin{equation}
\label{Eq:OpTr}
  \Tr (q) = \frac{1}{(2\pi \sqrt{-1}\hbar)^{\dim Q}}
  \int_{T^*Q} q ( -, \hbar) \, \frac{\omega^{\dim Q}}{(\dim Q) !} ,
\end{equation}
where $\omega$ is the canonical symplectic form on $T^*Q$. Moreover, by the remarks
above, $\Tr$ is also a trace with respect to the Fedosov star product $\star$.
Finally note that for all operators $A \in \APDO^\infty (Q)$
\begin{equation}
  \label{eq:quinvsigma}
  \tr \big( \Op_\hbar \sigma_\hbar (A) - A \big) = 0 .
\end{equation}

{\bf Step 3.}
The final step reduces the computation of the higher indices to the algebraic
higher indices using the global symbol calculus of step 2.
We begin by defining the analytic higher index using the localized $K$-theory
of {\sc Moscovici--Wu} \cite{moswu}. Let $Q$ be a compact riemannian manifold and consider the
smoothing operators $\Psi DO^{-\infty}(Q)$ acting on $L^2(Q)$. These operators have
a smooth Schwartz kernel, and therefore
$\Psi DO^{-\infty}(Q)\cong C^\infty_\text{\rm\tiny cpt}(Q\times Q)$.
Note that by assumptions on $Q$, every element $K \in \Psi DO^{-\infty}(Q)$ is
trace-class,
and $\Psi DO^{-\infty}(Q)$ is dense in the space of trace class operators on $L^2(Q)$.
For any finite open covering $\calU$ of $Q$, we define
\[
 \Psi DO^{-\infty}(Q,\calU):=\{K\in\Psi DO^{-\infty}(Q) \mid \operatorname{supp} (K)
 \subset \calU^2\},
\]
where $\calU^k := \bigcup_{U\in\calU} U^k$ for $k\in \N^*$.
Now let $M_\infty \big( \Psi DO^{-\infty}(Q,\calU)\big)$ be the inductive limit
of all $N\times N$-matrices with entries in $\Psi DO^{-\infty}(Q,\calU)$.
Likewise, define
$M_\infty \big( \Psi DO^{-\infty}(Q,\calU)^\sim\big)$ and
$M_\infty ( \C )$, where
$\Psi DO^{-\infty}(Q,\calU)^\sim := \Psi DO^{-\infty}(Q,\calU) \oplus \C$.
With these preparations, one defines
\begin{equation}
\label{Eq:DefCovKTheory}
\begin{split}
   K^0& (Q,\calU)  :=K_0 \big(\Psi DO^{-\infty}(Q,\calU) \big) := \\
   &:= \big\{ (P,e) \in M_\infty \big( \Psi DO^{-\infty}(Q , \calU )^\sim \big)
   \times M_\infty (\C) \big) \mid
   P^2 =P , \: P^* =P, \: \\
   &\hspace{8mm}  e^2 =e , \:  e^* =e  \: \text{ and } \:
   P-e \in M_\infty \big( \Psi DO^{-\infty}(Q,\calU)\big)\big\} / \sim ,
\end{split}
\end{equation}
where $(P,e) \sim (P', e')$ for projections
$P,P' \in M_\infty \big( \Psi DO^{-\infty}(Q,\calU)^\sim\big)$
and $e,e' \in  M_\infty (\C )$, if the elements $P$ and $P'$
can be joined by a continuous and piecewise $\calC^1$ path of projections in some
$M_N \big( \Psi DO^{-\infty}(Q,\calU)\big)$ with $N\gg 0 $
and likewise for $e$ and $e'$ (see \cite[Sec.~1.2]{moswu} for further details).
Elements of $K^0 (Q,\calU)$ are represented as equivalence classes
of differences $R:= P-e$, where $P$ is an idempotent in
$M_\infty \big( \Psi DO^{-\infty}(Q,\calU)^\sim\big)$, $e$ is a projection in
$M_\infty (\C)$, and the difference $P-e$ lies in
$M_\infty \big( \Psi DO^{-\infty}(Q,\calU)\big)$.

A (finite) refinement $\calV\subset\calU$ obviously leads to an inclusion
$\Psi DO^{-\infty}(Q,\calV)\hookrightarrow \Psi DO^{-\infty}(Q,\calU)$ which
induces a map $K^0(Q,\calV)\to K^0(Q,\mathcal{U})$. With these maps, the
localized $K$-theory of $Q$ is defined as
\begin{equation}
\label{Eq:LocKTheory}
  K^0_{\rm\tiny loc}(Q):=\lim_{\longleftarrow\atop \calU \in \operatorname{Cov}^\text{\rm\tiny fin}(Q)}
  K^0(Q,\calU).
\end{equation}
Concretely, this means that elements of $K^0_{\rm\tiny loc}(Q)$ are given by families
\begin{equation}
\label{Eq:LocKTheoryRep}
  \big( [P_\calU - e_\calU]\big)_{\calU \in
  \operatorname{Cov}^\text{fin}(Q)}
\end{equation}
of equivalence classes of pairs of projectors
in matrix spaces over $\Psi DO^{-\infty}(Q,\calU)^\sim$
such that $e_\calU \in M_\infty (\C )$ for every finite covering $\calU$ and
$(P_\calU,e_\calU)\sim (P_\calV,e_\calV)$ in
$M_\infty \big(\Psi^{-\infty}(Q,\calU)^\sim \big)$
whenever $\calV\subset\calU$.

Following \cite{moswu}, we now construct the so-called (even)
Alexander-Spanier-Chern character map
\[
  \Ch^\text{\rm\tiny AS}_{2\bullet} : K^0_{\rm\tiny loc}(Q) \rightarrow
  H^\text{\rm\tiny AS}_{2\bullet} (Q).
\]
As a preparation for the construction we set for every subset $W\subset Q$,
$k\in \N$ and every finite covering $\calV$ of $Q$
\[
\begin{split}
  & \operatorname{st}^k (W,\calV) \, :=
  \bigcup_{(V_1, \ldots ,V_k) \in \operatorname{chain}^k (W,\calV) }
  V_1 \cup \ldots \cup V_k, \quad \text{where } \\
  & \operatorname{chain}^k (W,\calV)  := \\
  & \hspace{1em}:= \{(V_1, \ldots ,V_k) \in \calV^k\mid
  W \cap V_1 \neq \emptyset , \: V_1 \cap V_2 \neq \emptyset,
  \: \ldots ,\:
  V_{k-1} \cap V_k \neq \emptyset\}.
\end{split}
\]
Then we define $\operatorname{st}^k (\overline{\calV})$
as the open covering of $Q$ with elements
$\operatorname{st}^k (\overline{V},\calV)$ where $V$ runs through the
elements of $\calV$. Obviously, one then has
\[
  \underbrace{\Psi DO^{-\infty}(Q,\calV) \cdot \ldots \cdot
  \Psi DO^{-\infty}(Q,\calV)}_{k-\text{times}} \subset
  \Psi DO^{-\infty}(Q,\operatorname{st}^k (\overline{\calV})) .
\]
Next let us fix an even homology degree $2k$ and a finite open covering
$\calU$ of $Q$. Then choose a finite open covering $\calU_0$ of $Q$ such that
$\operatorname{st}^k (\overline{\calU_0})$ is a refinement of $\calU$.
Now let $R_{\calU_0} := P_{\calU_0}- e_{\calU_0} \in \Psi DO^{-\infty}(Q,\calU_0 )$
represent  an element of $K^0 (Q,\calU_0 )$ as defined above, and put
for $f_0, \ldots , f_{2k} \in \calC^\infty (Q)$
\begin{equation}
\label{Eq:DefASChern}
\begin{split}
  \big( \Ch^\text{\rm\tiny AS}_{2k} (R_{\calU_0}) \big) &\,
  (f_0 \otimes \ldots \otimes f_{2k} ):= \\
  := \, & (-2\pi i)^k \frac{(2k)!}{k!}\varepsilon^{2k}
  \tr \big( (f_0 P_{\calU_0} f_1 \ldots f_{2k}P_{\calU_0} ) -  ( f_0 e _{\calU_0}f_1 \ldots f_{2k}e_{\calU_0} ) \big)
\end{split}
\end{equation}
It has been shown in \cite[Sec.~1.4]{moswu} that the right hand side even
defines a cycle in $C_{2k}^\text{\rm\tiny AS}(M,\calU)$, hence one obtains
a homology  class
$\Ch^\text{\rm\tiny AS}_{2k} (R_{\calU_0},\calU) \in
H_{2k}^\text{\rm\tiny AS}(M,\calU)$.
Moreover, a family $R=(R_\calV)_{\calV \in \operatorname{Cov}^\text{\rm\tiny fin}(Q)}$
defining a local K-theory  class gives rise to a family of compatible homology classes
$\Ch^\text{\rm\tiny AS}_{2k} (R_{\calU_0},\calU)$,
$\calU \in \operatorname{Cov}^\text{\rm\tiny fin}(Q)$,
hence by the universal properties of inverse limits one finally obtains a character map
$\Ch^\text{\rm\tiny AS}_{2k} :K^0_\text{\rm\tiny loc} (Q) \rightarrow
H^\text{\rm\tiny AS}_{2k} (Q)$ indeed. Let us now reformulate the pairing
$\left\langle [f] , \Ch^\text{\rm\tiny AS}_{2k} ([R])\right\rangle$, where
$[f]$ denotes an Alexander-Spanier cohomology class of degree $2k$.
Without loss of generality, we can assume that $[f]$ has the form
$[f_0 \otimes \ldots \otimes f_{2k}]$ with the $f_i$ being smooth functions
on  $Q$. Then note that the operator trace $\tr$ on $L^2(Q)$
induces a trace on $\Psi^{-\infty}(Q)$. With this trace, equation
\eqref{eq:charxdef1} defines a morphism
\[
 \scrX^\calU_{\tr}:C^{2k}_{AS}(Q,\calU)\to
 C^{2k}_\lambda \left( \Psi^{-\infty}(Q , \calU_0 )\right)
\]
which is uniquely determined by the requirement
\[
  \scrX^\calU_{\tr}(f_0\otimes\ldots\otimes f_k)
  (R_0\otimes\ldots\otimes R_k):=
  \tr\left(f_0R_0\cdots f_k R_k\right),
\]
where the $f_i$ on the right hand side are viewed as bounded multiplication
operators on $L^2(Q)$, and the $R_i$ are elements of $\Psi^{-\infty}(Q,\calU_0)$.
Note that on the right hand side
$C^{2k}_\lambda \left( \Psi^{-\infty}(Q , \calU_0 )\right)$ is the restriction
of  the space of cyclic $2k$-cochains
$C^{2k}_\lambda \left( \Psi^{-\infty} (Q) )\right)$ to elements
of $\Psi^{-\infty}(Q , \calU_0 )$.
By the construction of the pairing in Alexander-Spanier homology in Remark
\ref{Rem:AShom} and the definition of $\Ch^\text{\rm\tiny AS}_{2k}$ above,
the pairing between localized $K$-theory and Alexander--Spanier cohomology can
be rewritten as
\begin{equation}
  \label{Eq.EquASPairs}
   \left\langle [f], \Ch^\text{\rm\tiny AS}_{2k} ([R]) \right\rangle
  =\left\langle \scrX^\calU_{\tr}( \varepsilon^{2k} f),\Ch (R_{\calU_0} )\right\rangle,
\end{equation}
where $R =(R_\calV)_{\calV \in \operatorname{Cov}^\text{\rm\tiny fin}(Q)}$ is as above,
$\Ch$ is the noncommutative Chern character (on the chain level) as defined by
Eq.~\eqref{eq:defchern}, and where $\calU$
is a sufficiently fine covering such that in particular $\calU^{2k+1}$ is contained
in the domain of the function $f$ defining the Alexander-Spanier cohomology class
$[f]$.

Let us now come to the definition of the localized index, or in other words,
the higher index which originally was defined by {\sc Connes--Moscovici}
in \cite[\S 2.]{conmos}. To this end assume first that  $E\rightarrow Q$
is an Hermitian vector bundle over $Q$ and that $D$ is an elliptic pseudodifferential
operator acting on the space of smooth sections $\Gamma^\infty (E)$.
The operator $D$ gives rise to an invertible principle symbol
$\sigma_\text{\tiny\rm pr} (D) \in \calC^0 (T^*Q \setminus Q) $.
Its restriction to the cosphere bundle will
be denoted by
\[
  \sigmares (D) := \sigmapr (D)_{|S^*Q}.
\]
The restricted principal symbol $\sigmares (D)$ defines an element in the odd
K-group $K_1 \big( \calC^\infty (S^*Q)\big)$. Morever, as explained in
\cite[p.~353]{conmos}, one can associate to $\sigmares (D)$ and each finite covering
$\calU$ of $Q$ an element
$R_\calU = P_\calU - e_\calU \in \Psi^{-\infty}(Q,\calU)$ which is
constructed as a difference of a certain pseudodifferential projection $P$ of order
$-\infty$ on $Q$ and a projection in the matrix algebra over $\C$ and which fulfills the
crucial relation
\[
  \ind (D) = \tr R_\calU .
\]
Note that $R_\calU$ is homotopic to the graph projection of $D$ (cf.~\cite{ellnesnat}),
and that the induced class $[R] \in K^0_\text{\rm\tiny loc} (Q)$ of
the family $R=(R_\calU)_{\calU}$ depends only
on the class of $\sigmares (D)$ in $K_1 \big( \calC^\infty (S^*Q)\big)$.
One thus obtains a map
$\partial : K_1 \big( \calC^\infty (S^*Q)\big) \rightarrow K^0_\text{\rm\tiny loc} (Q)$
which we call the \textit{local index map}.
Next let $[f]$ be an even Alexander--Spanier cohomology
class of degree $2k$ which is represented by the function
$f \in \calC^\infty (Q^{2k+1})$. Then one defines
the \textit{localized index} or \textit{higher index} of $D$ at $[f]$ as
the pairing
\begin{equation}
\label{Eq:deflocind}
 \ind_{[f]} (D) := \left\langle  [f] ,
 \Ch^\text{\rm\tiny AS}_{2k} \big(\partial [\sigmares (D)]\big)\right\rangle .
\end{equation}
Note that according to the work of \cite{moswu}, this localized
index can be transformed into the original definition of the localized
index by {\sc Connes--Moscovici}:
\begin{equation}\label{Eq:locexplocind}
  \begin{split}
  \ind_{[f]} (D) \, & := (-1)^k \int_{Q^{2k+1}}
  \tr \big( R_\calV(x_0, x_1) \cdot \ldots \cdot R_\calV(x_{2k-1},x_{2k}) \big)
  \, f(x_0 , \ldots , x_{2k} ) d\mu^{2k+1} \\
  & = \scrX^{\calU}_{\tr} (f) \big( R_\calV \otimes \ldots \otimes R_\calV \big) ,
  \end{split}
\end{equation}
where here $\mu$ is the volume form on $Q$, $R:= (R_\calU )_\calU :=\partial \sigmares (D)$,
and $\calV$ is a finite covering sufficiently fine such that $\calV^{2k+1}$ is contained in
$\calU^{2k+1}$, the domain of the function $f$ defining the Alexander-Spanier cohomology class $[f]$.

Now let $a_i = \sigma_\hbar (A_i)$, $i=0,\ldots ,k$ be the asymptotic symbols of
pseudodifferential operators $A_i \in \APDO^{-\infty}$. For all Alexander Spanier
cochains $f_0 \otimes \ldots \otimes f_k \in \calC^\infty (Q^{k+1})$ the following
relation then holds true asymptotically in $\hbar$:
\begin{equation}
\label{eq:symbol}
  \begin{split}
    \scrX_{\tr} & \big( f_0 \otimes \ldots \otimes f_k \big)
    \big( A_0 \otimes \ldots \otimes A_k\big)
    := \tr \big( f_0 A_0 \cdot \ldots \cdot f_k A_k \big) =  \\
    & = \tr \big( f_0 \Op_\hbar (a_0) \cdot \ldots \cdot f_k \Op_\hbar (a_k) \big) \\
    & = \tr \big( \Op_\hbar (f_0 a_0) \cdot \ldots \cdot \Op_\hbar (f_k a_k) \big) \\
    & = \tr \big( \Op_\hbar (f_0 a_0 \circledast \ldots \circledast f_k a_k )\big) \\
    & = \Tr \big( f_0 \circledast  a_0 \circledast \ldots \circledast f_k  \circledast a_k  \big)\\
    & = \Tr \big( f_0 \star  a_0 \star \ldots \star f_k  \star a_k \big) \\
    & = \scrX_{\Tr} \big( f_0 \otimes \ldots \otimes f_k \big)
    \big( a_0 \otimes \ldots \otimes a_k\big).
  \end{split}
\end{equation}
Hereby, we have used that $f_i \Op_\hbar (a_i) = \Op_\hbar (f_ia_i)$, and that by
Eq.~(\ref{eq:prodschwker})
\[
  \tr \big( \Op_\hbar (a_i \circledast a_{i+1} )\big) =
  \tr \big( \Op_\hbar (a_i)\Op_\hbar (a_{i+1}) \big).
\]
Using the results from Step I together with Eqns.~\eqref{Eq.EquASPairs} and \eqref{eq:symbol}
one now obtains with $r_\calV = :\sigma_\hbar  R_\calV $ the asymptotic symbol of
$R_\calV$, and $f = f_0 \otimes \ldots \otimes f_{2k}$
\begin{displaymath}
  \begin{split}
   \ind_{[f]} & (D)=
   \left\langle [f], \Ch^\text{\rm\tiny AS}_{2k}\big(\partial [\sigmares (D)]\big)\right\rangle
   \stackrel{(\ref{Eq.EquASPairs})}{=}
   \left\langle \scrX^{\calU}_{\tr} \varepsilon^{2k}( f) ,\Ch\big( R_{\calV} \big)\right\rangle =\\
   & \stackrel{(\ref{eq:symbol})}{=}
   \left\langle \scrX_{\Tr} \varepsilon^{2k} (f) , \Ch r_{\calV} \right\rangle =
    \left\langle \normscrX_{\Tr} \varepsilon^{2k} ([f]) , \Ch r\right\rangle
   \stackrel{\text{Prop}. \ref{prop:QISM-ASCyc}}{=}
   \left\langle \scrQ \lambdaAS \varepsilon^{2k} ([f]) , \Ch r \right\rangle = \\
   & \hspace{2mm}
   = \: \frac{1}{(2\pi\sqrt{-1})^k}\int_{T^*Q} f_0 df_1 \wedge \ldots \wedge df_{2k} \wedge
   \hat {A}(T^*Q) \Ch (V_1-V_2) .
  \end{split}
\end{displaymath}
Hereby, $V_1 -V_2$ is the virtual vector bundle obtained by the asymptotic limit
$\hbar \searrow 0$ of $r$, and $r$ is the symbol of $R_Q$ with $Q$ denoting here the
trivial covering of $Q$.
We have thus reproved the following result from \cite{conmos}.
\begin{theorem}
  For an elliptic differential operator $D$ on a riemannian manifold $Q$
  and an Alexander--Spanier cohomology class $[f]$ of degree $2k$ with compact support
  the localized index is given by
  \[
    \ind_{[f]} (D)= \frac{1}{(2\pi\sqrt{-1})^k}
    \int_{T^*Q} f_0 df_1 \wedge \ldots \wedge df_{2k} \wedge
    \hat {A}(T^*Q) \Ch (V_1-V_2) .
  \]
\end{theorem}

%%% Local Variables:
%%% mode: latex
%%% TeX-master: "HigherIndex"
%%% End:

%% file: AnaHighIndOrb-new.tex
%
%
\section{A higher analytic index theorem for orbifolds}
\label{sec:HAITO}
In this section, which comprises the final part of this work, we prove the higher
index theorem for elliptic differential operators on orbifolds as an application of
the higher algebraic index theorem for proper \'etale groupoids of Section
\ref{Sec:GenOrbi}. This generalizes the higher index theorem by {\sc Connes--Moscovici}
to the orbifold setting.

Although our strategy for the proof is the same as in Section \ref{sec:HAITM},
the generalization is by no means straightforward: we start  in
Section \ref{subsec:oASc} with defining the Alexander--Spanier cochain complex for
proper \'etale groupoids $\sfG$. This cochain complex depends on the groupoid
structure, and instead of being localized to the diagonal, the cochains are localized
to the so-called ``higher Burghelea spaces''.  In Section \ref{constr-as}, we explain
how cohomology classes are represented by functions in a sufficiently small
neighbourhood of these Burghelea spaces so that we can have
cocycles acting on $L^2(\sfG_0)$ by convolution. This is important in
Section \ref{pairing} for the pairing with orbifold localized $K$-theory.

In Section \ref{subsec:RelOrbAScycCoh}, we relate orbifold Alexander--Spanier
cohomology to the cyclic cohomology of a deformation quantization of the convolution
algebra. In Section \ref{pairing}, the orbifold version of localized $K$-theory is
introduced in terms of a filtration in which smoothing operators on $\sfG_0$ are
localized to the diagonal and invariance is imposed. Via a Chern character, such
$K$-theory classes pair with localized Alexander--Spanier cocycles.

The link between the two pairings of Alexander--Spanier cohomology, namely on the one
side the pairing with localized $K$-theory and on the other side with the cyclic
cohomology of a deformation quantization, is given by a global $\hbar$-dependent
symbol calculus for pseudodifferential operators on orbifolds as constructed in
\cite{ppt}. This induces a deformation quantization over the cotangent bundle of the
underlying orbifolds with which we can compare the two pairings. The higher index
theorem finally follows by application of this idea to the
canonical localized $K$-theory class induced by the elliptic operator.

\subsection{Orbifold Alexander-Spanier cohomology}
\label{subsec:oASc}
As before, we denote by $M$ an orbifold given as a quotient space of a proper \'etale
Lie groupoid $\sfG_1\rightrightarrows\sfG_0$. The orbifold version of the
Alexander--Spanier sheaf complex is constructed as follows: again we consider the
space of loops $B^{(0)}\subset\sfG_1$. On this space define the following sheaves:
\[
  \scrC^{k}_{\rm AS, tw}(\calO):=s^{-1}\calO^{\boxtimes(k+1)}_{\sfG_0},
\]
where $\calO_{\sfG_0}$ is a sheaf of unital algebras as before.
We introduce a cosimplicial structure with coface operators
$\bar{\delta}^i:\scrC^{k}_{\rm AS, tw}(\calO)\to
\scrC^{k+1} _{\rm AS, tw}(\calO),~i=0,\ldots,k+1$ given by
\[
\bar{\delta}^i(f_0\otimes\ldots\otimes f_{k}):=f_0\otimes\ldots\otimes f_{i-1}\otimes 1\otimes f_{i+1}\otimes\ldots\otimes f_{k},
\]
and degeneracies $s^i:\scrC^{k}_{\rm AS, tw}(\calO)\to
\scrC^{k-1} _{\rm AS, tw}(\calO),~i=0,\ldots,k-1$ defined by
\[
  \bar{s}^i(f_0\otimes\ldots\otimes f_{k}):=
  f_0\otimes\ldots\otimes f_if_{i+1}\otimes\ldots\otimes f_k.
\]
So far, nothing new, but this time the cyclic structure $\bar{t}^k:\scrC^{k}_{\rm AS, tw}(\calO)\to\scrC^{k}_{\rm AS, tw}(\calO)$ is given by
\[
\bar{t}(f_0\otimes\ldots\otimes f_k):=f_1\otimes\ldots\otimes f_k\otimes \theta^{-1}(f_0),
\]
where $\theta:B^{(0)}\to\sfG_1$ is the cyclic structure of the groupoid $\Lambda(\sfG)$.
Recall, cf.\ \cite[Def. 3.3.1]{crainic}, that $\theta$, and also $\theta^{-1}$,
equips $\Lambda(G)$ with the structure of a {\em cyclic groupoid}. Using that
notion, it is not difficult to verify that with these structure maps
$\scrC^{\bullet}_{\rm AS, tw}(\calO)$ is a cocyclic sheaf on the cyclic groupoid
$(\Lambda(\sfG),\theta^{-1})$ and  gives rise to an $\infty$-cocyclic object
$(\scrC^{\bullet}_{\rm AS, tw}(\calO),\bar{\delta}, \bar{s},\bar{t})$
in the category of $\sfG$-sheaves over $B^{(0)}$ such that $\bar{t}^{k+1}=\theta^{-1}$
in each degree $k$.
\begin{remark}
Pulling back the standard cyclic sheaf of algebras $\calO^\natural_{\sfG_0}$ on $\sfG_0$ to $B^{(0)}$, there is a way to twist the structure maps
by the cyclic structure $\theta$, cf.\ \cite{crainic}. The cyclic sheaf above is simply the cyclic dual of this one. Notice that there is no twist in the
degeneracies because exactly the face operator containing the twist in
$s^{-1}\calO^\natural_{\sfG_0}$ is not used in the definition of the dual,
cf.~\cite[\S 6.1]{loday}.
\end{remark}
Associated to the underlying simplicial complex is the Hochschild sheaf complex $(\scrC^\bullet_{\rm AS, tw}(\calO),\bar{\delta})$ with differential
$\bar{\delta}=\sum_{i=0}^k(-1)^i\bar{\delta}^i$.
\begin{definition}
\label{dfn:as-coh}
The orbifold Alexander--Spanier cohomology $H^\bullet_{\rm AS, orb}(M,\calO)$ of $M$ with values in $\calO$ is defined to be the
groupoid sheaf cohomology of the complex $(\scrC^\bullet_{\rm AS, tw}(\calO),\bar{\delta})$.
\end{definition}
As alluded to in the notation, orbifold Alexander--Spanier cohomology is independent
of the particular groupoid $\sfG$ representing its Morita equivalence class. In fact
we have:
\begin{proposition}
\label{nat-iso-as}
There is a natural isomorphism
\[
  H^\bullet_{\rm AS, orb}(M,\calO)\cong H^\bullet(\tilde{M},\Bbbk).
\]
\end{proposition}
\begin{proof}
As for manifolds, cf.\ Proposition \ref{prop:resAS}, the inclusion
$\underline{\Bbbk}\hookrightarrow \scrC^\bullet_{\rm AS, tw}$ is a quasi-isomorphism
in $\mathsf{Sh}(\Lambda(\sfG))$ since it is clearly compatible with the $\sfG$-action
on both sheaves. But for the locally constant sheaf $\underline{\Bbbk}$ we have the
natural isomorphism
$H^\bullet(\Lambda(\sfG),\underline{\Bbbk})\cong H^\bullet(\tilde{M},\Bbbk)$.
\end{proof}
As groupoid cohomology, orbifold Alexander--Spanier cohomology can be computed
using the Bar complex of $\Lambda(\sfG)$. However, instead of using the nerve
of $\Lambda(\sfG)$, we shall  use the isomorphic Burghelea spaces associated
to $\sfG$ to write down such a Bar complex. Introduce
\[
B^{(k)} :=\big\{ (g_0, \ldots, g_k)\in \sfG^{k+1} \mid
  s(g_0)=t(g_1), \ldots, s(g_{k-1})=t(g_k), s(g_{k})= t(g_0)\big\} .
\]
These Burghelea spaces $B^{(k)}$ form a simplicial manifold with face maps
\begin{equation}
\label{cycll-burg}
d_i(g_0,\ldots,g_k)=
\begin{cases}
(g_0,\ldots,g_ig_{i+1},\ldots,g_k),&0\leq i\leq k-1,\\
(g_kg_0,\ldots, g_{k-1}),&i=k.
\end{cases}
\end{equation}
Consider now the map $\bar{\sigma}_k:B^{(k)}\to\sfG_0^{(k+1)}$ given by
$\bar{\sigma}_k(g_0,\ldots,g_k)=(s(g_0),\ldots,s(g_k))$.
With this we define for each $k\in \N$ the sheaf $\calS_k := \bar{\sigma}_k^* \big(
\calO^{\hat{\boxtimes}(k+1)} \big)$, the pullback sheaf of
$\calO^{\hat{\boxtimes} (k+1)}$ to $B^{(k)}$. We write
$\calAS^k (\sfG, \calO) :=\Gamma \big( B^{(k)}, \calS_k \big) $ and observe
that a (bornologically) dense subspace of the
space of sections $\Gamma \big( B^{(k)}, \calS_k \big)$ is given by
sums of sections of the form
\[
\begin{split}
  f = f_0 \otimes \cdots \otimes f_k :
  %&\, B^{(k)} \rightarrow
  %\bigsqcup_{(g_0,\cdots,g_k) \in B^{(k)}}
  %\big( \calO^{\hat\boxtimes k+1}\big)_{s(g_0), \ldots , s(g_k)} , \\
 (g_0,\cdots,g_k) \mapsto (f_0)_{[g_0]} \otimes \ldots \otimes (f_k)_{[g_k]} ,
\end{split}
\]
where $(f_i)_{[g_i]}\in \calO_{s(g_i)}$ and $(g_0,\ldots,g_k)\in B^{(k)}$.
With this in mind, we introduce a simplicial structure on
$\calAS^\bullet (\sfG,\calO)$ by means of the coface maps
$\delta^i:\calAS^{k-1} (\sfG ,\calO) \rightarrow \calAS ^k(\sfG ,\calO)$
defined as
\[
\begin{split}
\delta^i(f_0\otimes\ldots&\otimes f_{k-1})_{[g_0,\ldots,g_k]}\\
&=\begin{cases}
1_{s(g_0)}\otimes (f_0)_{[g_0]}\otimes (f_1)_{[g_2]}\otimes\ldots\otimes (f_{k-1})_{[g_kg_0]}&i=0\\
(f_0)_{[g_0]}\otimes\ldots\otimes (f_i)_{[g_ig_{i+1}]}g_i^{-1}\otimes 1_{s(g_i)}\otimes\ldots
\otimes (f_{k-1})_{[g_{k}]}&\mbox{for}~ 1\leq i \leq k.
\end{cases}
\end{split}
\]
Codegeneracies are given by
\[
\begin{split}
s^i&(f_0 \otimes \ldots \otimes f_{k+1})_{[g_0, \cdots, g_{k}]}\\&=
(f_0)_{[g_0]} \otimes \ldots \otimes (f_{i-1})_{[g_{i-1}]} \otimes
  (f_i)_{[ s(g_{i-1}]} \cdot (f_{i+1})_{[ g_i]}  \otimes (f_{i+2})_{[g_{i+1}]} \otimes
  \ldots \otimes (f_{k+1})_{[g_k]}.
  \end{split}
  \]
Finally, we can define a compatible cyclic structure
$t^k:\calAS^k (\sfG,\calO) \to\calAS^k(\sfG,\calO)$
by
\[
t^k(f_0 \otimes \cdots \otimes f_k)_{[g_0, \cdots, g_k]}=
  (-1)^k(f_0)_{[g_k]} \otimes (f_1)_{[g_0]} \otimes \ldots \otimes (f_k)_{[g_{k-1}]}.
\]
It is straightforward to show that with these structure maps
$\calAS^\natural (\sfG ,\calO)$ is a cyclic cosimplicial vector space indeed.

To relate the above introduced cosimplicial complex with the Bar
complex of the sheaf cohomology of $\scrC^\bullet_{\rm AS,
tw}(\calO)$ on $\Lambda \sfG$ in Definition \ref{dfn:as-coh} we
identify $B^{(k)}$ with $\Lambda\sfG^{(k)}$ by the map $\nu$
\[
  \nu(g_0, \cdots, g_k)=(g_1\cdots g_kg_0, g_1, \cdots, g_k).
\]
The induced isomorphism $\nu_*$ on $\calAS^\natural$ is computed to
be
\[
\begin{split}
 \nu_* &\,(f_0\otimes \cdots\otimes f_k)(g_1\cdots g_kg_0, g_1,\cdots, g_k)=\\
 &=\big( (f_0(g_0)) g_1\cdots g_kg_0 , (f_1(g_1))g_2\cdots g_kg_0,
 \cdots, (f_k(g_k))g_0 \big).
\end{split}
\]
The  image $\nu_*(\calS_k)$ then is a sheaf on $\Lambda\sfG^{(k)}$.
Observe that the sheaf $\calS_k$ can be understood as the
pullback of a sheaf $\twcalO^k$ on $B^{(0)}$ through the map
\[
(g_1\cdots g_kg_0, \ldots, g_k)\mapsto g_1\cdots g_kg_0.
\]
It is easy to check that $\nu_*$ defines an isomorphism between the
complex $\calAS^\natural (\sfG, \calO)$ and the Bar complex on
$\Lambda\sfG$ of $\scrC^\natural_{\rm AS, tw}(\calO)$. Since $\Lambda(G)$
is proper, the Bar complex is quasi-isomorphic to the complex of invariant
sections on $B^{(0)}$. Denote by $\beta:B^{(k)}\to B^{(0)}$ the map
$\beta(g_0,\ldots,g_k)=g_0\cdots g_k$. Putting all this together, we have
\begin{proposition}
\label{Prop:EquORBAS}
  For every proper \'etale Lie groupoid $\sfG$ and sheaf $\calO$ as above
  \[
    \beta_* : \calAS^\natural (\sfG, \calO) \rightarrow
    \Gamma_\text{\rm inv} \big( B^{(0)},
    \scrC^\natural_{\rm AS, tw}(\calO) \big).
  \]
is a quasi-isomorphism of cochain complexes.
\end{proposition}

\subsection{Explicit realization of Alexander--Spanier cocycles}
\label{constr-as}
Recall that in the case of manifolds the Alexander--Spanier cochain complex could be written as a direct limit of a cochain complex of functions defined on a neighbourhood of the diagonal $\Delta_{k+1}:M\to M^{k+1}$ given in terms of the choice of an open covering of $M$.
This realization of Alexander--Spanier cocycles was crucial in the definition of the pairing with localized $K$-theory.
In this section we will generalize this construction to proper \'etale groupoids $\sfG$, where this time the role of the diagonal is played by the Burghelea space $B^{(k)}\hookrightarrow \sfG_1^{k+1}$.

Let $\sfG_1\rightrightarrows\sfG_0$ be a proper \'etale groupoid modeling an orbifold $M$, and $U\subset\sfG_0$ an open set. A {\em local bisection}, cf.~\cite[\S 5.1]{MoeMrcIFLG}, on $U$ is a local section $\sigma:U\to\sfG_1$ of the source map $s:\sfG_1\to\sfG_0$ such that $t\circ\sigma:U \to \sfG_0$ is an open embedding.
This second property of local bisections shows that they define local diffeomorphisms of $\sfG_0$ and as such the product
\begin{equation}
\label{comp-bs}
(\sigma_1\sigma_2)(x):=\sigma_1(t(\sigma_2(x)))\sigma_2(x)
\end{equation}
is defined if the domain of $\sigma_1$ contains the image of $t\circ\sigma_1$.
\begin{definition}
A covering $\calU=\{U_i\}_{i\in I}$ of $\sfG_0$ is said to be $\sfG$-trivializing if it satisfies the
following two  conditions:
\begin{itemize}
\item[$i)$] $\calU$ is the pull-back of a covering  of $M$ along the projection $\pi:\sfG_0\to M$. By this we mean that there exists a covering $\overline{\calU}$ of $M$ such that $\calU$ consists of the connected components of $\pi^{-1}(\overline{U}),~\overline{U}\in\overline{\calU}$.
\item[$ii)$] For all $g\in\sfG_1$, there exists an $i\in I$ and a bisection $\sigma_{i}$ defined on $U_i$ with $\sigma_{i}(s(g))=g$. In particular, $s(g)\in U_i$.
\end{itemize}
\end{definition}
 Remark that such a covering always exists and is completely determined by the induced covering of the quotient $M$. We will therefore denote the set of coverings of $M$ satisfying property $ii)$ above
 by $\mbox{Cov}_{\sfG}(M)$. Clearly, $\mbox{Cov}_\sfG(M)$ is directed by the notion of refinement. Remark that the property of being $\sfG$-trivializing very much depends
 on the groupoid $\sfG$, and not the quotient. As an easy example, consider a manifold $M$: when represented as a groupoid with only identity arrows, any covering satisfies the properties above. However, when represented as a \v{C}ech-groupoid associated to a fixed covering $\calU$, only
 those coverings that refine the covering $\calU'=\{U_i\cap U_j\}_{i,j\in I}$ are trivializing.

 Also, the $\sigma_i$ in $ii)$ are uniquely determined by $g$ because $\sfG$ is \'etale. Because of this, we shall write $\sigma_i^g$ for this local bisection.
 Furthermore, since $s\circ\pi=t\circ\pi$, we have $\pi(t(\sigma_i^g(U_i)))=\pi(U_i)$ so there exists
 a $j\in I$ such that $(t\circ\sigma_i^g)(U_i)=U_j$. Associated to the covering are the subsets
$\sfG_{ij}\subset\sfG_1,~i,j\in I$ defined by
\[
\sfG_{ij}:=\{g\in\sfG_1 \mid s(g)\in U_j,~\sigma_j^g(U_j)= U_i\}.
\]
The conditions on the covering ensures that $\bigcup_{i,j\in I}\sfG_{ij}=\sfG$.
With this notation, we introduce
\[
B^{(k)}_{\calU}:=\bigcup_{i_0,\ldots,i_k\in I}\sfG_{i_0i_1}\times\ldots\times\sfG_{i_ki_0}\hookrightarrow \sfG_1^{k+1}.
\]
Remark that there is a canonical embedding $B^{(k)}\subset B^{(k)}_{\calU}$.
\begin{lemma}
\label{cyclic-nghbhd}
The family of spaces $B^{(k)}_{\calU},~k\in\N$ carries a canonical cyclic
manifold structure which extends the cyclic structure on $B^{(\bullet)}$.
\end{lemma}
\begin{proof}
Let $(g_0,\ldots,g_k)\in B^{(k)}_\calU$. By definition, there are $i_0,\ldots,i_k\in I$ such that
$g_{j}\in \sfG_{i_j i_{j+1}}$. Let $\sigma_j$ be the unique local bisection $\sigma_j:U_{i_j}\to\sfG_1$
corresponding to $g_j$. By construction, $\sigma_{j+1}(U_{i_{j+1}})=U_{i_j}$, and we can define
the product of $g_j$ and $g_{j+1}$ as
\[
g_j\odot g_{j+1}:=(\sigma_j\sigma_{j+1})(s(g_{j+1})),
\]
with the product of the local bisections as in \eqref{comp-bs}. It is not difficult to check that this definition of the product is independent of the choice of $i_0,\ldots,i_k$. To see that it is associative it is best to think of the local bisections $\sigma_j$ as elements in the pseudogroup of local diffeomorphisms of $\sfG_0$. Finally, with this composition, we can define the cyclic structure on $B^{(k)}_\calU$ by the same formulae as in \eqref{cycll-burg}.  The proof that this indeed define a cyclic manifold is then routine.
Clearly, it induces the canonical cyclic structure on the Burghelea spaces.
\end{proof}
\begin{remark}
  Note that the product $g_1 \odot g_2$ coincides with the groupoid
  composition, if $s(g_1)=t(g_2)$. The product $\odot$ can thus be
  understood as an extension of the groupoid product around the
  ``orbifold diagonal'' meaning around the set of composable arrows.
\end{remark}

Let us now introduce the following complex:
\[
C^k_{AS}(\sfG,\calU):=C^\infty(B^{(k)}_\calU).
\]
The differential $\delta:C^k_{AS}(\sfG,\calU)\to C^{k+1}_{AS}(\sfG,\calU)$ is
defined by the formula
\[
\begin{split}
  (\delta f)\, & (g_0,\ldots,g_{k+1}):= \\
  & := \sum_{i=0}^k (-1)^i
  f(g_0,\ldots,g_i\odot g_{i+1},\ldots,g_k)+(-1)^{k+1}f(g_{k+1}\odot g_0,\ldots,g_k).
\end{split}
\]
Since the differential is defined in terms of the underlying simplicial structure on $B^{(\bullet)}_\calU$, we automatically have $\delta^2=0$.
\begin{example}
\label{example}
Consider the transformation groupoid $\Gamma\times X\rightrightarrows X$
associated to a group action of a discrete group $\Gamma$ on a manifold $X$.
By definition, $s(\gamma,x)=x$, $t(\gamma,x)=\gamma(x)$ for $x\in X$,
$\gamma\in\Gamma$, and a bisection on $X$ is given by an element
$\gamma\in\Gamma$. In this case, the trivial covering of $X/\Gamma$ obviously
satisfies the conditions $i)$ and $ii)$ above.  Unravelling the definition
this leads to the following complex associated to the trivial covering of $X$:
$C^k_{AS}(\Gamma\times X,X):=C^\infty\big( (\Gamma\times X)^{k+1} \big)$ and the
differential is given by
\begin{equation}
\label{diff-trsf}
\begin{split}
(\delta f)(\gamma_0,x_0,\ldots,\gamma_{k+1},x_{k+1}):=\sum_{i=0}^k&(-1)^i(\gamma_0,x_0,\ldots,x_{i-1},\gamma_i\gamma_{i+1},x_{i+1},\ldots,\gamma_{k+1},x_{k+1})\\&+(-1)^{k+1}f(\gamma_{k+1}\gamma_0,x_0,\ldots,\gamma_k,x_k).
\end{split}
\end{equation}
In particular, for $\Gamma$ the trivial group and any covering $\calU$, we find exactly the complex \eqref{Eq:DefASCochains}.
\end{example}
Clearly, if a covering $\calU$ satisfies condition $i)$ and $ii)$ above, a
refinement $\calV\hookrightarrow\calU$ also satisfies these conditions
and therefore induces a canonical map $C^\bullet_{AS}(\sfG,\calU)\to
C^\bullet_{AS}(\sfG,\calV)$. With this, we can take the direct limit
over the set of coverings of the orbifold $M$.
\begin{proposition}
In the limit, there is a canonical isomorphism
\[
  \lim_{\longrightarrow \atop \calU \in \operatorname{Cov}_{\sfG} (M)}
     C^\bullet_{\rm\tiny AS} (\sfG,\calU) \cong
     \mathcal{AS}^\bullet(\sfG,\calO).
\]
\end{proposition}
\begin{proof}
As the cover gets finer, the set $B^{(k)}_\calU\hookrightarrow  \sfG^{ k+1 }$
shrinks to the Burghelea space $B^{(k)}$. Therefore, the restriction of a
function $f\in C^\infty(B^{(k)}_\calU)$ to the germ $f|_{B^{(k)}}$ of $B^{(k)}$ induces the linear isomorphism as in the statement of the Proposition.
It is straightforward to show that this map is compatible with the differentials.
\end{proof}
By Proposition \ref{nat-iso-as}, we therefore have that the cohomology of the limit complex
\[
\lim_{\longrightarrow \atop \calU \in \operatorname{Cov}_{\sfG} (M)}
\left(
     C^\bullet_{\rm\tiny AS} (\sfG,\calU),\delta\right)
\]
equals $H^\bullet(\tilde{M},\Bbbk)$. In fact, unravelling all the isomorphisms involved, we have:
\begin{corollary}
The cohomology class in $H^k(\tilde{M},\Bbbk)$ induced by a cocycle $f=f_0\otimes\ldots\otimes f_{2k}\in C^\infty(B^{(k)}_\calU)$ is represented by the closed invariant  differential form on $B^{(0)}$ given by
\[
\nu_*\left(\lambda^k_k(f)|_{B^{(k)}}\right)_{g}
=\sum_{g_0,\ldots,g_k\in B^{(k)}\atop g_0\cdots g_k=g}
\sum_{\sigma\in S_{k+1}}f_{\sigma(0)}(g_0)df_{\sigma(1)}(g_1)\wedge\ldots\wedge df_{\sigma(k)}(g_k).
\]
In particular, if $f$ has compact support, the resulting differential form is compactly supported.
\end{corollary}
This gives us an explicit
way of representing cohomology classes in $H^\bullet(\tilde{M},\Bbbk)$
by cocycles defined on a sufficiently small neighbourhood of
the "orbifold diagonal" $B^{(k)}\hookrightarrow \sfG^{k+1}$.
For example, for a transformation
groupoid as in Example \ref{example}, we have
\[
\widetilde{M}=\left.\coprod_{\left<\gamma\right>\in\mbox{\tiny Conj}(\Gamma)}X^{\left<\gamma\right>}\right\slash {Z_{\left<\gamma\right>}}.
\]
As we have seen above, there are enough global bisections in this case, and we can use the trivial covering. We write $f=\sum_{\gamma\in\Gamma}f_\gamma U_\gamma$ for an element
in $f\in C^\infty(\Gamma\times X)$. The function
\[
f_{\left<\gamma\right>}=\sum_{\nu\in\Gamma}1U_{\nu\gamma\nu^{-1}},
\]
is closed under the Alexander--Spanier differential, $\delta f_{\left<\gamma\right>}=0$,
as an easy argument shows. By the canonical projection onto the direct limit complex, it induces
a cocycle of degree zero in the Alexander--Spanier complex. It is not difficult to see
that this is a generator of $H^0(X^{\left<\gamma\right>}\slash Z_{\left<\gamma\right>},\Bbbk)\subset
H^0(\tilde{M},\Bbbk)$.

\subsection{Relating orbifold Alexander--Spanier cohomology with cyclic
cohomology}
\label{subsec:RelOrbAScycCoh}

We assume in this step $\sfG_0$ is equipped with an invariant
symplectic form $\omega$. Let $\calA^{((\hbar))}$ be a (local) deformation
quantization on $\sfG_0$. According to \cite{ta},
$\calA^{((\hbar))}\rtimes \sfG$ is a deformation quantization over the
groupoid $\sfG$, which by definition is a deformation of the convolution
algebra on $\sfG$. In \cite{ppt}, we constructed a universal trace $\Tr$ on
$\calA^{((\hbar))}\rtimes \sfG$. The trace functional is defined by
\begin{equation}
  \label{eq:defunitr}
  \Tr (a) := \int_{B^{(0)}} \, \Psi^{2n-\ell (g)}_{2n} (a) ,  \quad
  a \in \calA^{((\hbar))} (\sfG_0),
\end{equation}
where $\Psi^{2n-\ell}_{2n}$ is defined in Section \ref{Sec:GenOrbi}
(cf.~Remark \ref{rem:trdens}).
In this step, we will use $\Tr$ to associate to each groupoid
Alexander-Spanier cocycle on $\sfG$ a cyclic cocycle on
$\calA^{((\hbar))}\rtimes \sfG$.

Note that there are two natural products on
$\calA^{((\hbar))}\rtimes \sfG$. Recall first that linearly
$\calA^{((\hbar))}\rtimes \sfG \cong \Gamma_\text{\rm
cpt}\big(\sfG_1,s^*\calA^{((\hbar))}\big)$, and that this space
carries the convolution product $\cstar$ defined by
\begin{equation}
\label{crossedprd}
  [f_1 \cstar f_2 ]_g= \sum_{g_1 \, g_2=g}\big( [f_1]_{g_1}g_2 \big)
  [f_2]_{g_2},
  \quad
  f_1,f_2 \in \Gamma_\text{\rm cpt}\big(\sfG_1,s^*\calA^{((\hbar))}\big), ~
  g\in \sfG_1,
\end{equation}
where $[f]_g$ denotes the germ of a section $\in \Gamma_\text{\rm
cpt}\big(\sfG_1,s^*\calA^{((\hbar))}\big)$ at the point $g\in
\sfG_1$. Secondly, the star product $\star$ can be canonically
extended to $\calA^{((\hbar))}\rtimes \sfG$ by putting
\begin{equation}
\label{extprd}
  [f_1 \star f_2 ]_g=  s_* [f_1]_g \star s_* [f_2]_g,
  \quad
  f_1,f_2 \in \Gamma_\text{\rm cpt}\big(\sfG_1,s^*\calA^{((\hbar))}\big), ~
  g\in \sfG_1.
\end{equation}
Now we can define $\scrX^{\sfG}_{\Tr}:\calAS^\bullet \big(\sfG,
\calC^\infty_{\sfG_0}((\hbar))\big) \rightarrow
C^\bullet\big( \calA^{((\hbar))}\rtimes \sfG \big)$ by
\begin{equation}
\label{eq:pairing-star-gpd}
\begin{split}
\scrX^{\sfG}_{\Tr}&(f_0\otimes \cdots \otimes f_k)(a_0 \otimes \cdots \otimes a_k)\\
&=\Tr_k \big( f_0 \star a_0 , \cdots , f_k \star a_k \big)
:=\Tr \big( (f_0 \star a_0)\cstar \cdots \cstar (f_k \star a_k) \big).
\end{split}
\end{equation}

Since in the definition of $\Tr (A)$ by Eq.~\eqref{eq:defunitr} only
the germ of $a \in \calA^{((\hbar))}\rtimes \sfG $ at $B^{(0)}$
enters, $\scrX^{\sfG}_{\Tr}(f)(a)$ with $f= f_0\otimes \cdots
\otimes f_n$ and $a=a_0 \otimes \cdots \otimes a_k$ depends only on
the germ of $(f_0 \star a_0)\cstar\cdots\cstar (f_k\star a_k)$ at
$B^{(0)}$. By definition of the products $\cstar$ and $\star$ on $
\calA^{((\hbar))}\rtimes \sfG $, the value $\scrX^\sfG _{\Tr}(f)(a)$
then depends only on the germs of $f_0\otimes \cdots \otimes f_k$
and $a_0\otimes \cdots \otimes a_k$ at $B^{(k)}$. In particular, if
$f$ vanishes around $B^{(k)}$, then $\scrX^\sfG _{\Tr}(f)=0$. This
shows that for each $k$, $\scrX^\sfG _{\Tr}$ is well defined as a
map from $\calAS \big(\sfG, \calC^\infty_{\sfG_0}((\hbar))\big)^k$
to $C^k \big( \calA^{((\hbar))}\rtimes \sfG \big)$. Moreover, one
checks immediately that
\[
  b \scrX^\sfG _{\Tr} = \scrX^\sfG _{\Tr} \delta \quad
  \text{and} \quad B \scrX^\sfG _{\Tr} (f) = 0,
  \text{ if $f \in \calAS \big(\sfG, \calC^\infty_{\sfG_0}((\hbar))\big)^k$}.
\]
We conclude that $\scrX^\sfG_{\Tr}$ defines a cochain map from the
the groupoid Alexander--Spanier cochain complex of $\sfG$ to the
cyclic cochain complex of $\calA^{((\hbar))}\rtimes \sfG$.

To relate our construction to higher indices of elliptic operators
on an orbifold $M$, we construct a cochain map $\scrX^M_{\Tr}$ from
groupoid Alexander--Spannier cochain complex of $\sfG$ to the cyclic cochain
complex of the algebra $\calA^{((\hbar))}_M$, which can be
identified as the algebra of $\sfG$-invariant smooth functions on
$\sfG_0$ equipped with a $\sfG$-invariant star product.

Let $c$ be a smooth cut-off function on $\sfG_0$ as is introduced in \cite[Sec. 1]{tu}. Define $e$ a smooth function on $\sfG$ by
\[
e(g):=c(s(g))^{\frac{1}{2}}c(t(g))^{\frac{1}{2}}.
\]
It is easy to check that $\sum_{g=g_1g_2}e(g_1)e(g_2)=e(g)$.
Let $E$ be the corresponding projection in $\calA^{((\hbar))}\rtimes \sfG$ with $E_{\hbar=0}=e$. (We point out that $e$ and $E$ may not be compactly supported but they can be chosen to be inside a proper completion of $\calA\rtimes \sfG$ and $\calA^{((\hbar))}\rtimes \sfG$ on which the convolution products are still well defined.) It is easy to check that $e$ commutes with all $\sfG$-invariant functions on $\sfG_0$ and similarly $E$ commutes with all elements of $\calA^{((\hbar))}_M$.

We will use $E$ to define a cochain map $\scrX^M_{\Tr}$ from groupoid Alexander-Spannier cochain complex of $\sfG$ to the cyclic cochain complex of $\calA^{((\hbar))}_M$.

Define $\scrX^{M}_{\Tr}: \calAS^\bullet \big(\sfG,
\calC^\infty_{\sfG_0}((\hbar))\big) \rightarrow C^\bullet\big(
\calA^{((\hbar))}_M\big)$ by
\begin{equation}\label{eq:orb-algebraic-pairing}
\begin{split}
\scrX^M_{\Tr}(f_0\otimes \cdots \otimes f_k)(a_0, \cdots,
a_k):=&\Tr\Big(\big(f_0\star (a_0\star_c E)\big)\star_c \cdots \star_c \big(f_k\star (a_k\star_c E)\big)\Big)\\
=&\Tr\Big(\big((f_0\star E)\star_c a_0 \big)\star_c \cdots \star_c \big((f_k\star E)\star_c a_k\big)\Big)
\end{split}
\end{equation}
where $a_0, \cdots, a_k$ are elements of $\calA^{((\hbar))}_M$
identified as $\sfG$-invariant functions on $\sfG_0$. We point out
that since $a_0, \cdots, a_k$ and $f_0, \cdots, f_k$ are compactly supported,
$\big((f_0\star E)\star_c a_0 \big)\star_c \cdots \star_c \big((f_k\star E)\star_c a_k\big)$ is also compactly
supported. Using the fact that $E$ commutes with $a_i$, we can quickly check the equality between the two expressions in the definition. Hence, the pairing
$\scrX^M_{\Tr}(f_0\otimes \cdots \otimes f_k)(\cdots)$ is well defined.
Similar to $\scrX^\sfG_{\Tr}$, one
can easily check that $\scrX^M_{\Tr}$ is compatible with the
differentials and therefore defines a cochain map.

Both $\scrX^\sfG_{\Tr}$ and $\scrX^M_{\Tr}$ are morphisms of sheaves
of complexes. Now we explain how to related $\scrX^\sfG_{\Tr}$ and
$\scrX^M_{\Tr}$ when $M$ is reduced. As shown in
\cite[Prop.~5.5]{nppt}, the algebra $\calA^{((\hbar))}\rtimes\sfG$
is Morita equivalent to the invariant algebra
$\big(\calA^{((\hbar))} (\sfG_0)\big)^\sfG$, if $M$ is reduced.
The Morita equivalence bimodules are given by
$P:=\calA^{((\hbar))}\rtimes\sfG\star_c E$ and $Q:=E\star_c
\calA^{((\hbar))}\rtimes\sfG$, where $P$ (and $Q$) is a left (right)
$\calA^{((\hbar))}\rtimes\sfG$ and right (left)
$\big(\calA^{((\hbar))} (\sfG_0)\big)^\sfG$ bimodule. In particular,
the map $\iota:\big(\calA^{((\hbar))} (\sfG_0)\big)^\sfG\to
\calA^{((\hbar))}\rtimes\sfG$ defined by $\iota(a)=E\star_c a\star_c E=a\star_c E$ is an
algebra homomorphism between the two algebras, and one can easily
check the following diagram to commute:
\[
\begin{diagram}
\node{\calAS^\bullet(\sfG, \calO)}\arrow{e,t}{\scrX^\sfG_{\Tr}}
\arrow{s,l}{Id}\node{C^\bullet(\calA^{((\hbar))}\rtimes
\sfG)}\arrow{s,r}{\iota}\\
\node{\calAS^\bullet(\sfG,
\calO)}\arrow{e,t}{\scrX^M_{\Tr}}\node{C^\bullet\big((\calA^{((\hbar))}
(\sfG_0))^\sfG\big)}
\end{diagram}.
\]

We point out that when $\sfG$ is a transformation groupoid of a
finite group $\Gamma$ acting on a symplectic manifold $X$, then one
can choose $e=E$ to be the element
\[
  e= \frac{1}{|\Gamma|} \, \sum_{\gamma\in \Gamma}\delta_\gamma,
\]
where $\delta_\gamma$ is the function on $\tilde U\times \Gamma$
such that $\delta_\gamma (x,\gamma) =1$ for every $x\in \tilde U$,
and which is $0$ otherwise. Note that the Morita equivalence between the
crossed product algebra $\calA^{((\hbar))}\rtimes \Gamma$ and the
invariant algebra $\calA^{((\hbar))}(X)^\Gamma\cong
\calA^{((\hbar))}_{M}$ with $M=X/\Gamma$ was proved by
{\sc Dolgushev} and {\sc Etingof} \cite{dolget}.

After the above discussion, we end this subsection with comparing the
constructions above with the quasi-isomorphism $Q$ from Section \ref{sec:cyclic-gpd}.
Since all the cochain maps involved are sheaf morphisms, the same local computations
as in the proof of Proposition \ref{prop:QISM-ASCyc} entail the following result.
\begin{proposition}
\label{prop:QISM-orbiASCyc}
  The sheaf morphisms $\normscrX^M_{\Tr} :
  \scrcCAS^\bullet \big( \calC^\infty_{\sfG_0} ((\hbar)) \big) \rightarrow
  \Tot^\bullet \calB \scrC^\bullet \big( \calA^{((\hbar))}_M\big)$
  and $\scrQ \circ \widetilde{\lambda}:
  \scrcCAS^\bullet \big( \calC^\infty_{\sfG_0}((\hbar)) \big) \rightarrow
  \Tot^\bullet \calB \normscrC^\bullet\big(\calA^{((\hbar))}_M\big)
  \hookrightarrow \Tot^\bullet \calB \scrC^\bullet\big(\calA^{((\hbar))}_M\big)$
  coincide in the derived category of sheaves on $M$. In particular,
  the morphism
  $\normscrX_{\Tr}: \scrcCAS^\bullet \big( \calC^\infty_{\sfG_0}((\hbar)) \big)
  \rightarrow \Tot^\bullet \calB \scrC^\bullet(\calA^{((\hbar))}_M)$
  is a quasi-isomorphism.
\end{proposition}
\subsection{Pairing with localized $K$-theory}
\label{pairing}
In this section we will define localized $K$-theory for orbifolds and its pairing with
the Alexander--Spanier cohomology defined in Section \ref{subsec:oASc}.
Let $Q$ be an orbifold modeled by a proper \'etale groupoid $\sfG$.
Pseudodifferential operators on orbifolds were introduced in \cite{gn1,gn2}
and \cite{bu} as operators on $C^\infty(Q)$ that in any local orbifold chart
can be lifted to invariant pseudodifferential operators on open subsets of
$\R^n$. Here we are interested in the algebra of smoothing operators that are
lifts of such smoothing operators on $Q$. However, the notion of invariance is
not straightforward, except for global quotient orbifolds.

First let us remark that $C^\infty(Q)$ embeds into $C^\infty(\sfG_0)$ as functions invariant
under $\sfG$ via pull-back along the projection $\pi:\sfG_0\to Q$. Consider
the algebra $\PDO^{-\infty}(\sfG_0)$ of smoothing operators on $\sfG_0$.
Let $\calU=\{U_i\}_{i\in I}$ be a $\sfG$-trivializing covering of $\sfG_0$ and
denote by $A_i$ the restriction of $A\in\Psi^{-\infty}(\sfG_0)$  to $U_i\in\calU$.
Define
\[
\begin{split}
 \PDO^{-\infty}_{\text{\tiny inv}}&\, (\sfG,\calU):=\\
   &\hspace{-12mm}:=  \{A\in\Psi^{-\infty}(\sfG_0) \mid \text{supp}(A)\subset\calU^2,
  \: A_i(gx,gy)=A_j(x,y)\: \text{for all $i,j\in I,g\in \sfG_{ij}$}\}.
\end{split}
\]
Note that this definition really makes sense, since $\sfG$ is \'etale, hence any
arrow $g\in\sfG_1$ induces, by the existence of a local bisection,  a local diffeomorphism
with support on a sufficiently small neighbourhood of $s(g)\in\sfG_0$. Therefore, we find:
\begin{proposition}
For a sufficiently fine covering $\calU$ of $\sfG_0$, any element
$A\in \PDO^{-\infty}_{\text{\rm inv}}(\sfG,\calU)$ defines a smoothing operator on
$\calC^\infty_\text{\rm\tiny cpt} (Q)$.
\end{proposition}
Observe that $\PDO^{-\infty}_{\mbox{\tiny inv}}(\sfG,\calU)$ is not a subalgebra of
$\PDO^{-\infty}(\sfG_0)$ because of both the support condition and the invariance condition.
However, we shall consider the space
$C^\bullet_{\lambda}(\PDO^{-\infty}_{\mbox{\tiny inv}}(\sfG,\calU))$
of cyclic cochains nonetheless. Let $\tr$ be the densely defined trace on
$\PDO^{-\infty}(\sfG_0)$ coming from the representation on $L^2(\sfG_0)$. Let $\ast$ be canonical commutative product on $C^\infty( \sfG)$ defined by $f_1\ast f_2(g):=f_1(g)f_2(g)$ for $f_1,f_2\in C^\infty(\sfG)$.  For
$f=f_0\otimes\ldots\otimes f_{2k}$ an element in $\calC^\infty_\text{\rm\tiny cpt}(B^{(2k)}_\calU)$ define, as before,
\[
\scrX^{\calU}_{\tr}(f)(A_0\otimes\ldots\otimes
A_{2k})=\tr_k\big((f_0\ast e)A_0,\ldots,(f_{2k}\ast e)A_{2k}\big),
\]
with $A_0,\ldots,A_{2k}\in \PDO^{-\infty}_{\mbox{\tiny
inv}}(\sfG,\calU_0)$, where $\calU_0$ is a G-trivializing cover such
that $\mbox{st}^{2k}(\calU_0)$ refines $\calU$, and $e$ is the
projection in $\calA\rtimes \sfG$ introduced in Section
\ref{subsec:RelOrbAScycCoh}.
\begin{proposition}
The following identities hold true:
\[
\begin{split}
\scrX^{\calU}_{\tr}(\delta( f))(A_0\otimes\ldots\otimes A_{2k})&=\scrX^{\calU}_{\tr}(f)(b(A_0\otimes\ldots\otimes A_{2k}))\\
\scrX^{\calU}_{\tr}(t(f))(A_0\otimes\ldots\otimes A_{2k})&=\scrX^{\calU}_{\tr}(f)(A_{2k}\otimes A_0\otimes\ldots\otimes A_{2k-1}).
\end{split}
\]
\end{proposition}
\begin{proof}
This is a direct computation: first observe that for $f\in \calC^\infty_\text{\rm\tiny cpt} (B^{(2k)}_\calU)$ and smoothing operators $A_0,\ldots,A_{2k}\in \PDO^{-\infty}_{\mbox{\tiny inv}}(\sfG,\calU_0)$, the pairing $\scrX^{\calU}_{\tr}(f)(A_0\otimes\ldots\otimes A_{2k})$ can be written as
\[
\begin{split}
&\sum_{t(g_i)=x_i,\atop i=0,\ldots,2k}\int_{G_0^{2k+1}}
f(g_0,\ldots,g_{2k})e(g_0)\cdots e(g_{2k})\\
&\qquad\qquad \qquad\qquad A_0(g_0(x_0),x_1)\cdots A_{2k}(g_{2k}(x_{2k}),x_0)dx_0\cdots dx_{2k} \\
&\: =\!\!\sum_{s(g_0)=x_i\atop i=0,\ldots,2k}\int_{G_0^{2k+1}}f(g_0,\ldots,g_{2k})e(g_0)\cdots e(g_{2k})\\
&\qquad\qquad\qquad\qquad A_0(g_0\odot \ldots \odot g_{2k}(x_0),x_1)\cdots A_{2k}(x_{2k},x_0)dx_0\cdots dx_{2k},
\end{split}
\]
where, to pass to the second line, we use invariance of the kernels
$A_i,~i=0,\ldots,2k$, and the composition $g_0\odot \ldots \odot
g_{2k}$ is as defined in Lemma \ref{cyclic-nghbhd}. From this
expression and the property that $e$ is a projection, the identities
of the Proposition easily follow.
\end{proof}
We can therefore morally think of $\scrX^{\calU}_{\tr}$ as being a morphism of
cochain complexes from  the compactly supported Alexander--Spanier complex
$(C^\bullet_{\rm AS,c}(\sfG,\calU),\delta)$ to the cyclic complex $(C^\bullet_\lambda(\PDO^{-\infty}_{\mbox{\tiny inv}}(\sfG,\calU_0)),b)$.

\subsubsection{Localized $K$-theory}
After these preparations, we can give a definition of localized $K$-theory for orbifolds.
Let $\calU$ be a $\sfG$-trivializing covering of $\sfG_0$ and consider the associated
subset $\PDO^{-\infty}_{\text{inv}}(\sfG,\calU)$ of smoothing operators. As before, unitalization
is denoted by a ${\mbox{}}^\sim$. With this, let us define
\begin{equation}
\label{Eq:GammaDefCovKTheory}
\begin{split}
   K_0 &\left(\PDO^{-\infty}_{\text{\tiny inv}}(\sfG,\calU)\right)\\
   &\quad := \big\{ (P,e) \in M_\infty \left( \Psi DO_\text{\tiny inv}^{-\infty}(\sfG, \calU )^\sim \right)
   \times M_\infty (\C) \mid
   P^2 =P , \: P^*=P,\: \\
   &\qquad \qquad e^2 =e , \  e^* =e  \ \text{ and } \:
   P-e \in M_\infty \left( \Psi DO_\text{\tiny inv}^{-\infty}(\sfG,\calU)\right)\big\} / \sim ,
\end{split}
\end{equation}
where $(P,e) \sim (P', e')$ for projections $P,P' \in M_\infty \big(
\PDO^{-\infty}_{\text{inv}}(\sfG,\calU)^{\sim}\big)$ and $e,e' \in
M_\infty (\C )$, if the elements $P$ and $P'$ can be joined by a
continuous and piecewise $\calC^1$ path of projections in some $M_N
\left( \PDO^{-\infty}_{\text{inv}}(\sfG,\calU)\right)$ with $N\gg 0 $ and
likewise for $e$ and $e'$. Elements of $ K_0 \left(\PDO^{-\infty}_{\text{inv}}(\sfG,\calU)\right)$ are
represented as equivalence classes of differences $R:= P-e$, where
$P$ is an idempotent in $M_\infty \big( \PDO^{-\infty}_{\text{inv}}(\sfG,\calU)^{\sim}\big)$, $e$ is a projection in
$M_\infty (\C)$, and the difference $P-e$ lies in $M_\infty \left(
\PDO^{-\infty}_{\text{inv}}(\sfG,\calU) \right)$.

A (finite) refinement $\calV\subset\calU$ obviously leads to an
inclusion $\PDO^{-\infty}_{\text{inv}}(\sfG,\calV) \hookrightarrow \PDO^{-\infty}_{\text{inv}}(\sfG,\calU)$ which induces a map
$K_0 \left(\PDO^{-\infty}_{\text{inv}}(\sfG,\calV)\right)\to K_0 \left(\PDO^{-\infty}_{\text{inv}}(\sfG,\calU)\right)$. With these maps,
the orbifold localized $K$-theory of $Q$ is defined as
\begin{equation}
\label{Eq:GammaLocKTheory}
  K^0_{\rm loc}(Q):=\lim_{\longleftarrow\atop \calU \in \operatorname{Cov}_{\sfG}(M)}
  K_0\left(\PDO^{-\infty}_{\text{inv}}(\sfG,\calU)\right).
\end{equation}
More precisely, this means that elements of $K^0_{\rm loc}(Q)$
are given by families
\begin{equation}
\label{Eq:GammaLocKTheoryRep}
  \big( [P_\calU - e_\calU]\big)_{\calU \in
  \operatorname{Cov}_{\sfG}(M)}
\end{equation}
of equivalence classes of pairs of projectors in matrix spaces over
$\PDO^{-\infty}_{\text{inv}}(\sfG,\calU)^{\sim}$ such that $e_\calU \in
M_\infty (\C )$ for every $\sfG$-trivializing covering $\calU$ and
$(P_\calU,e_\calU)\sim (P_\calV,e_\calV)$ in $M_\infty
\left(\PDO^{-\infty}_{\text{inv}}(\sfG,\calU)^{\sim} \right)$ whenever
$\calV\subset\calU$.

\subsubsection{Pairing with Alexander--Spanier cohomology}
Finally, let us describe the pairing of the thus defined localized
$K$-theory with orbifold Alexander--Spanier cohomology. Let $\calU$
be a $\sfG$-trivializing covering of $\sfG_0$, and
$f=f_0\otimes\ldots\otimes f_{2k}\in \calC^\infty_\text{\rm\tiny
cpt}(B^{(2k)}_\calU)$ a cocycle, i.e, $\delta f=0$. Choose a
$\sfG$-trivializing covering $\calU_0$ of $\sfG_0$ such that
$\operatorname{st}^{2k} (\calU_0)$ refines $\calU$. Now let
$R_{\calU_0} := P_{\calU_0}- Q_{\calU_0} \in
\PDO^{-\infty}_{\text{inv}}(\sfG,\calU_0)$ represent an element of
$K_0\left(\PDO^{-\infty}_{\text{inv}}(\sfG,\calU_0)\right)$ as
defined above. Define
\begin{equation}
\label{Eq:DefOrbASChern}
\begin{split}
  \big( \Ch^\text{\rm\ AS}_{2k} (R_{\calU_0}) \big) \,
  (f)  := \,  (-2\pi i)^k \frac{(2k)!}{k!}\varepsilon^{2k}
  &\Big(\tr \big((f_0\ast e) P_{\calU_0} (f_1\ast e) \ldots (f_{2k}\ast e)P_{\calU_0} \big)\\
 &-  \tr\big( (f_0\ast e) Q_{\calU_0} (f_1\ast e) \ldots (f_{2k}\ast e)Q_{\calU_0} \big)
\Big),
\end{split}
\end{equation}
where $\tr$ is the canonical operator trace on
$\PDO^{-\infty}(\sfG_0)$ and $e$ is the projection introduced in
Section \ref{subsec:RelOrbAScycCoh}. We remark that the $(f_0\ast e)
P_{\calU_0} (f_1\ast e) \ldots (f_{2k}\ast e)P_{\calU_0}$ and
$(f_0\ast e)Q_{\calU_0} (f_1\ast e) \ldots (f_{2k}\ast
e)Q_{\calU_0}$ are well defined trace class operators on
$L^2(\sfG_0)$, because $\operatorname{st}^{2k} (\calU_0)$ is finer
than $\calU$.

The same arguments as in \cite[Sec. 2]{moswu} now prove the following result.
\begin{proposition}\label{prop:orbpairing} In the limit when the covering gets finer, the pairing defined by Eq. \eqref{Eq:DefOrbASChern} is independent of all choices and induces a map
\[
H^{ev}_\text{\rm{\tiny cpt}} (\tilde{Q},\C)\times K^0_{loc}(Q)\to\C.
\]
\end{proposition}

\subsection{Operator-Symbol calculus on orbifolds and the higher analytic index}
In this final subsection we will define the higher analytic index of an elliptic differential operator on a
reduced orbifold and, using the algebraic index theorem, derive a topological expression computing
this number. Throughout this section, we denote by $Q$ a reduced compact riemannian orbifold modeled by a
proper \'etale groupoid $\sfG$. The groupoid $T^*\sfG$ therefore models the cotangent bundle $T^*Q$.
\subsubsection{Orbifold pseudodifferential operators and the symbol calculus}
\label{sssec:orbpdo}
Here we recall the symbol calculus on proper \'etale groupoids of \cite{ppt} and relate it
to the theory of pseudodifferential operators on orbifolds by imposing invariance. As for the
smoothing operators in Section \ref{pairing}, invariance only makes sense when the operators are localized to a sufficiently small neighbourhood of the diagonal in $\sfG_0\times\sfG_0$.
Let $\calU$ be a $\sfG$-trivializing cover of $\sfG_0$, and choose a cut-off function
$\chi:\sfG_0\times\sfG_0\to [0,1]$ as in \eqref{cut-off} with $\mbox{supp}(\chi)\subset\calU^2$
which is invariant:
\[
\chi(gx,gy)=\chi(x,y),\quad \mbox{for all}~g\in\sfG_{ij},~x,y\in U_i\times U_i.
\]
These choices define a quantization map as in \eqref{op-quant}. Observe that
the groupoid $\sfG$ acts on the sheaf $ \Sym^m$ of symbols on $\sfG_0$, since
they are just functions on $T^*\sfG_0$. It therefore makes sense to consider
the subspace $\Sym^m_{\rm inv}$ of invariant global symbols of order $m$.
With this, we see from the explicit formula \eqref{form_op-quant}
that the quantization provides a map
\[
\Op :\:  \Sym^m_{\rm inv} \rightarrow \PDO^m_{\rm inv} (\sfG,\calU),
\]
where, as for the smoothing operators,
\[
\begin{split}
\PDO^m_{\rm inv} &\, (\sfG,\calU):= \\
  := & \, \{A\in\PDO^m(\sfG_0),\mid \supp(A)\subset\calU^2 , \: gA_ig^{-1}=A_j,
  \text{for all } i,j\in I,g\in \sfG_{ij}\}.
\end{split}
\]
Indeed, both the support and the invariance properties follow from the corresponding
properties of the cut-off function $\chi$. In the opposite direction, the symbol map
$\sigma$ defined in equation \eqref{def-symbol} maps $\sigma: \PDO^m_{\rm inv} (\sfG,\calU)
\to \Sym^m_{\rm inv}$ and we therefore have an isomorphism
\[
\Sym^\infty_{\rm inv}\slash\Sym^{-\infty}_{\rm inv}\cong
\PDO^\infty_{\rm inv} (\sfG,\calU)\slash\PDO^{-\infty}_{\rm inv}(\sfG,\calU),
\]
induced by $\Op$ and $\sigma$. Now, since pseudodifferential operators have smooth kernels
off the diagonal, one observes that the right hand side inherits an algebra structure from
the product in $\PDO^\infty(\sfG_0)$ even though $\PDO^\infty_{\rm inv} (\sfG,\calU)$ is not
closed under operator composition. Therefore, going over to asymptotic families of symbols,
one obtains a deformation quantization of $T^*Q$ by defining the product on invariant
asymptotic families of symbols as in equation \eqref{as-product}.
We refer to \cite[Appendix 2]{ppt} for more details about this operator-symbol calculus.
Moreover, the operator trace on $L^2(Q)$ defines a trace $\tr$
on this deformation quantization. We observed in \cite{ppt} that as in the
manifold case, this canonical deformation quantization
using the asymptotic symbol calculus is
isomorphic to one constructed by a Fedosov connection. Under the corresponding
isomorphism, the operator $\tr$ is identified with the trace $\Tr$ defined
by Eq.~\eqref{eq:defunitr}.
Besides this, we can use now a similar argument as in Sec. \ref{subsec:RelOrbAScycCoh} for the construction
of $\scrX_{\Tr}$, to show that asymptotically in $\hbar$ the locally defined
maps $\scrX_{\tr}(U)$ glue together to a sheaf morphism
$\normscrX_{\tr} : \scrCAS^\bullet \big(\calC^\infty_{T^*\sfG_0} ((\hbar))\big)
\rightarrow \scrC^\bullet\big(\calA^{((\hbar))}\big)$.
Furthermore, we can pull back functions on $Q$ to $T^*Q$,
hence we obtain a quasi-isomorphism from
$\scrCAS^\bullet\big(\calC^\infty_Q((\hbar))\big)$ to
$\scrCAS^\bullet\big(\calC^\infty_{T^*Q}((\hbar))\big)$.
Using the same arguments as for the proof of Eq.~\eqref{eq:symbol}, we can
show now that the induced cochain map
$\normscrX_{\tr}: \scrCAS^\bullet\big( \calC^\infty_{T^*Q}\big)\to
\scrC^\bullet\big( \calA^{((\hbar))}_{T^*Q}\big)$ agrees with the map
$\scrX_{\Tr}$.
\subsubsection{The orbifold higher analytic index}
Let $D$ be an elliptic differential operator on the reduced orbifold $Q$.
We denote by the same symbol $D$ its lift to a $\sfG$-invariant elliptic
operator on $\sfG_0$. With the symbol calculus developed in the previous
section we can now prove the following:
\begin{proposition}
The elliptic operator $D$ defines a canonical element $[D]\in K^0_{\rm loc}(Q)$.
\end{proposition}
\begin{proof}
By the definition of localized $K$-theory, cf.~\eqref{Eq:GammaLocKTheory},
we first have to construct an element in
$K_0\left(\PDO^{-\infty}_{\rm inv}(\sfG,\calU)\right)$ for any $\sfG$-trivializing
cover $\calU$, and second for any refinement $\calV\subset\calU$ a homotopy
between the corresponding K-theory elements localized in $\calV$ respectively $\calU$.
To achieve the first we use the operator-symbol calculus developed in Section
\ref{sssec:orbpdo} and follow the standard procedure (cf.~\cite[Sec.~3.2.2]{Eg-Sc})
to find a symbol function $e\in \Sym^\infty _{\rm inv} (Q)$  such that
$\sigma(D)e-1$ and $e\sigma(D)-1$ are in $\Sym^{-\infty}_{\rm inv}$.
Choose a $\sfG$-trivializing covering $\calU '$ such that ${\rm st}^2(\calU ')$ refines
$\calU$, and a corresponding invariant cut-off function $\chi$, we define the quantization
map as in \eqref{op-quant}. It follows that both $D\Op(e)-I$ and $\Op(e)D-I$ are elements in
$\PDO^{-\infty}_{\rm inv}(\sfG,\calU')$, since $D$ is an invariant differential operator.
Write $E=\Op(e)$, and define $S_0:=I-DE$, and $S_1:=I-ED$, and
\[
L=\left(\begin{array}{cc}S_0& -E-S_0B\\ D&S_1\end{array}\right).
\]
Then the matrix $R$ defined by
\[
R=L\left(\begin{array}{cc}I&0\\ 0&0\end{array}\right)L^{-1}-\left(\begin{array}{cc}0&0\\ 0&I\end{array}\right)
\]
is a formal difference of projectors in $M_2\big(\PDO^{-\infty}_{\rm inv}(\sfG,\calU)\big)$
which defines an element in $K_0\left(\PDO^{-\infty}_{\rm inv}(\sfG,\calU)\right)$. Second,
for a refinement $\calV\subset\calU$ we have two element
$R_\calU\in K_0\left(\PDO^{-\infty}_{\rm inv}(\sfG,\calU)\right) $ and
$R_\calV\in K_0\left(\PDO^{-\infty}_{\rm inv}(\sfG,\calV)\right)$ defined by using cut-off
functions $\chi_\calU$ and $\chi_\calV$. But then the family of projectors $R_t,~t\in [0,1]$
defined using the cut-off function $\chi_t=t\chi_\calU+(1-t)\chi_\calV$ gives the desired
homotopy proving that both projectors define the same element in
$K_0\left(\PDO^{-\infty}_{\rm inv}(\sfG,\calU)\right)$. In total, this defines the element
$[D]\in K^0_{loc}(Q)$. It is independent of any choices made.
\end{proof}
We are now ready to define the higher analytic index of $D$. Let
$[f]\in H^{2k}_\text{\rm\tiny cpt}(\tilde{Q},\C)$ be a compactly supported cohomology class
of degree $2k$, represented by an Ale\-xan\-der--Spanier cocycle
$f\in \calC^\infty_\text{\rm\tiny cpt} (B^{(2k)}_{\calU})$ satisfying $\delta(f)=0$ for some
$\sfG$-trivializing cover $\calU$. Choose a $\sfG$-trivializing cover $\calV$ such that
${\rm st}^{2k}(\calV)$ refines $\calU$.
Then by the above discussion, $R_\calV$ defines an element in
$K_0\left(\Psi^{-\infty}_{\text{inv}}(\sfG,\calV)\right)$, which can be paired with $f$.
Hence we define the $[f]$-localized index of $D$ to be
\[
 \ind_{[f]}=:\Ch_{2k}^{\text{AS}}(R_\calV)(f),
\]
which is independent of the choices of the representative $f$ in its cohomology class and
the coverings $\calU, \calV$.

Using the previously obtained results from this section
one proves exactly like for Eq.~(\ref{eq:symbol})
that by comparing Eq.~(\ref{eq:orb-algebraic-pairing}) and Eq.~(\ref{Eq:DefOrbASChern})
the higher analytic index of $D$ on $Q$ can be computed using
the corresponding higher algebra index of $r_D$, where $r_D$ is the
asymptotic symbol of $R_D$.
Therefore, we can apply Thm.~\ref{thm:orb-alg-index} to compute
$\ind_{[f]}(D)$. This proves our last result.
\begin{theorem}\label{thm:orb-ana-ind}
 Let $D$ be an elliptic pseudodifferential operators
 on a reduced orbifold $Q$, and $[f]$ a compactly supported orbifold cyclic
 Alexander--Spanier cohomology class of degree $2j$. Then
\[
  \ind_{[f]}(D)=\sum_{r=0}^j\int_{\widetilde{T^*Q}}
  \frac{1}{(2\pi \sqrt{-1})^{j-r} \, m}
  \frac{\widetilde{\lambda}^{2j-2r}(f)\wedge
  \hat{A}(\widetilde{T^*M}) \, \Ch_\theta(\sigmapr(D))}{\Ch_\theta(\lambda_{-1}N)},
\]
where $\ell$, $\Ch_\theta$, $\lambda_{-1}N$, and $m$ are as in
Theorem \ref{thm:orb-alg-index}.
\end{theorem}

We end this section with two remarks about the above Theorem \ref{thm:orb-ana-ind}.
\begin{enumerate}
\item When we take the Alexander--Spanier cohomology class $1\in
H^0(\tilde{Q})$, the localized index $\ind_{[1]}(D)$ is the
classical index of the elliptic operator $D$ on $Q$. Theorem
\ref{thm:orb-ana-ind} in this case reduces to the Kawasaki's index
theorem \cite{Kaw}, and our proof is identical to the one given in
\cite{ppt}.
\item In the case that $Q$ is a global quotient orbifold represented by a
transformation groupoid as in Example \ref{example}, if we take the
cocycle $f_{\left<\gamma\right>}$ introduced at the end of Section
\ref{constr-as}, the localized index
$\ind_{f_{\left<\gamma\right>}}(D)$ can be computed using a theorem
by Atiyah and Segal in \cite{AtiSe}.
\end{enumerate}

%%% Local Variables:
%%% mode: latex
%%% TeX-master: "HigherIndex"
%%% End:

%% file: AppLocCycHom.tex
\section{Cyclic cohomology}
\subsection{The cyclic bicomplex}
Here we briefly recall the definition of Connes' $(b,B)$-complex computing
cyclic cohomology. Let $A$ be a unital algebra over a field $\Bbbk$. The
Hochschild chain complex $\big( C_\bullet (A) , b\big)$ resp.~the
normalized Hochschild chain complex $\big( \normC_\bullet (A),b\big)$ is
given by
\[
 C_k (A):=A\otimes_\Bbbk A^{\otimes k}
 \quad \text{resp.} \quad \normC_k (A):=A\otimes_\Bbbk(A/\Bbbk)^{\otimes k}
\]
equipped with the differential $b : C_k (A) \rightarrow C_{k-1} (A)$,
\[
\begin{split}
  b ( a_0\otimes\ldots\otimes a_k ) &:=
  \sum_{i=0}^{k-1}(-1)^ia_0\otimes\ldots\otimes a_ia_{i+1}\otimes\ldots\otimes
  a_k +(-1)^ka_ka_0\otimes\ldots \otimes a_{k-1}.
\end{split}
\]
Note that $b$ passes down to $\normC_\bullet (A)$.
The homology of $\big( C_\bullet (A),b\big)$ is called the Hochschild homology
of $A$ and is denoted by $HH_\bullet(A)$. It naturally coincides with the
homology of the normalized Hochschild chain complex. Introduce the operator
$\normB : \normC_k(A)\rightarrow \normC_{k+1}(A)$ by the formula
\[
  \normB ( a_0\otimes \ldots \otimes a_k ) :=
  \sum_{i=0}^k(-1)^{ik}1\otimes a_i\otimes\ldots\otimes a_k\otimes a_0\otimes
  \ldots\otimes a_{i-1}.
\]
This defines a differential, i.e., $\normB^2=0$, and we have $[\normB,b]=0$,
so we can form the $(b,\normB)$-bicomplex
\[
\xymatrix{
  \ldots \ar[d]_{b} & \ldots\ar[d]_{b} &\ldots\ar[d]_{b}\\
  \normC_2(A)\ar[d]_{b}  & 
  \normC_1(A)\ar[d]_{b} \ar[l]^{\normB} &
  \normC_0(A) \ar[l]^{\normB} \\
  \normC_1(A) \ar[d]_{b} & \normC_0(A\ar[l]^{\normB})\\
  \normC_0(A)
  }
\]
The total complex associated to this (normalized) mixed complex
\[
  \normcalB_k(A)=\bigoplus_{i=0}^{[k/2]}\normC_{k-2i}(A),
\]
equipped with the differential $b+\normB$, is the fundamental complex
computing the cyclic homology $HC_\bullet(A)$.
The dual theory is obtained by taking the $\Hom_\Bbbk (-,\Bbbk)$ of this
complex with the induced differentials, also denoted $b$ and $B$. For example
the normalized Hochschild cochain complex is given by
$\normC^\bullet(A):=\Hom_\Bbbk (\normC_\bullet(A),\Bbbk)$ and this leads to
the normalized mixed cyclic cochain complex
$\big( \normcalB^\bullet(A) , b , \normB\big)$.  This is the mixed complex
that we will mainly use throughout this paper.
For further information on Hochschild and cyclic homology theory
and in particular for the definition of $B$ and $\calB^\bullet(A)$ in the
general, not normalized, case see \cite{loday}.

\begin{remark}
  Note that for $\calA$ a sheaf of algebras over a topological space $M$,
  the assignments $U \mapsto C_k \big( \Gamma (U,\calA )\big)$ and
  $U \mapsto C^k \big( \Gammacpt (U,\calA )\big)$,
  where $U$ runs through the open subsets of $M$ are presheaves on $M$.
\end{remark}

\begin{remark}
  Throughout this paper we consider only algebras resp.~sheaves of algebras
  which additionally carry a  bornology compatible  with the algebraic
  structure. It is understood that the Hochschild and cyclic (co)homologies
  considered have to be compatible with the bornology meaning that as tensor
  product functor we take the completed bornological tensor product and as
  $\Hom$-spaces we choose the space of bounded linear maps between two
  bornological linear spaces. See \cite{MeyACC,PflPosTanTse} for details
  on bornologies.
\end{remark}

\subsection{Localization}
\label{subsec:loc}
Let $\Bbbk$ denote one of the ground rings $\R$, $\R[[\hbar]]$ or
$\R((\hbar))$, and let $M$ be
a smooth manifold. Let $\calO_{M,\Bbbk}$, or just $\calO$ if no confusion can
arise, be the sheaf of smooth functions $\calC^\infty_M$, if $\Bbbk=\R$,
the sheaf $\calC^\infty_M[[\hbar]]$, if $\Bbbk=\R[[\hbar]]$, and finally the
sheaf $\calC^\infty_M ((\hbar ))$, if  $\Bbbk=\R((\hbar))$. Assume that $\calO$
carries an associative local product $\cdot$, which can be either given by
the standard pointwise product of smooth functions or by a formal deformation
thereof. Note that in each  case, $\calO$ carries the structure of a
sheaf of bornological algebras  and that
\[
  \calO_{M^k,\Bbbk} = \calO^{\hat{\boxtimes} k}_{M,\Bbbk},
\]
where $\hat{\boxtimes}$ denotes the completed bornological exterior tensor
product.

Now let $X\subset M$ be a (locally) closed subset. Then put for each open
$U\subset M$
\[
  \calJ_{X,M,\Bbbk} (U) := \{ F \in \calO (U) \mid (DF)_{|X\cap U} = 0
  \text{ for all differential operators $D$ on $M$} \}.
\]
Obviously, these spaces form the section spaces of an ideal sheaf
$\calJ_{X,M,\Bbbk}$ in $\calO$; we denote it briefly by $\calJ_X$ if no
confusion can arise. The pullback of the
quotient sheaf $\calO / \calJ_{X,M,\Bbbk}$ by the canonical embedding
$\iota : X \hookrightarrow M$ gives rise to a sheaf of Whitney fields on
$X$ (cf.~\cite{Mal:IDF,BraPfl}). The resulting sheaf
$\iota^* \big( \calO / \calJ_{X,M,\Bbbk} \big)$ will be denoted by
$\calE_{X,M,\Bbbk}$ or $\calE_X$ for short.

Next let $\Delta_k : M \rightarrow M^k$ be the diagonal embedding. The
constructions above then give rise to  sheaf complexes
$\scrC_\bullet (\calO)$ and $\scrC^\bullet (\calO)$ defined as  follows.
For $k \in \N$ and $U\subset M$ open put
\begin{align}
  \scrC_k (\calO ) (U) & :=
   \Gamma \big( \Delta_{k+1} (U), \calE_{\Delta_{k+1} (M), M^{k+1},\Bbbk} \big)
   \quad \text{and} \\
  \scrC^k (\calO ) (U) & :=
   \Hom \big( \Gammacpt
   \big( \Delta_{k+1} (U), \calE_{\Delta_{k+1} (M), M^{k+1},\Bbbk} \big) ,
   \Bbbk \big).
\end{align}
Clearly, the $\scrC_k (\calO ) (U)$ resp.~$\scrC^k (\calO ) (U)$ are the
sectional spaces of a fine sheaf on $M$. Since $b$ and $B$ map the ideal
$\calJ_{\Delta_{k+1},M^{k+1},\Bbbk} (U)$ to $\calJ_{\Delta_k,M^k,\Bbbk} (U)$
resp.~$\calJ_{\Delta_{k+2},M^{k+2},\Bbbk} (U)$, the differentials $b$ and $B$
descend to $\scrC_\bullet (\calO )$ and $\scrC^\bullet (\calO )$. Thus we
obtain mixed sheaf complexes $\big( \scrC_\bullet (\calO ) ,b ,B\big)$ and
$\big( \scrC^\bullet (\calO ) ,b ,B\big)$. Obviously, there are normalized
versions of these mixed sheaf complexes which we will also use in this article.
Finally, for each open $U\subset M$ we have natural maps
\begin{align}
  \rho_k : \, &  C_k \big( \Gamma ( U , \calO ) \big) \rightarrow
  \scrC_k (\calO ) (U), \\
  \nonumber &  a_0 \otimes \ldots \otimes a_k \mapsto
  a_0 \otimes \ldots \otimes a_k + \calJ_{\Delta_{k+1} (U) , U, \Bbbk}
  \quad \text{and} \\
   \rho^k : \, & \scrC^k (\calO ) (U)  \rightarrow C^k
  \big( \Gammacpt ( U, \calO (U) \big) , \\
  \nonumber & F \mapsto
  \Big( a_0 \otimes \ldots \otimes a_k \mapsto
  F \big( a_0 \otimes \ldots \otimes a_k + \calJ_{\Delta_{k+1} (U) , U, \Bbbk}
  \big) \Big) .
\end{align}
Clearly, these maps are even morphisms of presheaves preserving the mixed complex
structures.
\begin{theorem}
\label{thm:locqism}
  The morphisms of mixed sheaf complexes $\rho_\bullet$ and $\rho^\bullet$
  are quasi-isomorphisms.
\end{theorem}
\begin{proof}
  For $\Bbbk =\R$ and $\calO$ the sheaf of smooth functions on $M$ the claim
  has been proven in \cite{BraPfl}. For $\Bbbk =\R[[\hbar]]$ the claim
  follows by a spectral sequence argument. Note that $\calO$ is filtered
  by powers of $\hbar$ in that case which induces a filtration on
  $C_\bullet \big( \Gamma ( U , \calO )$ and $\scrC_k (\calO ) (U)$. Consider
  the associated spectral sequences. The corresponding $E_1$-terms are the sheaf
  complexes associated to the sheaf of smooth functions on $M$ for which we
  already know that they are quasi-isomorphic. But this entails that the
  limits of these spectral sequences
  $C_\bullet \big( \Gamma ( U , \calO )$ and $\scrC_k (\calO ) (U)$
  have to be quasi-isomorphic, too. Likewise one checks
  that the complexes $\scrC^k (\calO ) (U)$ and $C^\bullet \big( \Gamma ( U , \calO )$
  are quasi-isomorphic in that case.
  By localizing $\hbar$ in this situation the claim follows also for
  $\Bbbk =\R((\hbar))$.
\end{proof}

%%% Local Variables:
%%% mode: latex
%%% TeX-master: "HigherIndex"
%%% End: